\documentclass[11pt]{article}
\usepackage{mathrsfs}
\usepackage{amsmath,amssymb, amsthm, amsopn,
  amsfonts,epsfig,graphics,relsize,exscale}{}
\usepackage{hyperref}
\hypersetup{
    colorlinks=true,
    linkcolor=red,
    filecolor=magenta,      
    urlcolor=cyan,
}
\usepackage[all]{xy}
\usepackage{graphicx}
\usepackage[T1]{fontenc}
\usepackage[latin1]{inputenc}
\usepackage{color,tocvsec2}
\usepackage{fouriernc}

\usepackage{pgf,tikz,pgfplots}
\pgfplotsset{compat=1.15}
\usepackage{mathrsfs}
\usetikzlibrary{arrows}
\usetikzlibrary{decorations.pathreplacing}
\definecolor{rd}{rgb}{1,0.3,0.35}

\usepackage[a4paper,width=14.66cm,top=2.52cm,bottom=2.52cm]{geometry}
\usepackage{makeidx}
\makeindex

\newcommand{\field}[1]{\mathbb{#1}} \newcommand{\rz}{\field{R}}
\newcommand{\cz}{\field{C}} \newcommand{\nz}{\field{N}}
\newcommand{\zz}{\field{Z}}

\newcommand{\ez}{\field{E}}

\newcommand{\Imag}{{\mathrm{Im\,}}}

\DeclareMathOperator\Real{\mathrm{Re}\,}

\DeclareMathOperator*{\wlim}{w-lim}

\newcommand{\Ker}{{\mathrm{Ker\,}}}
\newcommand{\Ran}{{\mathrm{Ran\,}}}

\newcommand{\ccup}{\mathop{\cup}}

\newcommand{\Weyl}{\mathrm{Weyl}}
\newcommand{\Wick}{\mathrm{Wick}}
\newcommand{\AWick}{\mathrm{A-Wick}}
\newcommand{\rmalg}{\mathrm{alg}}
\newcommand{\rmloc}{\mathrm{loc}}
\newcommand{\rmsym}{\mathrm{sym}}
\newcommand{\rmin}{\mathrm{in}}
\newtheorem{theorem}{Theorem}[section]
\newtheorem{lemma}[theorem]{Lemma}

\newtheorem{proposition}[theorem]{Proposition}
\newtheorem{definition}[theorem]{Definition}
\newtheorem{remark}[theorem]{Remark}

\title{Waves in a random medium: Endpoint Strichartz estimates
and number estimates}
\author{
S.~Breteaux
\thanks{Universit{\'e} de Lorraine, CNRS, IECL, F-57000 Metz, France. sebastien.breteaux@univ-lorraine.fr}\\
  F.~Nier
\thanks{LAGA, Universit{\'e} de Paris XIII, 99 avenue
J.B.~Cl{\'e}ment, F-93430~Villetaneuse, France. nier@math.univ-paris13.fr}\\
}
\begin{document}
\maketitle
\begin{abstract}
  In this article we reconsider the problem of the propagation of waves in a random medium in a kinetic regime. The final aim of this program would be the understanding of the conditions which allow to derive a kinetic or radiative transfer equation. Although it is not reached for the moment, accurate and somehow surprising number estimates in the Fock space setting, which happen to be propagated by the dynamics on macroscopic time scales, are obtained. Keel and Tao endpoint Strichartz estimates play a crucial role after being combined with a Cauchy-Kowalevski type argument. Although the whole article is focussed on the simplest case of Schr{\"o}dinger waves in a gaussian random potential of which the translation into a QFT problem is straightforward, several intermediate results are written in a general setting in order to be applied to other similar problems.
\end{abstract}
\\
\section{Introduction}
\label{sec:intro}

The asymptotic analysis or random homogenization of wave propagation
in a random medium, in a kinetic or diffusive regime has motivated
several works in the recent decades. It is not our purpose here to
give an exhaustive list but we think essentially of two different approaches: the
one initiated by G.~Papanicolaou and coauthors (see e.g. \cite{FGPS,Pap,RPK}) with a rather complete
review by J.~Garnier in \cite{Gar} and the one proposed by L.~Erd{\"o}s, H.T.~Yau and later
with  M.~Salmhofer in
\cite{ErYa}\cite{EYS1}\cite{EYS2}\,. 
Those two approaches formulate their results in terms of a kinetic (or diffusive)
evolution equation for some weak limit of scaled Wigner functions.
The main difference between the two approaches can be summarized as
follows
: The first approach presented in \cite{Gar} modeled 
on the
problem of randomly layered media (see \cite{FGPS}) focusses on
 space-time wave
functions, by solving a space-time PDE (it can be a Schr{\"o}dinger or a wave
equation) with random coefficients but with a smooth and essentially
deterministic right-hand side. With very strong assumptions on the
right-hand side of the equation, essentially deterministic and smooth,
a kinetic equation is written for the distributional weak limit of the
Wigner function  associated with the space-time wave function.
The work of \cite{ErYa}\cite{EYS1}\cite{EYS2} is concerned
with Cauchy problems, at the quantum level for the Schr{\"o}dinger
equation and semiclassically at a classical level for a linear
Boltzmann equation in \cite{ErYa} or a heat equation in
\cite{EYS1}\cite{EYS2}.
The strategy of this second approach consists after writing a Dyson
expansion (the iteration of Duhamel's formula), in making an accurate
combinatorial analysis of Feynman diagrams which  label  all the
random interaction terms of the expanded Dyson series. This Dyson
expansion technique was actually already used for a similar problem by
H.~Spohn in \cite{Spo}.
The final step which gives the asymptotic
behaviour of the Wigner transform, essentially relies on the accurate
control and expression of the remaining terms of the series by using
stationary phase asymptotic expressions for the many oscillating
integrals. The results of this second approach always require strong
assumptions on the initial data at the initial time $t=0$ and prove
weak convergence results at the macroscopic time $t\neq 0$\,. 

The main difficulty in this problem is concerned with the control of
recollisions and especially the proof that the asymptotic evolution is
Markovian, or given by some semigroup associated to a kinetic of heat
equation,
 although the multiple scattering process of waves could
destroy this markovian aspect. Depending on the asymptotic regime, the effective
asymptotic evolution could be affected  by some memory or non local in
time effect. In the considered asymptotic problems, it must be checked
that those memory effects vanish asymptotically.
In the approach reviewed in \cite{Gar} which  is concerned with rather
general random fields, this is proved by estimating higher moments.
In the approach of \cite{ErYa} the combinatorial accurate analysis of
Feynman diagrams, is reminiscent of the accurate control of
recollision terms by G.~Gallavotti in \cite{Gal} for the classical
Lorentz gas problem (Wind tree model). Both approaches bring accurate
information about a difficult problem in slightly different frameworks
and with various range of applications.\\

However those results remain unsatifactory from the mathematical point
of view and for the following reason: The dynamics of (quantum) waves
is given by a semigroup (actually a unitary group when there is no
dissipation) and the asymptotic kinetic or diffusive limits are also
given by well defined (semi)-groups. In the Cauchy problem approach,
one does not yet understand the dynamically stable class of initial
data which makes the derivation of a classical kinetic or heat
equation possible. Actually the results of
\cite{ErYa}\cite{EYS1}\cite{EYS2} are themselves puzzling because with
very specific initial data at time $t=0$\,, they prove the asymptotic
expected behaviour at the macroscopic time $t\neq 0$\,. But this means
that the time evolved quantum state at the macroscopic time $t/2\neq
0$\,, enters in the class of admissible initial data for which the
asymptotic evolution can be proved for a nonzero time interval (at the
macroscopic scale). Such  initial data do not enter in the very
specfic class considered at time $t=0$\,. In the space time approach
reviewed in \cite{Gar} the strong assumptions on the right-hand side
compared with the weak convergence results of the wave function, have
been considered in a negative way. Actually what is called ``statistical
stability'' is shown to fail with rough data (see \cite{Bal}). But no
positive answer for a general class of random right-hand side seems to
emerge. Although the two approaches are about slightly different
problems, they seem related at least for some basic random processes
on which  we will focus in this article.\\
Our hope is that  such an analysis about the
propagation of random waves in a random medium should lead to results
relying on dynamically stable hypotheses. We are led in this direction
by the strategy
followed by the second author with Z.~Ammari in
\cite{AmNi1}\cite{AmNi2} where they managed to give a general and
robust class of initial data, dynamically stable, such that the
quantum mean field dynamics can be followed.\\
About this very technical question a first attempt was tried by the
first author in \cite{Bre}. The idea was to exploit the link between
gaussian random fields (and possibly other fields like the poissonian
random fields) with quantum field theory. It rapidly appears that the
asymptotic problem, of waves in a random medium in a gaussian random
field in the kinetic regime,  cannot be thought as an infinite semiclassical
problem like the bosonic mean field problem. It has some similarities
but the strength of the free wave propagator and the translation
invariance lead to non quadratic and non ``semiclassical'' Wick
quantized operators. For this reason the coherent state method
presented in \cite{Bre} led to an accurate Ansatz, only for $O(h^{1/2})$
macroscopic times, where $h>0$ is the chosen small parameter, and the
derivation of a linear Boltzmann equation was possible only by forcing
the markovian nature of the asymptotic evolution by reinitializing on
some intermediate time scale the random potential. It was not at all
satisfactory. Actually the number estimates that we prove in this
article confirm that a coherent state approach cannot work for those
problems.\\

Another issue of this problem is the good understanding of the
dispersive properties of the free wave propagator with the asymptotic
behaviour of waves in a random medium. The different behaviours
expected in small dimension, $d\leq 2$ for the Schr{\"o}dinger equation in
the kinetic regime compared to $d\geq 3$\,, are closely linked with the time integrability
of the dispersion relation ($L^{1}-L^{\infty}$ estimates). In the
community of nonlinear PDE's, Strichartz estimates are known to be
more robust and effective than the pointwise in time
$L^{1}-L^{\infty}$ estimate. With the endpoint Strichartz estimates
proved by Keel and Tao in \cite{KeTa}, those inequalities are now
well adapted for linear critical problems. This article shows that
they actually lead to very accurate and somehow surprising ``number
estimates'' with some non trivial consequences.\\

Before giving the outline of this text, let us point out some
limitations and features of the present analysis: 
\begin{itemize}
\item We are not yet able to derive a full kinetic equation, except if
  one makes some connection with the existing results of
  \cite{ErYa}. The class of good initial data for which an asymptotic
  equation can be written is not yet identified.
\item We work essentially with the Schr{\"o}dinger equation in the
  presence of a gaussian random potential in the kinetic regime, as
 what we think to be the simplest, and richest model problem from the point of view
 of available structures.
\item Once the two previous points are made clear, the interested
  reader will realize that several argument, especially the one making
  use of Strichartz estimates, have been written in a sufficiently
  general framework in order to be transposed in another framework.
\item Some results like the possibility to define Wigner
  measures for all times, the localization in energy of the
  propagation phenomena, the class of potential corresponding to the
  scale invariant potential for Strichartz estimates, definitely bring
  a partial but accurate information. 
\end{itemize}

Our main results are about accurate number estimates, stated in
Proposition~\ref{pr:contracStri} in a rather general abstract setting
and in Theorem~\ref{th:applStri} for the case of our model problem of
the Schr{\"o}dinger equation with a gaussian translation invariant
potential in the kinetic regime and dimension $d\geq 3$\,.\\
\noindent 
\textbf{Outline of the article;}
\begin{description}
\item[a)] In Section~\ref{sec:randomFock} the link between gaussian Hilbert
  spaces and the bosonic Fock space is recalled and the equations in
  which we are interested are explicitely written.
\item[b)] In Section~\ref{sec:centermass} the translation invariance is used in
  order to make appear in a crucial way the center of mass variable,
  with respect to the \underline{position} of the field variable. The
  expression of the creation and annihiliation operators are given
  explicitely  in
  the center of mass and relative variables and finally
  $L^{p}$-estimates are carefully checked for those creation and
  annihilation operators under the suitable assumptions on the
  potential.
\item[c)] Section~\ref{sec:StriCM} reviews the known results about endpoint
  Strichartz estimates, and gives consequences in connection with the
  $L^{p}$-estimate in the center of mass given in
  Section~\ref{sec:centermass}. Then a rather general fixed point is proved
  which combines endpoint Strichartz estimates with an adaptation of 
Cauchy-Kowalevski techniques.
\item[d)] In Section~\ref{sec:conseqStri}, the general assumptions of Section~\ref{sec:StriCM} are checked
  in the framework of the Schr{\"o}dinger equation with a gaussian random
  field in the kinetic regime and ambient dimension $d\geq 3$\,.
\item[e)] Consequences and a priori information,  for the asymptotic evolution of Wigner
  functions are given in Section~\ref{sec:evolsemi}, withouth computing
  them.
\item[f)] Finally various approximation or stability results are
  deduced as consequences of the general estimates proved in Sections~\ref{sec:StriCM},~\ref{sec:conseqStri} and~\ref{sec:evolsemi}.
\end{description}

Before starting, be aware of the following assumed framework and
conventions: \\
 
All our Hilbert spaces, real or complex, are separable. 
All measures are assumed sigma-finite. On a set $\mathcal{X}$ endowed with a
sigma-set, a generic sigma-finite measure will be denoted
$\mathbf{dx}$\,, while the normal calligraphy $dx$ will be reserved
for the Lebesgue measure on $\mathcal{X}=\rz^{d}$\,. When $(\mathcal{X},\mathbf{dx})$ and
$(\mathcal{Y},\mathbf{dy})$ are two sigma-finite measured spaces, the
notation $L^{p}_{x}L^{q}_{y}$\,, $1\leq p,q\leq +\infty$\,, is used
for $L^{p}(\mathcal{X},\mathbf{dx};L^{q}(\mathcal{Y},\mathbf{dy}))$\,. However a more
general version of $L^{p}_{x}L^{q}_{y}$ will be introduced in Subsection~\ref{sec:genLpxLqy}.
\section{Random fields and Fock space}
\label{sec:randomFock}

\subsection{Gaussian Hilbert space and random fields}
\label{sec:gaussHilb}

Let $\mathcal{G}$ be the stochastic gaussian measure (see
e.g. \cite{Jan}) on the Lebesgue measured space 
$(\rz^{d},\mathcal{L}, dy)$\,. This defines a real Hilbert gaussian space
indexed by $L^{2}(\rz^{d},dy;\rz)$ which is generated, as a Hilbert space, by the centered
real gaussian variables $X_{A}\sim N(0,|A|)$\,, with
$A$ measurable set of $\rz^{d}$ and $|A|=\int_A dy$\,. By Minl{\"o}s
theorem (see \cite{Sim}) the space $L^{2}(\Omega,\mathcal{G};\rz)$
which contains powers of those gaussian processes can be realized with
$\Omega=\mathcal{S}'(\rz^{d},dy;\rz)$\,.\\
Complex valued elements $F\in L^{2}(\Omega,\mathcal{G};\cz)$ are
written $F=\Real F+i\Imag F$\,, $\Real F, \Imag F\in
L^{2}(\Omega,\mathcal{G};\rz)$ handled by the $\rz$-linearity of the
decomposition.\\
Once the complexification is fixed in this order (see \cite{Jan} for
an accurate description of various complex structures of gaussian measures),
the chaos decomposition of elements in
$F\in L^{2}(\Omega, \mathcal{G};\cz)$
can be written
\begin{equation}
  \label{eq:chaos}
F(\omega)=\displaystyle{\bigoplus_{n=0}^{\infty}}\int_{\rz^{dn}}F_{n}(y_{1},\ldots,y_{n}):X_{y_{1}}\cdots
X_{y_{n}}:~dy_{1}\cdots dy_{n}\,,
\end{equation}
where
\begin{itemize}
\item $F_{n}(y_{\sigma(1)},\ldots,y_{\sigma(n)})=F_{n}(y_{1},\ldots,y_{n})$ for
all $\sigma\in \mathfrak{S}_{n}$ and complex valued functions are
treated by the $\rz$-linearity of  the decomposition
$F_{n}=\Real(F_{n})+i\Imag{F_{n}}$\,;
\item the above symmetry can be written $F_{n}=S_{n}F_{n}$ where
  $S_{n}$ is the symmetrizing orthogonal projection on $
  L^{2}(\rz^{dn}, dy_{1}\cdots dy_{n};\cz)$ given by
  \begin{equation}
    \label{eq:symmproj}
    (S_{n}F_{n})(y_{1},\ldots,y_{n})=\frac{1}{n!}\sum_{\sigma\in \mathfrak{S}_{n}}F_{n}(y_{\sigma(1)},\ldots,y_{\sigma(n)})\,;
  \end{equation}
\item the family $(X_{y})_{y\in\rz^{d}}$ is made of
jointly gaussian real centered random fields such that
$\ez (X_{y}X_{y'})=\delta(y-y')$\,, which actually means
$$
\ez[(\int_{\rz^{d}}f(y)X_{y}~dy)(\int_{\rz^{d}}g(y')X_{y'}dy')]=\int_{\rz^{d}}f(y)g(y)~dy
$$
for all $f,g\in \mathcal{S}(\rz^{d};\cz)$\,;\footnote{We follow the general
probabilistic convention which omits the $\omega$ argument with
$X_{y}=X_{y}(\omega)$ e.g. in formula \eqref{eq:chaos}.}
\item products or
Wick products of singular random  variables $X_{y_{j}}$, $j=1\ldots J$\,, must be considered in their
weak formulation as well;
\item $:Y_{1}\cdots Y_{n}:$ stands for the Wick product of the random
  variables $Y_{1},\ldots,Y_{n}$\,;
\item with the assumed symmetry of the $F_{n}$ components\,,
\begin{equation}
    \label{eq:chaosnorm}
\ez(|F|^{2})=\int_{\Omega}|F(\omega)|^{2}~d\mathcal{G}(\omega)=\sum_{n=0}^{\infty}n!\int_{\rz^{dn}}|F_{n}(y_{1},\ldots,y_{n})|^{2}~dy_{1}\cdots dy_{n}=\sum_{n=0}^{\infty}n!\|F_{n}\|_{L^{2}}^{2}\,.
\end{equation}
\end{itemize}
A field is a random function of $x\in \rz^{d}$ and we shall consider
$F:\rz^{d}\times \Omega\to \cz$\,. A real gaussian centered translation invariant
field can be written
$$
\mathcal{V}(x,\omega)=\int_{\rz^{d}}V(y-x) \, X_{y}~dy\,.
$$
An element $F\in L^{2}(\rz^{d}_{x}\times \Omega,dx\otimes \mathcal{G};\cz)$ has the
chaos decomposition
\begin{align}
\label{eq:chaosx1}
  F(x,\omega)&=\displaystyle{\bigoplus_{n=0}^{\infty}}\int_{\rz^{dn}}\tilde{F}_{n}(x,y_{1},\ldots,y_{n}):X_{y_{1}}\cdots
X_{y_{n}}:~dy_{1}\cdots dy_{n}\\
\label{eq:chaosx2}
&=\displaystyle{\bigoplus_{n=0}^{\infty}}\int_{\rz^{dn}}F_{n}(x,y_{1}-x,\ldots,y_{n}-x):X_{y_{1}}\cdots
X_{y_{n}}:~dy_{1}\cdots dy_{n}
\end{align}
where
$F_{n}(x,y_{1},\ldots,y_{n})=\tilde{F}_{n}(x,y_{1}+x,\ldots,y_{n}+x)$
shares the same symmetry  in $(y_{1},\ldots,y_{n})$ as $\tilde{F}_{n}$
and
\begin{equation}
  \label{eq:chaosxnorm}
\|F\|_{L^{2}(\rz^{d}_{x}\times
  \Omega)}^{2}=\int_{\rz^{d}}\ez(|F(x,\cdot)|^{2})~dx
=\sum_{n=0}^{\infty}n!\|\tilde{F}_{n}\|_{L^{2}(\rz^{d}\times \rz^{dn})}^{2}=\sum_{n=0}^{\infty}n!\|F_{n}\|_{L^{2}(\rz^{d}\times \rz^{dn})}^{2}
\end{equation}
Assumptions on the real potential function $V$ will be specified later but we can already compute
the product $\mathcal{V}(x,\omega) \, F(x,\omega)$ by making use of Wick
formula (see e.g. \cite{Jan}-Theorem~3.15)
$$
X_{y}\, :X_{y_{1}}\cdots X_{y_{n}}{:} \, =\: :X_{y}X_{y_{1}}\cdots X_{y_{n}}{:}
+\sum_{j=1}^{n}\delta(y-y_{j})\, :X_{y_{1}}\cdots X_{y_{j-1}}
\underbrace{X_{y_{j}}}_{\mathrm{removed}} X_{y_{j+1}} \cdots X_{y_{n}}{:}
$$
which leads to the chaos decomposition of
$\mathcal{V}(x,\omega)F(x,\omega)$ as
\begin{multline}
\int_{\rz^{d(n+1)}}\frac{1}{(n+1)!}\sum_{\sigma\in \mathfrak{S}_{n+1}}V(y_{\sigma(n+1)}-x)\, F_{n}(x,y_{\sigma(1)}-x,\ldots,
y_{\sigma(n)}-x)~:X_{y_{1}}\cdots X_{y_{n+1}}{:}~dy_{1}\cdots
dy_{n+1}\,\\
\label{eq:WickV}
+\int_{\rz^{d(n-1)}}n\left[\int_{\rz^{d}}V(y)\, F_{n}(x,y,y_{1}-x,\ldots,y_{n-1}-x)~dy\right]~:X_{y_{1}}\cdots
X_{y_{n-1}}{:}~\,dy_{1}\cdots dy_{n-1}\,.
\end{multline}

\subsection{The Fock space presentation}
\label{sec:Focksp}
The chaos decomposition \eqref{eq:chaos} provides the isomorphism
between $L^{2}(\Omega, \mathcal{G};\cz)$ and the bosonic Fock space
$$
\Gamma(L^{2}(\rz^{d},dy;\cz))=\displaystyle{\bigoplus_{n=0}^{\infty}}(L^{2}(\rz^{d},dy;\cz))^{\odot
n}
$$ 
where for a (real or complex) Hilbert space $\mathfrak{h}$\,, $\mathfrak{h}^{\odot n}$ is the symmetric Hilbert completed
tensor product, equal to $\cz$ (or $\mathbb{R}$) for $n=0$\,,  endowed with the norm
such that
\begin{equation}
  \label{eq:normFock}
\|\varphi^{\otimes n}\|_{\mathfrak{h}^{\odot
    n}}=\|\varphi\|_{\mathfrak{h}}^{n}\quad,\quad
\|f_{n}\|_{L^{2}(\rz^{d},dy;\cz)^{\odot
    n}}=\|f_{n}\|_{L^{2}(\rz^{dn},dy_{1}\cdots dy_{n};\cz)}\,.
\end{equation}
The above direct sum is also the Hilbert completed direct sum. Note
that the Fock space norm \eqref{eq:normFock} differs from the
$\mathfrak{h}^{\odot n}$-norm chosen in \cite{Jan} in adequation with
Wick products by a factor $\sqrt{n!}$\,.
The unitary operator from $L^{2}(\Omega,\mathcal{G};\cz)$ to
$\Gamma(L^{2}(\rz^{d}_{y},dy;\cz))$ is thus given by
$$
F\mapsto \displaystyle{\bigoplus_{n=0}^{\infty}}f_{n}
\quad,\quad f_{n}=\sqrt{n!} \, F_{n}\,,
$$
since
$$
\|F\|_{L^{2}(\Omega,\mathcal{G};\cz)}^{2}=\sum_{n=0}^{\infty}n!
\|F_{n}\|_{L^{2}(\rz^{dn},
  dy_{1}\cdots dy_{n};\cz)}^{2}=\sum_{n=0}^{\infty}\|f_{n}\|^{2}_{L^{2}(\rz^{d},dy;\cz)^{\odot
  n}}\,.
$$
The Fock space $\Gamma(\mathfrak{h})$ is endowed with densely defined
Wick-quantized operators. For a monomial symbol
 $b(z)=\langle z^{\otimes q}\,,
\tilde{b}z^{\otimes p}\rangle$ with $\tilde{b}\in
\mathcal{L}(\mathfrak{h}^{\otimes p}; \mathfrak{h}^{\otimes q})$\,,
the Wick quantization $b^{\Wick}$ is defined on $\displaystyle{\bigoplus_{n\in
  \nz}^{\rmalg}} \, \mathfrak{h}^{\odot n}$ by
$$
b^{\Wick}f_{n+p}=\frac{\sqrt{(n+p)!(n+q)!}}{n!}S_{n+q}(\tilde{b}\otimes
\mathrm{Id}^{\otimes n})f_{n+p}
$$
where $S_{m}:\mathfrak{h}^{\otimes m}\to \mathfrak{h}^{\odot
m}$ is the symmetrizing orthogonal projection given by
\begin{equation}
  \label{eq:symmproj2}
S_{m}(g_{1}\otimes\cdots\otimes g_{m})=\frac{1}{m!}\sum_{\sigma\in
   \mathfrak{S}_{m}}g_{\sigma(1)}\otimes \cdots\otimes g_{\sigma(m)}
\end{equation}
already introduced in \eqref{eq:symmproj}.\\
Basic examples in our case $\mathfrak{h}=L^{2}(\rz^{d},dy;\cz)$ are given by
\begin{align*}
   a(g)&=(\langle g\,, z\rangle)^{\Wick}\,,\quad 
a(g)f_{n}(y_1,\dots,y_{n-1})=\sqrt{n}\int_{\rz^{d}}\overline{g(y)}f_n(y_{1},\ldots,
     y_{n-1},y)~dy\,,\\
  a^{*}(f) &=(\langle z\,,\, f\rangle)^{\Wick}\,,\quad 
a^{*}(f)f_{n}(y_1,\dots,y_{n+1})=\frac{\sqrt{n+1}}{(n+1)!}\sum_{\sigma\in \mathfrak{S}_{n+1}}f(y_{\sigma(1)})f_{n}(y_{\sigma(2)},\ldots,y_{\sigma(n+1)})\,,\\
 \phi(V)&=(\sqrt{2}\Real\langle V\,,\,z\rangle)^{\Wick}\,, \quad 
\phi(V)=\frac{1}{\sqrt{2}}[a(V)+a^{*}(V)]\,,\\
 d\Gamma(A)&=(\langle z,Az\rangle)^{\Wick}\,, \quad 
d\Gamma(A)=\sum_{k=0}^{n-1}\mathrm{Id}^{\otimes k}\otimes
   A\otimes \mathrm{Id}^{\otimes n-1-k}\,.
\end{align*}
with
\[[a(g),a^{*}(f)]=a(g)a^{*}(f)-a^{*}(f)a(g)=\langle g\,,\,  f\rangle\mathrm{Id}\,,\]
Remember also that  more generally, if $(A,D(A))$ generates a strongly continuous
semigroup of contractions $e^{tA}$, $t\geq 0$, then
$\Gamma(e^{tA})f_{n}=[e^{tA}]^{\otimes n}f_{n}$ defines a strongly
continuous semigroup of
contractions $\Gamma(e^{tA})$ on $\Gamma(\mathfrak{h})$ with generator
denoted by $(d\Gamma(A),D(d\Gamma(A)))$\,, which extends the above
definition of $d\Gamma(A)$\,. In particular this makes sense for
$A=-iB$ with $(B,D(B))$ self-adjoint on $\mathfrak{h}$ and $(d\Gamma(B),D(d\Gamma(B)))$
is a self-adjoint operator on $\Gamma(\mathfrak{h})$ when $(B,D(B))$
is self-adjoint on $\mathfrak{h}$\,.\\
According to \eqref{eq:chaosx2}\eqref{eq:chaosxnorm}, random
$L^{2}(\rz^{d},dx;\cz)$ functions $F(x,\omega)$
can be written as elements $f$ of $L^{2}(\rz^{d},dx;\cz)\otimes
\Gamma(L^{2}(\rz^{d},dy;\cz))$\,,
\begin{eqnarray*}
  &&
F(x,\omega)\mapsto
     f(x,\cdot -x)=\displaystyle{\bigoplus_{n\in\nz}}f_{n}(x,y_{1}-x,\ldots,y_{n}-x)\\
\text{with}&&f_{n}\in
L^{2}_{\mathrm{sym}}(\rz^{d}_{x}\times \rz^{dn},dxdy_{1}\cdots dy_{n};\cz)\,,\\
&&
\|F\|^{2}_{L^{2}(\rz^{d}\times \Omega,dx\otimes \mathcal{G})}
=\sum_{n=0}^{\infty}\|f_{n}\|^{2}_{L^{2}(\rz^{d}\times
              \rz^{dn},dxdy_{1}\cdots dy_{n})}\,,
\end{eqnarray*}
and where $L^{2}_{\mathrm{sym}}$ refers to the exchange symmetry in the $y$-variables.\\
When $V\in L^{2}(\rz^{d},dy;\rz)$ and
$\mathcal{V}(x,\omega)=\int_{\rz^{d}}V(y-x)\, X_{y}~dy$\,, the Wick product
formula \eqref{eq:WickV} for $\mathcal{V}(x,\omega)F(x,\omega)$ is
transformed into
\begin{equation}
  \label{eq:prodFock}
\mathcal{V}(x,\omega)F(x,\omega)\mapsto [a(V)+a^{*}(V)]f(x,\cdot -x)=[\sqrt{2}\phi(V)f](x,\cdot -x)\,.
\end{equation}
With the notation $D_{y}=\frac{1}{i}\partial_{y}=
\begin{pmatrix}
\frac{1}{i}\partial_{y^{1}}\\
\vdots\\
\frac{1}{i}\partial_{y^{d}}  
\end{pmatrix}
$ the operator $(x\cdot D_{y},D(x\cdot D_{y}))$\,, with
$x\cdot D_{y}=\sum_{k=1}^{d}x^{k}D_{y^{k}}$\,, is
essentially self-adjoint on $\mathcal{S}(\rz^{d},dy;\cz)$ for all
$x\in\rz^{d}$\,. This defines a strongly continuous unitary
representation of the additive group $(\rz^{d}_{x},+)$ on $L^{2}(\rz^{d}_{x})\otimes
\Gamma(L^{2}(\rz^{d}_{y}))$ given by
$$
e^{-ix\cdot d\Gamma(D_{y})}(\displaystyle{\bigoplus_{n=0}^{\infty}}f_{n}(x,y_{1},\ldots,y_{n}))=\displaystyle{\bigoplus_{n=0}^{\infty}}f_{n}(x,y_{1}-x,\ldots,y_{n}-x)\,.
$$
Therefore the above unitary correspondence $F(x,\omega)\mapsto
f(x,\cdot -x)$  gives a unitary correspondence
\begin{equation}
  \label{eq:unitFock}
F\in L^{2}(\rz^{d}\times \Omega, dx\otimes \mathcal{G};\cz)\mapsto f\in
L^{2}(\rz^{d},dx;\cz)\otimes \Gamma(L^{2}(\rz^{d},dy;\cz))\,,
\end{equation}
while \eqref{eq:prodFock} becomes for $V\in L^{2}(\rz^{d},dy;\rz)$
\begin{equation}
  \label{eq:prodFock2}
\mathcal{V}F\mapsto \left[\sqrt{2}\Phi(V)f\right]\,.
\end{equation}
We now translate a general pseudo-differential operator in the $x$-variable\,,
$a^{\Weyl}(x,D_{x})\otimes
\mathrm{Id}_{L^{2}(\Omega,\mathcal{G};\cz)}$ under the above
transformation \eqref{eq:unitFock}.\\
When $\mathfrak{h}$ is a complex Hilbert space, we recall that
$L^{2}(\rz^{d},dx;\cz)\otimes \mathfrak{h}$ equals
$L^{2}(\rz^{d},dx;\mathfrak{h})$ and
\begin{itemize}
\item
the
  Fourier transform, with the normalization
$$
Fu(\xi)=\int_{\rz^{d}}e^{-i\xi \cdot x}u(x)~dx\,,\quad F^{-1}v(x)=\int_{\rz^{d}}e^{ix \cdot \xi}v(\xi)~\frac{d\xi}{(2\pi)^{d}}\,,
$$
is
unitary from $L^{2}(\rz^{d},dx;\mathfrak{h})$ to
$L^{2}(\rz^{d},\frac{d\xi}{(2\pi)^{d}};\mathfrak{h})$\,;
\item $\mathcal{S}(\rz^{d};\mathfrak{h})$\,,
$\mathcal{S}'(\rz^{d};\mathfrak{h})$ and the Fourier transform have
the same properties as in the scalar case $\mathfrak{h}=\cz$\,.
\end{itemize}
Be aware that the behavior of the Fourier transform when
$\mathfrak{h}$ is a general Banach space is more tricky
according to \cite{Pee}. So when $\mathfrak{h}$ is a Hilbert space, we consider
pseudo-differential operators in the $x$-variable of the form
$a^{\Weyl}(x,D_{x})=a^{\Weyl}(x,D_{x})\otimes
\mathrm{Id}_{\mathfrak{h}}$ for a symbol $a\in
\mathcal{S}'(\rz_{x,\xi}^{2d};\cz)$ given by its Schwartz' kernel
$$
[a^{\Weyl}(x,D_{x})](x,y)=\int_{\rz^{d}}e^{i(x-y)\cdot \xi}a\left(\frac{x+y}{2},\xi\right)~\frac{d\xi}{(2\pi)^{d}}\,.
$$
When $\mathfrak{h}=\cz$\,, $a^{\mathrm{Weyl}}(x,D_{x})$ is a
continuous endomorphism of $\mathcal{S}(\rz^{d}_{x};\cz)$ and
$\mathcal{S}'(\rz^{d}; \cz)$ with the formal adjoint
$\overline{a}^{\mathrm{Weyl}}(x,D_{x})$ and the alternative
representations:
\begin{itemize}
\item When $v,u\in \mathcal{S}(\rz^{d};\cz)$\,,
$$
\langle v\,,\,
a^{\Weyl}(x,D_{x})u\rangle=\int_{\rz^{2d}}a(x,\xi)\, W[v,u](x,\xi)~\frac{dx\, d\xi}{(2\pi)^{d}}
$$
where $W[v,u]$ is the Wigner function  of the pair $[v,u]$ (or the
Weyl symbol of $|u\rangle\langle v|$), given by
$$
W[v,u](x,\xi)=\int_{\rz^{d}}e^{i\xi\cdot s}\,u(x+\frac{s}{2})\,\,\overline{v}(x-\frac{s}{2})~ds\,,
$$
and which belongs to $\mathcal{S}(\rz^{2d};\cz)$\,.
\item By setting
$\llbracket P,X\rrbracket=p_{\xi}\cdot x-p_{x}\cdot \xi$ for
  $P=(p_{x},p_{\xi})$\,, $X=(x,\xi)$ in $\rz^{2d}=T^{*}\rz^{d}$\,, and 
$$
\mathcal{F}a(P)=\int_{\rz^{2d}}e^{i\llbracket P,X\rrbracket}a(X)~\frac{dX}{(2\pi)^{d}}
$$
we have $a=\mathcal{F}(\mathcal{F}a)$ in
$\mathcal{S}'(\rz^{2d})$\,. When $\mathcal{F}a\in
L^{1}(\rz^{2d};\cz)$\,,
$$
a^{\Weyl}(x,D_{x})=\int_{\rz^{2d}}\mathcal{F}a(P)\,
\,\tau_{P}~\frac{dP}{(2\pi)^{d}}\,,
$$
where
$\tau_{P}=e^{i(p_{\xi}\cdot x-p_{x}\cdot D_{x})}=[e^{i(p_{\xi}\cdot x-p_{x}\cdot \xi)}]^{\Weyl}(x,D_{x})$
is  the unitary phase translation 
$$
\tau_{P}\,u\,(x)=e^{ip_{\xi}\cdot(x-p_{x}/2)}u(x-p_{x})\,.
$$
In particular, the above integral is a
$\mathcal{L}(L^{2}(\rz^{d},dx;\cz))$-integral when $\mathcal{F}a\in
L^{1}(\rz^{2d},dP;\cz)$ and a fortiori when $a\in \mathcal{S}(\rz^{2d};\cz)$\,.
\end{itemize}
With those two remarks, for  a general 
$a\in \mathcal{S}'(\rz^{2d};\cz)$ the integral
$$
a^{\Weyl}(x,D_{x})
=\int_{\rz^{2d}}\mathcal{F}a(P) \,\underbrace{e^{i(p_{\xi}\cdot x-p_{x}\cdot D_{x})}}_{=\tau_{P}}~\frac{dP}{(2\pi)^{d}}
$$
can be interpreted as the weak limit
$$
a^{\Weyl}(x,D_{x})=\wlim_{n\to\infty} \int_{\rz^{2d}}\mathcal{F}a_{n}(P)\, e^{i(p_{\xi}\cdot x-p_{x}\cdot D_{x})}~\frac{dP}{(2\pi)^{d}}\,,
$$
where $a_{n}\in \mathcal{S}(\rz^{2d};\cz)$ is any approximation of
$a\in \mathcal{S}'(\rz^{2d};\cz)$\,.\\
While considering the $a^{\Weyl}(x,D_{x})\otimes
\mathrm{Id}_{\mathfrak{h}}$\,, the same construction makes sense after
noticing that for $u,v\in \mathcal{S}(\rz^{d};\mathfrak{h})$\,, the
Wigner transform $W[v,u]$ belongs to
$\mathcal{S}(\rz^{2d}_{x,\xi};\mathcal{L}^{1}(\mathfrak{h}))$\footnote{$\mathcal{L}^{p}(\mathfrak{h})$
denotes the Schatten space of compact operators for $1\leq p\leq+\infty$\,.}
and 
$$
\langle v\,,\,
a^{\Weyl}(x,D_{x})u\rangle =\mathrm{Tr}\left[[a^{\Weyl}(x,D_{x})\otimes
\mathrm{Id}_{\mathfrak{h}}] \, |u\rangle\langle v|\right]
=\int_{\rz^{2d}}a(x,\xi)\, \mathrm{Tr}[W[v,u]](x,\xi)~\frac{dxd\xi}{(2\pi)^{d}}\,.
$$ 
We apply this with $\mathfrak{h}=L^{2}(\Omega,\mathcal{G};\cz)$ and
$\mathfrak{h}=\Gamma(L^{2}(\rz^{d},dy;\cz))$:
We start from
$$
a^{\Weyl}(x,D_{x})=a^{\Weyl}(x,D_{x})\otimes
\mathrm{Id}_{L^{2}(\Omega,\mathcal{G};\cz)}=\wlim_{n\to\infty}\int_{\rz^{2d}}\mathcal{F}a_{n}(P)e^{i(p_{\xi}\cdot x-p_{x}\cdot D_{x})}~\frac{dP}{(2\pi)^{d}}\,,
$$ 
the correspondance
$$
a^{\Weyl}(x,D_{x})F\mapsto e^{ix\cdot d\Gamma(D_{y})}a^{\Weyl}(x,D_{x})e^{-ix\cdot d\Gamma(D_{y})}f\,,
$$
and
$$
e^{ix\cdot \lambda}e^{i(p_{\xi}\cdot x-p_{x}\cdot D_{x})}(e^{-ix\cdot \lambda}\times)=e^{i(p_{x}\cdot x-p_{x}\cdot (D_{x}-\lambda))}\quad\text{for~all}~\lambda\in
\rz^{d}
$$
which gives by the functional calculus, the equality of unitary operators
$$
e^{ix\cdot d\Gamma(D_{y})}e^{i(p_{\xi}\cdot x-p_{x}\cdot D_{x})}(e^{-ix\cdot d\Gamma(D_{y})})=e^{i(p_{x}\cdot x-p_{x}\cdot (D_{x}-d\Gamma(D_{y})))}\,.
$$
We deduce that for $a\in \mathcal{S}'(\rz^{2d};\cz)$\,,
$a^{\Weyl}(x,D_{x})F\in
\mathcal{S}'(\rz^{d}_{x};L^{2}(\Omega,\mathcal{G};\cz))$ is
transformed into
\begin{equation}
\label{eq:WeyldG}
  a^{\Weyl}(x,D_{x}) \, F\mapsto
  a^{\Weyl}(x,D_{x}-d\Gamma(D_{y})) \, f\in \mathcal{S}'(\rz^{d};\Gamma(L^{2}(\rz^{d},dy;\cz)))\,.
\end{equation}
with
$$
a^{\Weyl}(x,D_{x}-d\Gamma(D_{y}))=
\wlim_{n\to\infty} \int_{\rz^{2d}}\mathcal{F}a_{n}(P)e^{i(p_{\xi}\cdot x-p_{x}\cdot (D_{x}-d\Gamma(D_{y})))}~\frac{dP}{(2\pi)^{d}}\,.
$$
Let us continue by applying the Fourier transform in the $x$-variable
with
$$
F_{x}u(\xi)=\int_{\rz^{d}}e^{-i\xi\cdot x}u(x)~dx\quad,\quad
F_{x}^{-1}u(x)=\int_{\rz^{d}}e^{ix\cdot \xi}u(\xi)~\frac{d\xi}{(2\pi)^{d}}\,
$$
and set for $f\in \mathcal{S}'(\rz^{d}_{x};\Gamma(L^{2}(\rz^{d},dy;\cz)))$
$$
\hat{f}=F_{x}f\in \mathcal{S}'(\rz^{d};\Gamma(L^{2}(\rz^{d},dy;\cz)))\,.
$$
With 
$$
F_{x} \, a^{\Weyl}(x,D_{x}) \, F_{x}^{-1}=
a^{\mathrm{Weyl}}(-D_{\xi},\xi)
$$
where the functional calculus leads to $F_{x} \, a^{\Weyl}(x,D_{x}-d\Gamma(D_{y})) \, F_{x}^{-1}=
a^{\Weyl}(-D_{\xi},\xi-d\Gamma(D_{y}))$\,,
we obtain the unitary correspondence
\begin{eqnarray}
    \label{eq:unitFockFou1}
  &&F\in L^{2}(\rz^{d}\times\Omega,dx\otimes\mathcal{G};\cz)\mapsto
  \hat{f}=F_{x}f\in
  L^{2}(\rz^{d},\frac{d\xi}{(2\pi)^{d}};\Gamma(L^{2}(\rz^{d},dy;\cz)))\,,\\
\label{eq:unitFockFou2}
\text{with}&&
F(x,\omega)=\sum_{n=0}^{\infty}\int_{\rz^{dn}}\frac{1}{\sqrt{n!}}f_{n}(x,y_{1}-x,\ldots,y_{n}-x):X_{y_{1}}\cdots
              X_{y_{n}}:~dy_{1}\cdots dy_{n}\,,
\end{eqnarray}
and where \eqref{eq:prodFock2} and \eqref{eq:WeyldG} become 
\begin{align}
  \label{eq:prodFockFou}
\mathcal{V} \, F&\mapsto \sqrt{2}\phi(V) \, \hat{f}\,,\\
\label{eq:WeyldGFou}
\underbrace{a^{\Weyl}(x,D_{x}) \, F}_{\in
             \mathcal{S}'(\rz^{d}_{\xi};L^{2}(\Omega,\mathcal{G};\cz))}&\mapsto
             \underbrace{a^{\Weyl}(-D_{\xi},\xi-d\Gamma(D_{y})) \, \hat{f}}_{\in \mathcal{S}'(\rz^{d}_{\xi};\Gamma(L^{2}(\rz^{d},dy;\cz)))}\,.
\end{align}
From this point of view, the Fock space and functional analysis
presentation is simpler than sticking with the usual chaos
decomposition \eqref{eq:chaosx1} where Fourier transforms and
pseudo-differential operators do not seem to have simple probabilistic
interpretation.\\
\begin{remark}
\label{re:finFock}
  As a final remark, all the above constructions can be tensorized with
  an additional separable Hilbert space
  $\mathfrak{h}'=L^{2}(Z,\mathbf{dz};\cz)$\,. 
\end{remark}

\subsection{Our problem}
\label{sec:ourproblem}
We aim at studying the stochastic partial differential equation
\begin{equation}
  \label{eq:dynalea}
    \begin{cases}
      i\partial_{t}F=-\Delta_{x}F+\sqrt{h}\mathcal{V}F\,,\\
F(t=0)=F_{0}\,,
    \end{cases}
\end{equation}
where 
\begin{itemize}
\item $\mathcal{V}$ is the translation invariant gaussian random field
$$
\mathcal{V}(x,\omega)=\int_{\rz^{d}}V(y-x) \, X_{y}~dy\,,
$$
with $V\in L^{2}(\rz^{d};\rz)$\,;
\item the solution $F(t,x,\omega,z)$ is seeked in
  $\mathcal{C}^{0}(\rz; L^{2}(\rz^{d}\times \Omega\times Z,dx\otimes
  \mathcal{G}\otimes \mathbf{dz};\cz))$\,;
\item $h>0$ is a small parameter which will tend to $0$\,.
\end{itemize}
In particular we will consider the asymptotic behavior of
 quantities
 \begin{equation}
   \label{eq:quantity}
\langle F(\frac{t}{h})\,,\, a^{\Weyl}(hx,D_{x})
F(\frac{t}{h})\rangle_{L^{2}(\rz^{d}\times \Omega\times Z)}=
\int_{Z}\ez\left[\langle F(\frac{t}{h},z)\,,\, a^{\Weyl}(hx,D_{x})F(\frac{t}{h},z)\rangle_{L^{2}(\rz^{d},dx)}\right]~\mathbf{dz}(z)
\end{equation}
for $a\in S(1,dx^{2}+d\xi^{2})$ and $t\in
[0,T]$\,. Remember that the symbol class $S(1,dx^{2}+d\xi^{2})$ is the
set of $\mathcal{C}^{\infty}$-functions on $\rz^{2d}$ with all
derivatives bounded on $\rz^{2d}$\,.\\
Note that the variable $z\in Z$ does not appear in the equation. The
dynamics is thus well defined when it is defined for
$Z=\left\{z_{0}\right\}$  and $\mathbf{dz}=\delta_{z_{0}}$\,.
A sufficient condition was provided in \cite{Bre} by making use of
Nelson commutator method.
\begin{lemma}
\label{le:dyn}\textbf{Proposition~4.4 in \cite{Bre}:}
Assume $V\in H^{2}(\rz^{d};\rz)$ then the operator
$-\Delta_{x}+\sqrt{h}\mathcal{V}$ is essentially self-adjoint on
$\displaystyle{\bigoplus_{n\in\nz}^{\rmalg}}\mathcal{S}(\rz^{d}_{x};(L^{2}(\rz^{d},dy;\cz))^{\odot
  n})$ which is a dense subset of $
L^{2}(\rz^{d}_{x},dx;L^{2}(\Omega,\mathcal{G};\cz))=L^{2}(\rz^{d}\times\Omega,dx\otimes
\mathcal{G};\cz)$ by \eqref{eq:chaosx1}.
\end{lemma}
\begin{remark}
   A side corollary of our analysis says that the dynamics is well
   defined under the assumption $V\in
   L^{r^{\prime}_{\sigma}}(\rz^{d};\rz)$ with
   $r^{\prime}_{\sigma}=\frac{2d}{d+2}$ in dimension $d\geq 3$\,, See
   Subsection~\ref{sec:quantlow} at the end of the article.
 \end{remark}
Lemma~\ref{le:dyn} provides a  natural self-adjoint realization of
$-\Delta_{x}+\sqrt{h}\mathcal{V}$ in $\mathfrak{h}=L^{2}(\rz^{d}\times\Omega\times
Z,dx\otimes \mathcal{G}\otimes \mathbf{dz};\cz)$ and any initial datum
$F_{0}\in \mathfrak{h}$
defines a unique solution $F\in
\mathcal{C}^{0}(\rz;\mathfrak{h})$\,.\\
There are various reasons for introducing an additional variable $z\in
Z$\,, and this trick will be used repeatedly. One of them is the
following: Starting with $Z=\left\{z_{0}\right\}$ and
$\mathbf{dz}=\delta_{z_{0}}$\,, one may consider instead of
$F(\frac{t}{h})=U_{\mathcal{V}}(\frac{t}{h})F_{0}$ with
$U_{\mathcal{V}}(t)=e^{-it(-\Delta_{x}+\sqrt{h}\mathcal{V})}$\,, the evolution of a
state 
$$
\varrho(\frac{t}{h})=U_{\mathcal{V}}(\frac{t}{h})\varrho_{0}U_{\mathcal{V}}^{*}(\frac{t}{h})
$$
with $\varrho_{0}\in \mathcal{L}^{1}(L^{2}(\rz^{d}\times
\Omega;\cz))$\,, $\varrho_{0}\geq 0$\,, $\mathrm{Tr}[\varrho_{0}]=1$
possibly replacing $\|F_{0}\|_{L^{2}}=1$\,. By writing
$\varrho_{0}=\varrho_{0}^{1/2}\varrho_{0}^{1/2}$ one gets
$$
\varrho(\frac{t}{h})=[U_{\mathcal{V}}(\frac{t}{h})\varrho_{0}^{1/2}][U_{\mathcal{V}}(\frac{t}{h})\varrho_{0}^{1/2}]^{*}
$$
where $F(t)=U_{\mathcal{V}}(t)\varrho^{1/2}_{0}$ is the solution to
\eqref{eq:dynalea} in 
\begin{align*}
&\mathcal{L}^{2}(L^{2}(\rz^{d}\times \Omega,dx\times
\mathcal{G};\cz))\simeq  L^{2}(\rz^{d}\times \Omega\times Z, dx\otimes
\mathcal{G}\otimes \mathbf{dz})
\\
\quad\text{with}\quad&
Z=\rz^{d}\otimes \Omega\,,\quad \mathbf{dz}=dx\otimes \mathcal{G}\,,
\end{align*}
while the trace to be computed at time $\frac{t}{h}$ equals
$$
\mathrm{Tr}\left[a^{\Weyl}(hx,D_{x})\varrho(\frac{t}{h})\right]=\int_{Z}\ez\left[\langle
F(\frac{t}{h},z)\,,\, a^{\Weyl}(hx,D_{x})F(\frac{t}{h},z)\rangle_{L^{2}(\rz^{d},dx)}\right]~\mathbf{dz}(z)\,.
$$
Thus considering  the evolution of non negative trace class operators
instead of projectors on wave functions, becomes the same problem by
introducing the suitable additional parameter $z\in Z$\,.\\
The unitary correspondence
\eqref{eq:unitFockFou1}\eqref{eq:unitFockFou2}, with
\eqref{eq:prodFockFou}\eqref{eq:WeyldGFou} and
Remark~\ref{re:finFock}, transforms the dynamics \eqref{eq:dynalea}
into
\begin{equation}
  \label{eq:dynaleaFou}
    \begin{cases}
      i\partial_{t}\hat{f}=(\xi-d\Gamma(D_{y}))^{2}\hat{f}+\sqrt{2h}\phi(V)\hat{f}\,,\\
\hat{f}(t=0)=\hat{f}_{0}\,,
    \end{cases}
\end{equation}
and the quantity \eqref{eq:quantity} into
\begin{equation}
   \label{eq:quantityFou1}
\langle \hat{f}(\frac{t}{h})\,,\, a^{\Weyl}(-hD_{\xi},\xi-d\Gamma(D_{y}))
\, \hat{f}(\frac{t}{h})\rangle_{L^{2}(\rz^{d}\times \Omega\times
              Z,\frac{d\xi}{(2\pi)^{d}}\otimes \mathcal{G}\otimes\mathbf{dz})}\,.
\end{equation}
We will see that the variable $\xi\in\rz^{d}$ and even some part $Y'$ of
the variable $Y=(y_{1},\ldots,y_{n})$, when the total number is fixed
to $n$\,,
can be taken as another parameter
like $z\in Z$ for some points of the analysis. This leads
to a parameter  $z'$-dependent, $z'=(\xi,Y',z)\in \rz^{d}\times
\rz^{dn'}\times Z$\,, analysis in $L^{2}(\rz^{d(n-n')},dY'')$\,. Those
parameters appear in 
Section~\ref{sec:centermass} by introducing the center of mass
$Y''=y_{G}=\frac{y_{1}+\cdots+y_{n}}{n}$ and the relative coordinates
$y_{j}'=y_{j}-y_{G}$\,,  a general functional framework for parameter
dependent Strichartz estimates and their consequences are presented in
Section~\ref{sec:StriCM} and finally those are detailled in
Section~\ref{sec:conseqStri} for \eqref{eq:dynaleaFou}.

\section{The Fock space and the center of mass}
\label{sec:centermass}
According to \eqref{eq:dynaleaFou} our stochastic dynamics has been
translated in a parameter dependent dynamics in the Fock space. 
We shall consider an additional unitary transform using the center of
mass and the relative variables
$$
y^{n}_{G}=\frac{y_{1}+\cdots+y_{n}}{n}\,,\quad
y_{j}'=y_{j}-y^{n}_{G}
$$
in the $n$-particles sector, $n\geq 1$\,. It trivializes the free
dynamics when $\mathcal{V}\equiv 0$ or $V\equiv 0$\,. The expression
of the interaction term $\sqrt{2h}\phi(V)$ becomes more tricky but
various general estimates are given here.

\subsection{The unitary transform associated with the center of mass}
\label{sec:unitcenter}
We shall use the following notations for $n\geq 1$:
\begin{itemize}
\item A generic element of $\rz^{dn}$ will be written
  \begin{equation}
    \label{eq:vectorRdn}
    Y_{n}=(y_{1},\ldots,y_{n})\quad \text{with}\quad 
\left|Y_{n}\right|^{2}=\sum_{j=1}^{n}|y_{j}|^{2}\,.
  \end{equation}
\item The center of mass of $Y_{n}\in \rz^{dn}$ will be written
  \begin{equation}
    \label{eq:center}
y_{G}=y_{G}^{n}=\frac{y_{1}+\cdots+y_{n}}{n}
\end{equation}
and the relative coordinates $y_{j}'=y_{j}-y_{G}^{n}$ will be gathered
into
\begin{equation}
  \label{eq:relaY}
Y_{n}'=(y_{1}',\ldots,y_{n}')=(y_{1}-y_{G}^{n},\ldots,y_{n}-y_{G}^{n})\,.
\end{equation}
The vector $Y_{n}'$ actually belongs to the subspace
$\mathcal{R}^{n}=\left\{Y_{n}\in \rz^{dn}\,, \sum_{j=1}^{n}y_{j}=0\right\}$ and
we recall
\begin{equation}
  \label{eq:quadcent}
\left|Y_{n}\right|^{2}=n|y_{G}^{n}|^{2}+\left|Y_{n}'\right|^{2}
=n|y_{G}^{n}|^{2}+\sum_{j=1}^{n}|y_{j}'|^{2}\,.
\end{equation}
\end{itemize} 
With those notations the map $\rz^{dn}\ni Y_{n}\mapsto (y_{G}^{n},
Y_{n}')\in \rz^{d}\times \mathcal{R}^{n}\subset \rz^{d}\times \rz^{dn}$ is a
measurable map and the image measure of the Lebesgue measure
$|dY_{n}|=\prod_{j=1}^{n}|dy_{j}|$ is nothing but
\begin{equation}
  \label{eq:defmun}
dy_{G}\otimes d\mu_{n}(Y'_{n})=dy_{G}\otimes [n^{d}dy_{1}\cdots dy_{n}\delta_{0}(y_{1}+\cdots+y_{n})]\,.
\end{equation}
For $n\geq 2$ we can write $d\mu_{n}(Y'_{n})=n^d \prod_{j\neq
  j_{0}}dy'_{j}$ for any fixed $j_{0}\in \left\{1,\ldots,n\right\}$ by
taking the linear coordinates $(y'_{j})_{j\neq j_{0}}$ on $\mathcal{R}^n$ where
$y'_{j_{0}}=-\sum_{j\neq j_{0}}y'_{j}$\,. For $n=1$\,,
$\mathcal{R}^{1}=\left\{0\right\}$ and integrating with respect to
$Y_{1}'=y'_{1}\in \mathcal{R}^{1}$ is nothing but
the evaluation at $y'_{1}=0$\,. 
\begin{definition}
\label{de:mesmu} On $\sqcup_{n=1}^\infty \rz^{dn}$ the measure $\mu$
carried by $\mathcal{R}=\sqcup_{n=1}^{\infty}\mathcal{R}^{n}$ is defined by
\begin{align*}
\forall g_{n}\in \mathcal{C}^{0}_{c}(\rz^{dn})\,,\quad
\int_{\mathcal{R}^n}g_{n}(Y')~d\mu_{n}(Y')
&=\int_{\rz^{dn}}g_{n}(y_{1},\ldots,y_{n})\delta_{0}(y_{1}+\cdots+y_{n})\,n^{d}~dy_{1}\cdots
dy_{n}
\\
&\stackrel{n\geq 2}{=}
\int_{\rz^{d(n-1)}}g_{n}(y'_{1},\ldots,y_{n-1}',-\sum_{j=1}^{n-1}y_{j}')\,n^{d}~dy'_{1}\cdots dy'_{n-1}\,.
\end{align*}
For  $1\leq p<+\infty$\,, the space $L^{p}(\mathcal{R},d\mu)$ is the direct sum
$\displaystyle{\bigoplus_{n=1}^{\infty}}L^{p}(\mathcal{R}^{n},d\mu_{n})$ completed with respect to the norm
$\|\displaystyle{\bigoplus_{n=1}^{\infty}}g_{n}\|_{L^{p}}=\left(\sum_{n=1}^{\infty}\|g_{n}\|_{L^{p}(\mathcal{R}^{n},d\mu_{n})}^{p}\right)^{1/p}$\,.
The closed subspace of symmetric functions,
$g_{n}(y_{\sigma(1)}',\ldots,y'_{\sigma(n)})=g_{n}(y'_{1},\ldots,y'_{n})$
for all $\sigma\in \mathfrak{S}_{n}$ and for all $n\geq 1$\,,
is then denoted by $L^{p}_{\mathrm{sym}}(\mathcal{R},d\mu(Y'))$\,.
\end{definition}
For $g_{n}\in L^{2}(\rz^{dn}\times Z, dY^{n}\otimes \mathbf{dz};\cz)$\,,
$n\geq1$\,, the function
\begin{equation}
  \label{eq:defU_G}
g_{G,n}(y_{G},Y'_{n},z)
=U_{G}g_{n}(y_{G},Y'_{n},z)=g_{n}(y_{G}+Y_{n}',z)
\end{equation}
belongs to $L^{2}(\rz^{d}\times \mathcal{R}^{n}\times Z, dy_{G}\otimes
d\mu_{n}\otimes \mathbf{dz};\cz)$ with 
\begin{eqnarray*}
  && \|U_{G}g_{n}\|_{L^{2}(\rz^{d}\times \mathcal{R}^{n}\times Z, dy_{G}\otimes
     d\mu_{n}\otimes \mathbf{dz})}=\|g_{n}\|_{L^{2}(\rz^{dn}\times
     Z,dY_{n}\otimes \mathbf{dz})}\\
\text{and}&&
g_{n}(Y_{n},z)=(U_{G}^{-1}g_{G,n})(Y_{n},z)=g_{G,n}(y^{n}_{G}, Y_{n}-y^{n}_{G},z)\,.
\end{eqnarray*}
Additionally  $U_{G}:L^{2}(\rz^{d},dy)^{\odot n}\mapsto
L^{2}(\rz^{d},dy_{G};L^{2}_{\mathrm{sym}}(\mathcal{R}^{n},d\mu_{n}))=L^{2}_{\mathrm{sym}}(\mathcal{R}^{n},d\mu_{n};L^{2}(\rz^{d},dy_{G}))$
is unitary and the same result holds for the parameter $z\in Z$ version.
\begin{proposition}
\label{pr:unit} The map $U_{G}$ extended by $U_{G}g_{0}(z)=g_{0}(z)$
for $n=0$\,, defines a unitary map
\begin{equation}
  \label{eq:unitUG}
U_{G}:L^{2}(Z,dz;\Gamma(L^{2}(\rz^{d},dy;\cz)))\to L^{2}(Z;\cz)\oplus
L_{\mathrm{sym}}^{2}(Z\times \mathcal{R}, dz\otimes d\mu; L^{2}(\rz^{d},dy_{G};\cz))\,.
\end{equation}
When $d\Gamma_{G}(A)=U_{G}[d\Gamma(A)\otimes \mathrm{Id}_{L^{2}(Z,dz)}]U_{G}^{-1}$ for a self-adjoint operator $(A,D(A))$ in
$L^{2}(\rz^{d}_{y},dy)$\,, the case $A=D_{y}$ gives
\begin{equation}
  \label{eq:dGammaDy}
d\Gamma_{G}(D_{y})=U_{G}\,d\Gamma(D_{y})\, U_{G}^{-1}=D_{y_{G}}\,.
\end{equation}
For any bounded measurable function $\phi$ on $\mathcal{R}\times Z$ the
multiplication by $\phi(Y',z)\big|_{\mathcal{R}^{n}}=\phi_{n}(Y'_{n},z)$ for
$n\geq 1$\,, while $\phi_{0}:Z\to \cz$\,, commutes with $d\Gamma_{G}(D_{y})=D_{y_{G}}$
according to 
$$
\forall t\in \rz^{d}\,, \forall u\in L^{2}(Z,\mathbf{dz};\cz)\oplus L^{2}_{\mathrm{sym}}(Z\times \mathcal{R},\mathbf{dz}\otimes d\mu;L^{2}(\rz^{d},dy_{G},\cz))\quad
e^{it\cdot D_{y_{G}}}(\phi u)=\phi(e^{it\cdot D_{y_{G}}}u)\,.
$$
A particular case is $\phi_{n}(Y'_{n},z)=\varphi(n)$ for a bounded function
$\varphi:\nz\to \cz$\,.
\end{proposition}
\begin{proof}
  The unitarity of $U_{G}$ comes at once from \eqref{eq:defU_G} and the
  componentwise unitarity already checked. 
For $d\Gamma_{G}(D_{y})=D_{y_{G}}$\,, simply write
$$
\partial_{y_{G}}g_{G,n}(y_{G},Y_{n}')=\partial_{y_{G}}g_{n}(y_{G}+y_{1}',\ldots,y_{G}+y_{n}')=
\sum_{j=1}^{n}(\partial_{y_{j}}g_{n})(y_{G}+y_{1}',\ldots,y_{G}+y_{n}')\,.
$$
The commutation statement comes from the separation of variables,
$y_{G}$ and $(Y',z)$\,.
\end{proof}
Introducing the center of mass thus simplifies the free transport part
of \eqref{eq:dynaleaFou}. It is not so for the interaction term
$\sqrt{2h}\phi(V)=\sqrt{h}[a(V)+a^{*}(V)]$\,. An explicit and useful
expression is nevertheless possible for
\begin{equation}
  \label{eq:defaGV}
  a_{G}(V)=U_{G} \, a(V) \, U_{G}^{-1}\quad\text{and}\quad
  a_{G}^{*}(V)=U_{G} \, a^{*}(V) \, U_{G}^{-1}\,.
\end{equation}
\begin{proposition}
  The operator $a_{G}(V)$ and $a_{G}^{*}(V)$ for $V\in
  L^{2}(\rz^{d},dy;\cz)$ have the following action on
  $f_{G,n}\in L^{2}_{\mathrm{sym}}(\mathcal{R}^{n}\times Z,d\mu_{n}\otimes
  \mathbf{dz};L^{2}(\rz^{d},dy_{G};\cz))$ for $n\geq 1$ and $f_{G,0}\in
  L^{2}(Z,\mathbf{dz};\cz)$ where we omit the transparent variable $z\in Z$:
  \begin{eqnarray}
    \label{eq:aG01}&&
    a_{G}(V)f_{G,0}=0\,,\quad
    [a_{G}(V)f_{G,1}]=\int_{\rz^{d}}\overline{V(y_{1})}f_{G,1}(y_{1})~dy_{1}\,,\\
\nonumber
\forall n>1\,,&&\hspace{-0.5cm}
      [a_{G}(V)f_{G,n}](y_{G},Y'_{n-1})=\sqrt{n}\int_{\rz^{d}}\overline{V(y_{G}+y_{n})}f_{G,n}(y_{G}+\frac{y_{n}}{n},
      Y_{n}-\frac{y_{n}}{n})~dy_{n}\,,\\
\label{eq:aGVn}
\text{with}&& Y_{n}=(y_{1}',\ldots,y_{n-1}',y_{n})\in
              \rz^{dn}\,,\quad Y'_{n-1}\in \mathcal{R}^{n-1}\,,\quad
              Y_{n}-\frac{y_{n}}{n}\in \mathcal{R}^{n}\,,
\\
&&
a_{G}^{*}(V)f_{G,0}(y_{G})=V(y_{G})f_{G,0}\,,\\
\nonumber
\forall n>0\,,&&
a_{G}^{*}(V)f_{G,n}(y_{G},Y'_{n+1})=\sqrt{n+1}S_{n+1}[V(y_{G}+y'_{n+1})f_{G,n}(y_{G}-\frac{y_{n+1}'}{n},Y_{n}+\frac{y_{n+1}'}{n})]\,,\\
\nonumber
\text{with}&&
Y_{n}=(y_{1}',\ldots,y_{n}')\in\rz^{dn}\,,\quad Y'_{n+1}\in
                   \mathcal{R}^{n+1}\,,\quad
Y_{n}+\frac{y_{n+1}'}{n}\in \mathcal{R}^{n}\,,
\\
\label{eq:aG*n}
\text{and}
&&
S_{n+1}[v(y'_{n+1})u(y_{1}',\ldots,y_{n+1}')]=
\frac{1}{(n+1)!}\sum_{\sigma\in \mathfrak{S}_{n+1}}v(y'_{\sigma(n+1)})u(y'_{\sigma(1)},\ldots,y'_{\sigma(n+1)})\,.
  \end{eqnarray}
\end{proposition}
\begin{proof}
Write for $n>1$\,,
\begin{align*}
[a_{G}(V)f_{G,n}](y^{n-1}_{G},Y'_{n-1})&=[a(V)U_{G}^{-1}f_{G,n}](Y'_{n-1}+y^{n-1}_{G})
\\
&=
\sqrt{n}\int_{\rz^{d}}\overline{V(\tilde{y}_{n})}[U_{G}^{-1}f_{G,n}](Y'_{n-1}+y^{n-1}_{G},\tilde{y}_{n})~d\tilde{y}_{n}\,.
\end{align*}
By setting $\tilde{y}_{n}=y_{G}^{n-1}+y_{n}$ the formula
$(U_{G}^{-1}g_{G,n})(\cdot )=g_{G,n}(y_{G}^{n},\cdot -y_{G}^{n})$ with
$$
y_{G}^{n}=\frac{y_{1}+\cdots+y_{n-1}+\tilde{y}_{n}}{n}=\frac{n-1}{n}y_{G}^{n-1}+\frac{\tilde{y}_{n}}{n}=y_{G}^{n-1}+\frac{y_{n}}{n}
$$
leads to 
\begin{align*}
  [a_{G}(V)f_{G,n}](y^{n-1}_{G},Y'_{n-1})&=
\sqrt{n}\int_{\rz^{d}}\overline{V(y_{G}^{n-1}+y_{n})}f_{G,n}(y_{G}^{n-1}+\frac{y_{n}}{n},Y'_{n-1}-\frac{y_{n}}{n},y_{n}-\frac{y_{n}}{n})~dy_{n}\\
&= \sqrt{n}\int_{\rz^{d}}\overline{V(y_{G}^{n-1}+y_{n})}f_{G,n}(y_{G}^{n-1}+\frac{y_{n}}{n},Y_{n}-\frac{y_{n}}{n})~dy_{n}
\end{align*}
with $Y_{n}=(y_{1}',\ldots,y'_{n-1},y_{n})$\,.
\\
The computation of $a_{G}^{*}(V)f_{G,n}$ is done by duality:
\begin{align*}
  \langle a_{G}^{*}(V)f_{G,n-1}\,,\, g_{G,n}\rangle &=
\langle f_{G,n-1}\,,\, a_{G}(V)g_{G,n}\rangle
\\
&=
\int_{\rz^{d}\times
    \mathcal{R}^{n-1}}\overline{f_{G,n-1}(y_{G}^{n-1},Y'_{n-1})}\times\\
&\left[\sqrt{n}
\int_{\rz^{d}}\overline{V(y_{G}^{n-1}+y_{n})}g_{G,n}(y^{n-1}_{G}+\frac{y_{n}}{n},Y_{n}-\frac{y_{n}}{n})~dy_{n}\right]~dy_{G}^{n-1}d\mu_{n-1}(Y'_{n-1})\,.
\end{align*}
Remember $Y_{n}=(y_{1}',\ldots,y'_{n-1},y_{n})$ and
$\tilde{Y}_{n}'=Y_{n}-\frac{y_{n}}{n}\in \mathcal{R}^{n}$\,. The change of variables
$$
\tilde{Y}_{n}'=Y_{n}-\frac{y_{n}}{n}\,,\quad
y_{G}^{n}=y_{G}^{n-1}+\frac{y_{n}}{n}
\,,\quad y_{G}^{n-1}=y_{G}^{n}-\frac{\tilde{y}_{n}}{n-1}
\,,\quad Y'_{n-1}=\tilde{Y}_{n-1}+\frac{y_{n}}{n}=\tilde{Y}_{n-1}+\frac{\tilde{y}_{n}}{n-1}\,,
$$
with 
\begin{align*}
dy_{n}dy_{G}^{n-1}d\mu^{n-1}(Y'_{n-1})&=
dy_{G}^{n-1}\delta_{0}(y_{1}'+\cdots+y_{n-1}')(n-1)^{d}dy_{1}'\cdots
dy_{n-1}'
\\
&=
dy_{G}^{n}\frac{n^{d}}{(n-1)^{d}}(n-1)^{d}\delta_{0}(\tilde{y}_{1}'+\cdots+\tilde{y}'_{n})d\tilde{y}_{1}'\cdots
    d\tilde{y}_{n}'
\\
&=
dy_{G}^{n}d\mu_{n}(\tilde{Y}'_{n})\,,
\end{align*}
gives
\begin{multline*}
\langle a_{G}^{*}(V) f_{G,n-1}\,,\, g_{G,n}\rangle 
= \sqrt{n}\int_{\rz^{d}\times
  \mathcal{R}^{n}}\overline{V(y_{G}^{n}+\tilde{y}_{n}')f_{G,n-1}(y_{G}^{n}-\frac{\tilde{y}_{n}}{n-1},\tilde{Y}_{n-1}'+\frac{\tilde{y}_{n}'}{n-1})}\times\\
g_{G,n}(y_{G}^{n},\tilde{Y}_{n}')~dy^{n}_{G}d\mu_{n}(\tilde{Y}'_{n})\,.
\end{multline*}
Replacing $n$ by $n+1$\,, while remembering that $a_{G}^{*}(V)f_{G,n}$
is symmetric in the variables $(y_{1}',\ldots,y_{n+1}')$ yields
$$
[a_{G}^{*}(V)f_{G,n}](y_{G},Y_{n+1}')
=\sqrt{n+1}S_{n+1}[V(y_{G}+y'_{n+1})f_{G,n}(y_{G}-\frac{y'_{n+1}}{n},Y_{n}+\frac{y'_{n+1}}{n})]
$$
with $Y_{n}=(y_{1}',\ldots,y_{n}')$\,.
\end{proof}

\subsection{General $L^{p}_{x}L^{q}_{y}$ spaces}
\label{sec:genLpxLqy}
When $(\mathcal{X},\mathbf{dx})$ and $(\mathcal{Y},\mathbf{dy})$ are sigma-finite
measured spaces $L^{p}_{x}L^{q}_{y}$\,, $1\leq p,q\leq +\infty$\,,
denotes the space
$L^{p}_{x}L^{q}_{y}=L^{p}(\mathcal{X},\mathbf{dx};L^{q}(\mathcal{Y},\mathbf{dy}))$\,.
This shortened notation is especially useful when estimates are
written in those spaces, like in Strichartz estimates (see
Section~\ref{sec:StriCM}). However the final space of the unitary map
$U_{G}$ in \eqref{eq:unitUG} shows already that the product space
$\mathcal{X}\times \mathcal{Y}$ is too restrictive. Below is a convenient generalization.
\begin{definition}
\label{de:LpxLqy}
Let $(\mathcal{X}_{n},\mathbf{dx_{n}})_{n\in\mathcal{N}}$ and $(\mathcal{Y}_{n},\mathbf{dy_{n}})_{n\in
\mathcal{N}}$ be at most countable families ($\mathcal{N}\subset \nz$)
of sigma-finite measured spaces. Let $\mathcal{X}=\sqcup_{n\in
  \mathcal{N}}\mathcal{X}_{n}$ and $\mathcal{Y}=\sqcup_{n\in \mathcal{N}}\mathcal{Y}_{n}$ be endowed with
the measures $\mathbf{dx}=\Sigma_{n\in \mathcal{N}}\mathbf{dx_{n}}$
and $\mathbf{dy}=\Sigma_{n\in \mathcal{N}}\mathbf{dy_{n}}$\,. In this
framework, the
space $L^{p}_{x}L^{q}_{y}$\,, $1\leq p,q\leq +\infty$\,, will denote
the closed subspace of $L^{p}(\mathcal{X},\mathbf{dx}; L^{q}(\mathcal{Y},\mathbf{dy}))$
given by
$$
L^{p}_{x}L^{q}_{y}=\left\{f\in
  L^{p}(\mathcal{X},\mathbf{dx};L^{q}(\mathcal{Y},\mathbf{dy}))\,,\quad f(x,y)=\sum_{n\in
  \mathcal{N}} 1_{\mathcal{X}_n}(x)1_{\mathcal{Y}_n}(y) \, f(x,y)\quad \mathrm{a.e.} \right\}\,.
$$ 
\end{definition}
The above definition is coherent with the specific product case, which
is the particular case $\mathcal{N}=\left\{0\right\}$\,. The
differences will be clear from the different frameworks when
$(\mathcal{X}_{n},\mathbf{dx_{n}})_{n\in \mathcal{N}}$ and
$(\mathcal{Y}_{n},\mathbf{dy_{n}})_{n\in \mathcal{N}}$ will be specified.\\
The two following properties of the product case are still valid in
this extended framework:
\begin{itemize}
\item 
The dual of $L^{p}_{x}L^{q}_{y}$\,, $1\leq q,p<+\infty$
is $L^{p'}_{x}L^{q'}_{y}$ with $\frac{1}{q'}+\frac{1}{q}=1$ and
$\frac{1}{p'}+\frac{1}{p}=1$\,.
\item  Minkowski's inequality says
\begin{equation}
  \label{eq:Minkowski}
  \|f\|_{L^{p}_{x}L^{q}_{y}}\leq
  \|f\|_{L^{q}_{y}L^{p}_{x}}\quad\text{for}\quad 1\leq q\leq p\leq +\infty\,.
\end{equation}
\end{itemize}
Below are examples, associated with the decomposition associated with
the introduction of the center of mass \eqref{eq:center} and the
relative coordinates \eqref{eq:relaY}, where those notations will be used
\begin{itemize}
\item $\mathcal{N}=\left\{n\right\}$\,, $n\geq 1$\,,
  $\mathcal{X}_{n}=\mathcal{R}_{n}\times Z'$\,, $\mathbf{dx_{n}}=d\mu_{n}\otimes
  \mathbf{dz'}$\,, $\mathcal{Y}_{n}=\rz^{d}$\,, $\mathbf{dy_{n}}=dy_{G}$ and 
$$
L^{p}_{(Y'_{n},z')}L^{q}_{y_{G}}=L^{p}_{x_{n}}L^{q}_{y_{G}}=L^{p}(\mathcal{R}_{n}\times Z',d\mu_{n}\otimes
\mathbf{dz'}; L^{q}(\rz^{d},dy_{G}))\,.
$$
The notation $L^{p}_{(Y'_{n},z'),\text{sym}}L^{q}_{y_{G}}$ will stand for
the closed subspace of functions which are symmetric with respect to
the variables $Y'_{n}\in \mathcal{R}_{n}$\,.
\item $\mathcal{N}=\left\{0,1\right\}$ with 
  \begin{eqnarray*}
    && \mathcal{X}_{0}=Z'\,,\quad \mathcal{X}_{1}=\mathcal{R}\times
       Z'=(\sqcup_{n=1}^{\infty}\mathcal{R}_{n})\times Z'\,,\quad
       \mathbf{dx_{0}}=\mathbf{dz'}\,,\quad
       \mathbf{dx_{1}}=d\mu\otimes \mathbf{dz'}\,,\\
&& \mathcal{Y}_{0}=\left\{0\right\}\,,\quad
   \mathcal{Y}_{1}=\rz^{d}\,,\quad
   \mathbf{dy_{0}}=\delta_{0}\,,\quad \mathbf{dy_{1}}=dy_{G}\,,
  \end{eqnarray*}
where
$$
L^{p}_{(Y',z')}L^{q}_{y_{G}}=L^{p}(Z',\mathbf{dz'})\oplus
L^{p}(\mathcal{R}\times Z',d\mu\otimes \mathbf{dz'};L^{q}(\rz^{d},dy_{G}))\,.
$$
With the same convention as above for
$L^{p}_{(Y',z'),\text{sym}}L^{q}_{y_{G}}$\,, which refers to the symmetry
 for the $Y'\in \mathcal{R}$ variable, the formula \eqref{eq:unitUG}
 becomes
$$
U_{G}: L^{2}(Z',\mathbf{dz'};\Gamma(L^{2}(\rz^{d},dy;\cz)))
\to L^{2}_{(Y',z'),\text{sym}}L^{2}_{y_{G}}\,.
$$
The general  spaces $L^{2}_{(Y',z'),\text{sym}}L^{p}_{y_{G}}$\,, $1\leq
p\leq +\infty$\,, will be especially
useful after Section~\ref{sec:StriCM}.
\item The previous example can be written with $\mathcal{N}=\nz$ and
  \begin{eqnarray*}
&&    \mathcal{X}_{0}=Z'\,,\quad
   \mathcal{X}_{n>0}=\mathcal{R}_{n}\times Z'\,,\quad
   \mathbf{dx_{0}}=\mathbf{dz'}\,,\quad
   \mathbf{dx_{n>0}}=d\mu_{n}\otimes\mathbf{dz'}\,,\\
&& \mathcal{Y}_{0}=\left\{0\right\}\,,\quad
   \mathcal{Y}_{n>0}=\rz^{d}\,,\quad
   \mathbf{dy_{0}}=\delta_{0}\,,\quad \mathbf{dy_{n>0}}=dy_{G}\,.
  \end{eqnarray*}
\end{itemize}
\subsection{$L^{p}_{y_{G}}$-Estimates for $a_{G}(V)$ and
  $a_{G}^{*}(V)$}
\label{sec:gausFp}
General $L^{p}$-esptimates, or more precisely
$L^{2}_{(Y',z),\text{sym}}L^{p}_{y_{G}}$-estimates,  are proved in
this paragraph for the operators $a_{G}(V)$ and 
$a_{G}^{*}(V)$\,. 
The use of the center of mass and the $L^{p}_{y_{G}}$ spaces,
will be extremely useful for the application of Strichartz estimates
in Section~\ref{sec:StriCM}.\\
Let us start with a simple application of Young's inequality.

\begin{lemma}
\label{le:borne-L2x-Lqprimey-1}
For any $q',p'\in [1,2]$ such that $q'\leq p'$\,,  let $r'\in [1,2]$
be defined by $\frac{1}{r'}=\frac{1}{2}+\frac{1}{q'}-\frac{1}{p'}$\,.
The inequality
\[
\|V(y_{G}+y^{\prime})\varphi(y_{G})\|_{L_{y'}^{2}L_{y_{G}}^{q'}}\leq \|V\|_{L^{r'}}\|\varphi\|_{L^{p'}}\,,
\]
holds for all $V\in L^{r'}(\rz^{d},dy;\cz)$ and all $\varphi\in L^{p'}(\mathbb{R}^{d},dy;\mathbb{C})$\,.
\end{lemma}

\begin{proof}
  The conditions
  $\frac{1}{r'}+\frac{1}{p'}=\frac{1}{2}+\frac{1}{q'}$\,, $1\leq q'\leq
  p'\leq 2$\,, ensure
$$
\frac{1}{r'}=\frac{1}{2}+\frac{1}{q'}-\frac{1}{p'}\in
[\frac{1}{2},1]\quad\text{and}\quad r'\in[1,2]\,.
$$
Young's inequality with $\frac{1}{\tilde{r}}+\frac{1}{\tilde{p}}=\frac{1}{\frac{2}{q'}}+1$
and $\tilde{r},\tilde{p},\frac{2}{q'}\geq1$
yields 
\begin{multline*}
\|V(y_{G}+y^{\prime})\varphi(y_{G})\|_{L_{y'}^{2}L_{y_{G}}^{q'}}\leq\||V|(y_{G}-y^{\prime})|\varphi|(y_{G})\|_{L_{y'}^{2}L_{y_{G}}^{q'}}=\||V(-\cdot)|^{q'}*|\varphi|^{q'}\|_{L^{2/q'}}^{1/q'} \\
\leq\||V|^{q'}\|_{L^{\tilde{r}}}^{1/q'}\||\varphi|^{q'}\|_{L^{\tilde{p}}}^{1/q'}\,.
\end{multline*}
By taking $\tilde{p}=\frac{p'}{q'}\in[1,2]$ and $r'=\tilde{r}q'$ we
obtain
$$
\|V(y_{G}+y^{\prime})\varphi(y_{G})\|_{L_{y'}^{2}L_{y_{G}}^{q'}}\leq\|V\|_{L^{r'}}\|\varphi\|_{L^{p'}}
$$
\end{proof}
The first result concerns the action of $a_{G}(V)$ and $a_{G}^{*}(V)$
on a fixed finite particles sector.
\begin{proposition}
\label{pr:Lpagstag}
For any $p',q'\in [1,2]$ such that $q'\leq p'$\,,  $2\leq p\leq q\leq
+\infty$\,, let $r'\in[1,2]$ be defined by
$\frac{1}{r'}=\frac{1}{2}+\frac{1}{q'}-\frac{1}{p'}$ like in
Lemma~\ref{le:borne-L2x-Lqprimey-1}. For any $V\in
L^{q'}(\rz^{d},dy;\cz)\cap L^{r'}(\rz^{d},dy;\cz)$\,, the creation and annihilation operators
satisfy the following estimates:
\begin{eqnarray}
  \label{eq:borne-a-star0}
\forall f_{G,0}\in L^{2}_{z}\,,&&
\hspace{-0.5cm}
  \|a_{G}^{*}(V)f_{G,0}\|_{L^{2}_{z}L_{y_{G}}^{q'}}\leq
  \|V\|_{L^{q'}}\|f_{G,0}\|_{L^{2}_{z}}\,,\\
\label{eq:borne-a-star(V)}
\forall n>0,\forall f_{G,n}\in
 L^{2}_{(Y_{n}^{\prime},z),\rmsym}L_{y_{G}}^{p'}\,,&&\hspace{-0.5cm}
\|a_{G}^{*}(V)\,f_{G,n}\|_{L_{(Y^{\prime,z)}_{n+1}}^{2}L_{y_{G}}^{q'}}\leq
  \|V\|_{L^{r'}}\sqrt{n+1}\|f_{G,n}\|_{L_{(Y_{n}^{\prime},z)}^{2}L_{y_{G}}^{p'}}\,,\\
\label{eq:borne-a-1}
\forall f_{G,1}\in L^{2}_{z}L^{q}_{y_{G}}\,,&&\hspace{-0.5cm}
\|a_{G}(V)f_{G,1}\|_{L^{2}_{z}}\leq\|V\|_{L^{q'}}\|f_{G,1}\|_{L^{2}_{z}L_{y_{G}}^{q}}\,,\\
\label{eq:borne-a(V)}
\forall n>1, \forall f_{G,n}\in L^{2}_{(Y^{\prime}_{n},z),\rmsym}L^{q}_{y_{G}}\,,&&\hspace{-0.5cm}
\|a_{G}(V)\,f_{G,n}\|_{L_{(Y_{n-1}^{\prime},n)}^{2}L_{y_{G}}^{p}}\leq \|V\|_{L^{r'}}\sqrt{n}\|f_{G,n}\|_{L_{(Y_{n}^{\prime},z)}^{2}L_{y_{G}}^{q}}\,.
\end{eqnarray}
A notable case is when $q'=r'$ and $p'=p=2$\,.
\end{proposition}

\begin{proof}
The variable $z\in Z$ is actually a parameter which can be forgotten
because our estimates are uniform w.r.t.~$z\in Z$\,.\\
For \eqref{eq:borne-a-star0} it suffices to notice
$[a^{*}_{G}(V)f_{G,0}](y_{G})=f_{G,0}\times V(y_{G})$\,.\\
The estimate of $a_{G}^{*}(V)f_{G,n}$ for  $n>0$ relies on
Lemma~\ref{le:borne-L2x-Lqprimey-1}. We start from the expression~\eqref{eq:aG*n}
$$
(a_{G}(V)^{*}f_{G,n})(y_{G},Y_{n+1}^{\prime})
=\sqrt{n+1}\,S_{n+1}\,V(y_{G}+y_{n+1}^{\prime})\,f_{G,n}(y_{G}-\frac{y_{n+1}^{\prime}}{n},Y_{n}+\frac{y_{n+1}^{\prime}}{n})
$$
with $Y'_{n+1}=(y'_{1},\ldots,y'_{n},y'_{n+1})\in \mathcal{R}^{n+1}$\,, $Y_{n}=(y'_{1},\ldots,y'_{n})\in \rz^{dn}$\,,
$Y_{n}+\frac{y'_{n+1}}{n}\in \mathcal{R}^{n}$\,.
The symmetrization $S_{n+1}$ simply takes the
average of $n+1$-terms which have all the same form as
$$
\sqrt{n+1}V(y_{G}+y_{n+1}^{\prime})\,f_{G,n}(y_{G}-\frac{y_{n+1}^{\prime}}{n},Y_{n}+\frac{y_{n+1}^{\prime}}{n})\,,
$$
after circular permutation of the variables $y'_{j}$ which does not
change the $L^{2}_{Y'_{n+1}}L^{q'}_{y_{G}}$-norm. We can therefore forget the
symmetrization $S_{n+1}$ for proving the upper bound
\eqref{eq:borne-a-star(V)}.
When $n>1$ integrations must be performed with respect to the
independent variables $(y_{2}',\ldots,y_{n}')\in
\rz^{d(n-1)}$\,. Remember that $(y_{2}',\ldots,y'_{n},y'_{n+1})$ are
coordinates on $\mathcal{R}^{n+1}$ such that $y'_{1}=-y'_{2}\cdots
-y'_{n}-y'_{n+1}$\,,
$d\mu_{n+1}(Y'_{n+1})=(n+1)^{d}dy_{2}'\cdots dy_{n+1}'$ and that the quantity
$$
\left\|V(y_{G}+y'_{n+1})f_{G,n}(y_{G}-\frac{y_{n+1}}{n},Y_{n}+\frac{y'_{n+1}}{n})\right\|_{L^{2}_{Y'_{n+1}}L^{q'}_{y_{G}}}
$$
equals
$$
(n+1)^{d/2}\|V(y_{G}+y'_{n+1})f_{G,n}(y_{G}-\frac{y'_{n+1}}{n},Y_{n}+\frac{y'_{n+1}}{n}))\|_{L^{2}(\rz^{d},dy'_{n+1};L^{2}(\rz^{d(n-1)},
  dy'_{2}\cdots dy'_{n};L^{q'}_{y_{G}}))}\,.
$$
When $y'_{n+1}\in\rz^{d}$ is fixed, setting
$y'_{1}=-\sum_{j=2}^{n}(y'_{h}+\frac{y'_{n+1}}{n})$ and
$Y'_{n}=Y_{n}+\frac{y'_{n+1}}{n}$\,, provides the  coordinates
$(y_{2}'+\frac{y'_{n+1}}{n},\ldots,
y'_{n}+\frac{y'_{n+1}}{n})$ on $\mathcal{R}^{n}$ with
$d\mu_{n}(Y'_{n})=n^{d/2}dy'_{2}\cdots dy'_{n}$\,,   and then
$\|V(y_{G}+y'_{n+1}) \, f_{G,n}(y_{G}-\frac{y'_{n+1}}{n},
Y_{n}+\frac{y'_{n+1}}{n})\|_{L^{2}(\rz^{d(n-1)},dy'_{2}\cdots
dy'_{n};L^{q'}_{y_{G}})}$ equals
$$
n^{-d/2}\|
V(y_{G}+y'_{n+1})f_{G,n}(y_{G}-\frac{y'_{n+1}}{n},Y'_{n})\|_{L^{2}_{Y'_{n}}L^{q'}_{y_{G}}}=
n^{-d/2}\|V(\tilde{y}_{G}+\frac{n+1}{n}y'_{n+1})f_{G,n}(\tilde{y}_{G},Y'_{n})\|_{L^{2}(Y'_{n}L^{q'}_{\tilde{y}_{G}})}\,.
$$
We deduce
\begin{align*}
  &\hspace{-1cm}\left\|V(y_{G}+y'_{n+1})f_{G,n}(y_{G}-\frac{y'_{n+1}}{n},Y_{n}+\frac{y'_{n+1}}{n})\right\|_{L^{2}_{Y'_{n+1}}L^{q'}_{y_{G}}}
\\
&\leq
\frac{(n+1)^{d/2}}{n^{d/2}}\left\|\|V(\tilde{y}_{G}+\frac{n+1}{n}y'_{n+1})f_{G,n}(\tilde{y}_{G},Y'_{n})\|_{L^{2}_{Y'_{n}}L^{q'}_{\tilde{y}_{G}}}\right\|_{L^{2}(\rz^{d},dy'_{n+1})}\\
&\leq
\left\|\|V(\tilde{y}_{G}+y')f_{G,n}(\tilde{y}_{G},Y'_{n})\|_{L^{2}_{Y'_{n}} L^{q'}_{\tilde{y}_{G}}}\right\|_{L^{2}(\rz^{d},dy')}
\\
&\leq 
\left\|\|V(\tilde{y}_{G}+y')f_{G,n}(\tilde{y}_{G},Y'_{n})\|_{L^{2}(\rz^{d},dy';L^{q'}_{\tilde{y}_{G}})}\right\|_{L^{2}_{Y'_{n}}}\,,
\end{align*}
after using the change of variable $y'=\frac{n+1}{n}y'_{n+1}$ in
$\rz^{d}$ for the third line and
$L^{2}_{y'}L^{2}_{Y'_{n}}=L^{2}_{Y'_{n}}L^{2}_{y'}$ for the last one. We now use 
Lemma~\ref{le:borne-L2x-Lqprimey-1} with
$$
\|V(\tilde{y}_{G}+y')f_{G,n}(\tilde{y}_{G},Y'_{n})\|_{L^{2}(\rz^{d},dy';L^{q'}_{\tilde{y}_{G}})}\leq\|V\|_{L^{r'}}\|f_{G,n}(\tilde{y}_{G},Y'_{n})\|_{L^{p'}_{\tilde{y}_{G}}}
$$ 
for almost all $Y'_{n}\in \mathcal{R}^{n}$ and $1\leq q'\leq p'\leq
2$\,, $\frac{1}{r'}=\frac{1}{2}+\frac{1}{q'}-\frac{1}{p'}$\,. Integrating w.r.t.~$Y'_{n}\in
\mathcal{R}^{n}$ gives
$$
  \left\|V(y_{G}+y'_{n+1})f_{G,n}(y_{G}-\frac{y'_{n+1}}{n},Y_{n}+\frac{y'_{n+1}}{n})\right\|_{L^{2}_{Y'_{n+1}}L^{q'}_{y_{G}}}
\leq
\|V\|_{L^{r'}}\|f_{G,n}\|_{L^{2}_{Y'_{n}}L^{p'}_{y_{G}}}\,.
$$
By multiplying  by $\sqrt{n+1}$ and with the symmetrization
$S_{n+1}$\,, we have proved
\eqref{eq:borne-a-star(V)}.\\
 The estimates with~$a_{G}(V)$
follow by duality using
\[
\|a_{G}(V)f_{G,n+1}\|_{L_{Y_{n+1}^{\prime}}^{2}L_{y_{G}}^{p}}=\sup_{\|f_{G,n}\|_{L_{Y_{n}^{\prime}}^{2}L_{y_{G}}^{p'}}=1}|\langle a_{G}(V)f_{G,n+1},f_{G,n}\rangle|\,.
\]
Indeed, if~$\|f_{G,n}\|_{L_{Y_{n}^{\prime}}^{2}L_{y_{G}}^{p'}}=1$,
then
\begin{align*}
|\langle a_{G}(V)f_{G,n+1},f_{G,n}\rangle|&=|\langle f_{G,n+1},a_{G}^{*}(V)f_{G,n}\rangle|\\
&\leq\|f_{G,n+1}\|_{L_{Y_{n+1}^{\prime}}^{2}L_{y_{G}}^{q}}\|a_{G}^{*}(V)f_{G,n}\|_{L_{Y_{n}^{\prime}}^{2}L_{y_{G}}^{q'}}\\
&\leq
\left\{
  \begin{array}[c]{ll}
       \|V\|_{L^{q'}}\|f_{G,1}\|_{L_{y_{G}}^{q}}&\text{when}~n=0\,,\\ 
\|V\|_{L^{r'}}\sqrt{n+1}\|f_{G,n+1}\|_{L_{Y_{n+1}^{\prime}}^{2}L_{y_{G}}^{q}}\,,&\text{when}~n>0\,,
  \end{array}
 \right.
 \end{align*}
which implies the bounds \eqref{eq:borne-a-1} and \eqref{eq:borne-a(V)}.
\end{proof}
\begin{remark}
\label{re:Brasc} Instead of Young's inequality one could  use the more general
Brascamp-Lieb inequality (see\cite{BrLi}\cite{Lie}). This would not
change the result (up to multiplicative constants). One may wonder whether it is possible to improve
Lebesgue's exponent, in particular the integrability by reaching
exponents $p<2$ in \eqref{eq:borne-a(V)} by strengthening the
assumptions on $V$\,. Actually it is not. Take $V\in
\mathcal{S}(\rz^{d})$ and $\varphi\in L^{2}(\rz^{d},dy;\cz)$\,, then
$a(V)\varphi^{\otimes n}=\sqrt{n}\langle V,\varphi\rangle
\varphi^{\otimes n-1}$ and $a_{G}(V)U_{G}^{-1}(\varphi^{\otimes n})$ cannot be
put in $L^{2}_{z,Y'_{n-1}}L^{p}_{y_{G}}$ with $p<2$ in general.
\end{remark}

\begin{proposition}
\label{pr:expaLp}
Take $\alpha,\alpha'\in\mathbb{R}$,
$\alpha<\alpha'$  and for $1\leq q'\leq p'\leq2$\,, $2\leq p\leq q\leq
+\infty$\,, and let $r'\in [1,2]$  be defined by
$\frac{1}{r'}=\frac{1}{2}+\frac{1}{q'}-\frac{1}{p'}$\,.
For any $V\in L^{r'}(\rz^{d})\cap L^{q'}(\rz^{d})$\,, the following
estimates hold
\begin{align}
  \label{eq:expalphast}
\forall f\in e^{-\alpha' N}L^{2}_{z,Y',\rmsym}L^{p'}_{y_{G}}\,,&\;
\|e^{\alpha N}a_{G}^*(V)f\|_{L^{2}_{z,Y'}L^{q'}_{y_{G}}}\leq
  \frac{\max(\|V\|_{L^{r'}},\|V\|_{L^{q'}})e^{\alpha'}}{2\sqrt{\alpha'-\alpha}}\|e^{\alpha' N}f\|_{L^{2}_{z,Y'}L^{p'}_{y_{G}}}\,,
\\
\label{eq:expalpha}
\forall f\in e^{-\alpha' N}L^{2}_{z,Y',\rmsym}L^{q}_{y_{G}}\,,&\;
 \|e^{\alpha N}a_{G}(V)f\|_{L^{2}_{z,Y'}L^{p}_{y_{G}}}\leq
\frac{\max(\|V\|_{L^{r'}},\|V\|_{L^{q'}})e^{-\alpha}}{2\sqrt{\alpha'-\alpha}}
\|e^{\alpha' N}f\|_{L^{2}_{z,Y'}L^{q}_{y_{G}}}\,.
\end{align}
Again, a notable case is when $q'=r'$ and $p=p'=2$\,.
\end{proposition}

\begin{proof}
By writing 
\begin{align*}
 &&  e^{\alpha N}a_{G}^{*}(V)e^{-\alpha' N}(\displaystyle{\bigoplus_{n=0}^{\infty}}f_{G,n})
&=\displaystyle{\bigoplus_{n=0}^{\infty}} e^{\alpha(n+1)-\alpha'n}a_{G}^{*}(V)f_{G,n}\,,\\
\text{and} && e^{\alpha N}a_{G}(V)e^{-\alpha'N}(\displaystyle{\bigoplus_{n=0}^{\infty}}f_{G,n})&=
\displaystyle{\bigoplus_{n=1}^{\infty}}e^{\alpha(n-1)-\alpha'n}a_{G}(V)f_{G,n}\,,
\end{align*}
Proposition~\ref{pr:Lpagstag} tells us that it suffices to bound
\begin{align*}
  &&
     \sup_{n\in\nz}\sqrt{n+1}e^{-(\alpha'-\alpha)(n+1)}e^{\alpha'} &\leq
     \frac{e^{\alpha'}}{\sqrt{2e}\sqrt{\alpha'-\alpha}}
\leq \frac{e^{\alpha'}}{2\sqrt{\alpha'-\alpha}}\,,
\\
\text{and}&&
\sup_{n\in\nz}\sqrt{n}e^{-(\alpha'-\alpha)n}e^{-\alpha} &\leq
             \frac{e^{-\alpha}}{\sqrt{2e}\sqrt{\alpha'-\alpha}}
\leq\frac{e^{-\alpha}}{2\sqrt{\alpha'-\alpha}}\,.
\end{align*}
\end{proof}

\section{Strichartz estimates in the center of mass variable}
\label{sec:StriCM}
Here we review the celebrated results of Keel and Tao in \cite{KeTa}
and adapt them to our framework.
We shall use like those authors the  short notations 
\begin{itemize}
\item $a(z)\lesssim b(z)$ for the uniform inequality 
$$
\forall z\in Z\,, \quad a(z)\leq C b(z)\,,
$$
where $C$ is a constant which depends only on the following data: the
dimension $d$ or the free one particle evolution on $\rz^{d}$\,;
\item for $1\leq p,q\leq+\infty$\,, various uses of
the general notation $L^{p}_{x}L^{q}_{y}$ introduced in
Definition~\ref{de:LpxLqy} will be specified;
\item except in specified cases, $L^{p}_{x}$ is used for $2\leq p\leq +\infty$ while $L^{p'}_{x}$
  is used for $1\leq p'\leq 2$\,.
\end{itemize}

\subsection{Endpoint Strichartz estimates}
\label{sec:endpoint}
Keel and Tao's results about endpoint Strichartz
estimates (see \cite{KeTa}) written with uniform inequalities,
obviously 
induce a parameter dependent version which will be needed.
They start with a time-dependent operator $U(t):\mathfrak{h}_{\rmin}\to
L^{2}_{x}=L^{2}(X,\mathbf{dx};\cz)$ where $t\in \rz$ and~$\mathfrak{h}_{\rmin}$ is a (separable)
Hilbert space of initial data.
We rather consider
a parameter dependent operator $U(t,z_{1}): \mathfrak{h}_{\rmin}\to L^{2}_{x}$
defined for $(t,z_{1})\in \rz\times Z_{1}$ such that
\begin{eqnarray}
\label{eq:hypStri1}
&&  \|U(t,z_{1})f\|_{L^{2}_{x}}\lesssim \|f\|_{\mathfrak{h}_{\rmin}}\,,\\
\label{eq:hypStri2}
&&\|U(t,z_{1})U^{*}(s,z)g\|_{L^{\infty}_{x}}\lesssim
  \frac{\|g\|_{L^{1}_{x}}}{|t-s|^{\sigma}}\quad \text{for~all}~t\neq s\,,
\end{eqnarray}
while $U^{*}(t,z_{1})$ may be defined only on a dense set of
$L^{1}_{x}$\,.\\ 
On the measured space $(Z_{1},\mathbf{dz_{1}})$ the map $(t,z_{1})\mapsto U(t,z_{1})f\in L^{2}_{x}$ is
assumed measurable for all $f\in \mathfrak{h}_{\rmin}$ and
$U(t):L^{w}(Z_{1},\mathbf{dz_{1}};\mathfrak{h}_{\rmin})\to L^{w}_{z_{1}}L^{2}_{x}$\,, where
$L^{w}_{z_{1}}L^{2}_{x}=L^{w}(Z_{1},\mathbf{dz_{1}};L^{2}(X,\mathbf{dx}))$ here, is defined by
pointwise multiplication $(U(t)f)(z_{1})=U(t,z_{1})f(z_{1})$\,.\\
The set of sharp $\sigma$-admissible space-time exponents is given by 
$$
q,r\geq 2\quad \frac{1}{q}+\frac{\sigma}{r}= \frac{\sigma}{2}\,,
$$
and the dual exponents are denoted by
 $q',r'$\,, $\frac{1}{q}+\frac{1}{q'}=1$\,,
 $\frac{1}{r}+\frac{1}{r'}=1$ with $1\leq q',r'\leq 2$\,,
$\frac{1}{q'}+\frac{\sigma}{r'}=\frac{\sigma+2}{2}$\,.\\
We will consider cases where $\sigma>1$ and the endpoint Strichartz
estimates for 
$P=(2,\frac{2\sigma}{\sigma-1})$ holds true.
The results for sharp $\sigma$-admissible pairs $(q,r)$ and
$(\tilde{q},\tilde{r})$ are:
\begin{itemize}
\item the homogeneous estimate
  \begin{equation}
    \label{eq:homStri}
\|U(t)f\|_{L^{w}_{z_{1}}L^{q}_{t}L^{r}_{x}}\lesssim \|f\|_{L^{w}(Z_{1},\mathbf{dz_{1}};\mathfrak{h}_{\rmin})}\,;
  \end{equation}
\item the inhomogeneous estimate
  \begin{equation}
    \label{eq:inhomStri}
    \|\int U(s)^{*}F(s)~ds\|_{L^{w}(Z_{1},dz_{1};\mathfrak{h}_{\rmin})}\lesssim
    \|F\|_{L^{w}_{z_{1}}L^{\tilde{q}'}_{t}L^{\tilde{r}'}_{x}}\,;
  \end{equation}
\item the retarded estimate
  \begin{equation}
    \label{eq:retardStri}
    \|\int_{s<t} U(t)U(s)^{*}F(s)~ds\|_{L^{w}_{z_{1}}L^{q}_{t}L^{r}_{x}}\lesssim \|F\|_{L^{w}_{z_{1}}L^{\tilde{q}'}_{t}L^{\tilde{r}'}_{x}}\,,
  \end{equation}
where $s<t$ can be replaced by $s>t$\,.
\end{itemize}
Keel and Tao's results are written in \cite{KeTa} with
$Z_{1}=\left\{z_{0}\right\}$ and $\mathbf{dz_{1}}=\delta_{z_{0}}$\,, but the uniform
inequalities with respect to $z_{1}\in Z_{1}$ can be integrated afterwards for
data in $L^{w}_{z_{1}}$\,.\\
By requiring $\sigma>1$\,,  the endpoint estimate
allows to  take $q=\tilde{q}'=2$ with the endpoint
exponents $r_{\sigma}=\frac{2\sigma}{\sigma-1}$ and
$r_{\sigma}'=\frac{2\sigma}{\sigma+1}$\,. This  is a very convenient
framework for 
fixed point and bootstrap method in our linear setting.

Below are the typical 
inequalities which will be used. 
In our applications like in Subsection~\ref{sec:gausFp}, the vaccuum
sector plays a separate role and it is convenient to use the general
Definition~\ref{de:LpxLqy} for $L^{w}_{z}L^{q}_{x}$
\begin{eqnarray*}
  && \mathcal{N}=\left\{0,1\right\}\quad,\quad Z=Z_{0}\sqcup
     Z_{1}\\
\text{and}&& \mathcal{X}_{0}=\left\{0\right\}\quad,\quad 
\mathcal{X}_{1}=X\quad,\quad \mathbf{dx_{0}}=\delta_{0}\quad,\quad \mathbf{dx_{1}}=\mathbf{dx}\,. 
\end{eqnarray*}
In particular the spaces $L^{2}_{z}L^{q}_{x}$ for $1\leq q\leq
\infty$ equal
\begin{equation}
  \label{eq:L2tzL2z}
L^{2}_{z}L^{q}_{x}=L^{2}(Z_{0},\mathbf{dz_{0}})\oplus
L^{2}(Z_{1},\mathbf{dz_{1}};L^{2}(X,\mathbf{dx}))=\underbrace{L^{2}_{z_{0}}}_{\text{vacuum}}\oplus L^{2}_{z_{1}}L^{2}_{x}\,.
\end{equation}
At this level the action of the dynamics $U(t)U(s)^{*}$ is considered
only on the $L^{w}_{z_{1}}L^{q}_{x}$ component\,.
\begin{proposition}
\label{pr:StriDT} 
Consider $L^{2}_{z}L^{q}_{x}=L^{2}_{z_{0}}\oplus
L^{2}_{z_{1}}L^{2}_{x}$ like in \eqref{eq:L2tzL2z} and according to
Definition~\ref{de:LpxLqy}.\\
Assume that there is a dense Banach space $D$ in $L^{2}_{z_{1},x}=L^{2}_{z_{1}}L^{2}_{x}$ such that
$D\subset L^{2}_{z_{1}}L^{r_{\sigma}}_{x}$ and
$U(t)U(s)^{*}u\in L^{2}_{z_{1}}L^{r_{\sigma}}_{x}$ is measurable
with respect to $t,s\in \rz$ for all $u\in D$ with the uniform estimate
$\|U(t)U(s)^{*}u\|_{L^{2}_{z_{1}}L^{r_{\sigma}}_{x}}\lesssim \|u\|_{D}$ for almost
all $t,s\in\rz$\,.\\
Assume that the bounded operator $B^{*}_{t,s}:L^{2}_{z}L^{2}_{x}\to
L^{2}_{z_{1}}L^{r_{\sigma}'}_{x}$ and its adjoint
$B_{t,s}:L^{2}_{z_{1}}L^{r_{\sigma}}_{x}\to L^{2}_{z}L^{2}_{x}$
are strongly  
measurable with respect to $(t,s)\in [0,T]\times[0,T]$ with the assumption
\begin{align}
\label{eq:Bts*hyp}
&&
\sup_{t\in[0,T]}\int_{0}^{T}\|B^{*}_{t,s}\|^{2}~ds&<+\infty
\:,&
\|B^{*}_{t,s}\|&=\|B_{t,s}^{*}\|_{L^{2}_{z_{1}}L^{r'_{\sigma}}_{x}\leftarrow
L^{2}_{z}L^{2}_{x}}\:,
\\
\label{eq:Btshyp}
\text{resp.}
&&
\sup_{s\in [0,T]}\int_{0}^{T}\|B_{t,s}\|^{2}dt&<+\infty\:,& \|B_{t,s}\|&=\|B_{t,s}\|_{L^{2}_{z}L^{2}_{x}\leftarrow L^{2}_{z_{1}}L^{r_{\sigma}}_{x}}\:.
\end{align}
 The operator $A_{T}^{*}$  (resp. $A_{T}$) defined by
\begin{align}
 \label{eq:defA*T}
  &&
[A_{T}^{*}f](t)&=1_{Z_{1}}(z)\int_{0}^{T}U(t)U(s)^{*}B^{*}_{t,s}f(s)~ds\,,\\\
 \label{eq:defAT}
\text{resp.}
&& [A_{T}f](t)&=\int_{0}^{T}B_{t,s}U(t)U(s)^{*}1_{Z_{1}}(z)f(s)~ds\,,
\end{align}
acts continuously on $L^{\infty}([0,T];L^{2}_{z}L^{2}_{x})$  (resp. extends as
a continuous operator on $L^{1}([0,T];L^{2}_{z}L^{2}_{x})$) with
\begin{align}
&&
\Ran A_{T}^{*}\subset
   L^{\infty}([0,T];L^{2}_{z_{1},x})\,,\quad
\Ker(A_{T})\supset L^{1}([0,T];L^{2}_{z_{0}})\,,
 \end{align}
 \begin{align}
  &&
\|(A_{T}^{*})^{n}\|_{\mathcal{L}(L^{\infty}([0,T];L^{2}_{z}L^{2}_{x}))}&\lesssim
\left(\sup_{t_{n+1}\in[0,T]}\int_{[0,T]^{n}}\|B^{*}_{t_{n+1},t_{n}}\|^{2}\ldots
\|B^{*}_{t_{2},t_{1}}\|^{2}~dt_{1}\cdots dt_{n}\right)^{1/2} 
, \label{eq:A*stT} \\
 \label{eq:AstT}
\text{resp.}&&
\|(A_{T})^{n}\|_{\mathcal{L}(L^{1}([0,T];L^{2}_{z}L^{2}_{x}))}
&\lesssim
\left(\sup_{t_{0}\in[0,T]}\int_{[0,T]^{n}}\|B_{t_{n},t_{n-1}}\|^{2}\ldots
\|B_{t_{1},t_{0}}\|^{2}~dt_{1}\cdots dt_{n}\right)^{1/2},
\end{align}
for all non zero $n\in \nz$\,.\\
When
$B^{\sharp}_{t,s}=B^{\sharp}_{t,s}1_{s<t}$ or
$B^{\sharp}_{t,s}=B^{\sharp}_{t,s}1_{s>t}$ ($B^{\sharp}=B^{*}$ resp. 
$B^{\sharp}=B$), the domain of integration  $[0,T]^{n}$ can be
replaced by the corresponding n-dimensional simplex $0< t_{1}<\ldots<t_{n}<T$ or $T>t_{1}\ldots>t_{n}>0$\,.
\end{proposition}
\begin{remark}
\label{re:DL1}
  The dense subspace $D$ is introduced in order to get a dense domain
  of $L^{1}([0,T];L^{2}_{z_{1},x})$ where $A_{T}$ is well defined by its
  integral formula. The
  extension to the whole space $L^{1}([0,T];L^{2}_{z_{1},x})$ is proved by
  using the fact that $L^{\infty}([0,T];L^{2}_{z_{1},x})$ is the dual of
  $L^{1}([0,T];L^{2}_{z_{1},x})$ and it cannot be done in the other way.\\
Examples where the dense subset $D$ is easy to construct are when
$L^{2}(X,\mathbf{dx};\cz)=L^{2}(\rz^{d},dx;\cz)$  and
$U(t)U(s)^{*}: H^{\mu}(\rz^{d};\cz)\to H^{\mu}(\rz^{d};\cz)$
are measurable and uniformly bounded
 w.r.t.~$t,s\in\rz$ for some
$\mu>d/2$\,. In this simple case, the set $D$  can be $L^{2}(Z_{1},\mathbf{dz_{1}};H^{\mu}(\rz^{d};\cz))$
with $\mu>\frac{d}{2}$\,.
\end{remark}
\begin{proof}
Let us start with $A_{T}^{*}$\,. When $f\in
L^{\infty}([0,T],dt;L^{2}_{z}L^{2}_{x})$ the function $1_{[0,T]}f$ belongs to
$L^{2}_{z}L^{2}_{t,x}$ and, for almost
all $t_{0}\in [0,T]$\,, the function
$(z,s,x)\mapsto B^{*}_{t_{0},s}1_{[0,T]}(s)f(s)$ belongs to 
$L^{2}_{z_{1}}L^{2}_{s}L^{r'_{\sigma}}_{x}$\,.
The inhomogeneous
endpoint Strichartz estimate implies for
almost all $t_{0}\in [0,T]$
\begin{equation}
  \label{eq:ineqLinf}
\|A_{T}^{*}f(t_{0})\|^{2}_{L^{2}_{z}L^{2}_{x}}\lesssim
\int_{0}^{T}
\|B^{*}_{t_{0},s}f(s)\|^{2}_{L^2_{z_1}L^{r'_\sigma}_x}~ds
\lesssim \left(\int_{0}^{T}\|B^{*}_{t_{0},s}\|^{2}~ds\right)\|f\|^{2}_{L^{\infty}([0,T];L^{2}_{z}L^{2}_{x})}\,.
\end{equation}
This proves firstly that $A_{T}^{*}$ acts continuously on
$L^{\infty}([0,T];L^{2}_{z}L^{2}_{x})$\,. The property
$\Ran A_{T}^*\subset L^{\infty}([0,T];L^{2}_{z_{1},x})$ comes from
the assumption $B^{*}_{t,s}:L^{2}_{z}L^{2}_{x}\to L^{2}_{z_{1}}L^{r_{\sigma}'}_{x}$ and the redundant
multiplication by $1_{Z_{1}}(z)$ in \eqref{eq:defA*T}\,. 
Secondly iterating
\eqref{eq:ineqLinf} with $(t_{0},s)=(t_{n+1},t_{n})$ leads to
\eqref{eq:A*stT}\,.\\
Consider now $A_{T}f$ when $f=1_{Z_{1}}(z)f\in L^{1}([0,T];L^{2}_{z_{1},x})$\,.
For $f$ in the dense subspace $L^{1}([0,T];D)$ of $L^{1}([0,T];L^{2}_{z_{1},x})$\,,  our
assumptions ensure that $A_{T}f$ belongs to
$L^{\infty}([0,T];L^{2}_{z}L^{2}_{x})\subset L^{1}([0,T];L^{2}_{z}L^{2}_{x})$ with
$$
\|A_{T}f\|_{L^{1}([0,T];L^{2}_{z}L^{2}_{x})}\lesssim C_{T}\|f\|_{L^{1}([0,T];D)}\,.
$$
With 
$$
\int_{0}^{T}\langle v(t)\,,\, A_{T}f(t)\rangle~dt
=\int_{0}^{T}\langle 1_{Z_{1}}(z) \int_{0}^{T}U(s)U^{*}(t)B_{t,s}^{*}v(t)~dt\,,
f(s)\rangle~ds
=\int_{0}^{T}\langle (\tilde{A}_{T}^{*}v)(s)\,,\, f(s)\rangle~ds\,,
$$
where  $B_{t,s}^{*}$ has simply been replaced by $B_{s,t}^{*}$
in $\tilde{A}_{T}^{*}v(t)=1_{Z_{1}}(z)\int_{0}^{T}U(t)U(s)^{*}B_{s,t}^{*}v(s)~ds$\,,
we obtain
$$
\forall v\in
L^{\infty}([0,T];L^{2}_{z}L^{2}_{x})\,,\quad
|\langle v\,,\, A_{T}f\rangle|\lesssim \left(\int_{0}^{T}\|B_{s,t}\|^{2}~ds\right)^{1/2}\|v\|_{L^{\infty}([0,T];L^{2}_{z}L^{2}_{x})}\|f\|_{L^{1}([0,T];L^{2}_{z_{1},x})}\,,
$$
while $L^{\infty}([0,T];L^{2}_{z}L^{2}_{x})=(L^{1}([0,T];L^{2}_{z}L^{2}_{x}))'$\,.\\
This proves that $A_{T}$ extends as a continuous operator 
$L^{1}([0,T];L^{2}_{z_{1},x})\to L^{1}([0,T];L^{2}_{z}L^{2}_{x})$ and the formula contains the extension by
$0$ on $L^{1}([0,T];L^{2}_{z_{0}})$\,, with
$L^{1}([0,T];L^{2}_{z}L^{2}_{x})=L^{1}([0,T];L^{2}_{z_{0}})\oplus L^{1}([0,T];L^{2}_{z_{1},x})$\,.
 Its adjoint is
$\tilde{A}_{T}^{*}:L^{\infty}([0,T];L^{2}_{z}L^{2}_{x})\to
L^{\infty}([0,T];L^{2}_{z}L^{2}_{x})$\,. The estimate \eqref{eq:A*stT} for
$\tilde{A}_{T}^{*}$ with $(\|B^{*}_{t,s}\|,t_{k})$ replaced by $(\|B^{*}_{s,t}\|=\|B_{s,t}\|,t_{n+1-k})$
yields \eqref{eq:AstT}\,.
\end{proof}
Note that when $B^{\sharp}_{t,s}=B^{\sharp}_{t,s}1_{t>s}$ or
$B^{\sharp}_{t,s}=B^{\sharp}_{t,s}1_{t<s}$ with
$\|B^{\sharp}_{t,s}\|\leq \beta$\,, the upper bounds of
\eqref{eq:A*stT} and \eqref{eq:AstT} are below
$$
\left(\frac{(\beta^{2}T)^{n}}{n!}\right)^{1/2}\lesssim
\left(\frac{e\beta^{2}T}{n}\right)^{n/2}\,.
$$
This gives a hint of times scales with respect to $\beta$\,,
e.g. when  $\beta^{2}T\leq C$ here, where iterative methods lead to
convergent series or the associated fixed point methods can be used. We
will use some refined versions of the scaling rule $\beta^{2}T\leq
C$\,. Although the $L^{p}_{t}$ spaces estimates are written with $p=+\infty$
and $p=1$\,, this scaling really relies on the endpoint Strichartz
estimate with $p=2$\,.\\
We complete our general corollaries of endpoint Strichartz estimates
with a result which combines the action of  operators like $B_{t,s}$
and $B^{*}_{t,s}$ in Proposition~\ref{pr:StriDT}.
\begin{proposition}
\label{pr:L2L1Linf}
Let $\mathcal{I},\mathcal{J}$ be at most countable families of
disjoint finite intervals, and set $UI=\sqcup_{I\in \mathcal{I}}I$ and
$UJ=\sqcup_{I\in \mathcal{J}}J$\,. For a given $\varphi_{\infty}\in
L^{\infty}(UJ;L^{2}_{z}L^{2}_{x})$ consider
\begin{eqnarray*}
  &&
\varphi_{1,I}(t)=1_{I}(t)\sum_{J\in \mathcal{J}}\int_{0}^{t}B_{1,IJ}U(t)U(s)^{*}B_{2,IJ}^{*}(s)\varphi_{\infty,J}(s)~ds
\\
\text{with}&&
              \varphi_{\infty,J}(s)=\varphi_{\infty}(s)1_{J}(s)\,,\\
\text{and}&&
\|B_{1,IJ}\|_{L^{2}_{z}L^{2}_{x}\leftarrow
             L^{2}_{z_{1}}L^{r_{\sigma}}_{x}}\leq
             \beta_{1,IJ}\quad,\quad
             \sup_{s\in
             J}\|B_{2,IJ}^{*}(s)\|_{L^{2}_{z_{1}}L^{r'_{\sigma}}_{x}\leftarrow
             L^{2}_{z}L^{2}_{x}}\leq \beta_{2,IJ}\,,
\end{eqnarray*}
where $B_{1,IJ}:L^{2}_{z_{1}}L^{r_{\sigma}}_{x}\to L^{2}_{z}L^{2}_{x}$
does not depend on $(t,s)\in I\times J$ while 
$B_{2,IJ}^{*}(s):L^{2}_{z}L^{2}_{x}\to L^{2}_{z_{1}}L^{r'_{\sigma}}_{x}$
does 
not depend on the time variable $t\in I$ and is strongly measurable
with respect to $s\in J$\,.
Then the function $\varphi_{1}=\sum_{I\in \mathcal{I}}\varphi_{1,I}$ belongs to $L^{1}(UI,dt;L^{2}_{z}L^{2}_{x})$ with
$$
\|\varphi_{1}\|_{L^{1}(UI,dt;L^{2}_{z}L^{2}_{x})}\lesssim
\left[\sum_{\tiny
    \begin{array}[c]{c}
I\in
  \mathcal{I},J\in \mathcal{J}\\
\inf J<\sup I
    \end{array}
  }|I|^{1/2}\beta_{1,IJ}\beta_{2,IJ}|J|^{1/2}\right]\|\varphi_{\infty}\|_{L^{\infty}(UJ,dt;L^{2}_{z}L^{2}_{x})}\,,
$$
as soon as $\left[\sum_{I\in \mathcal{I},J\in \mathcal{J}}1_{]0,+\infty[}(\sup I-\inf J)
|I|^{1/2}\beta_{1,IJ}\beta_{2,IJ}|J|^{1/2}\right]<+\infty$\,.
\end{proposition}
\begin{proof}
 Every term of $\varphi_{1,I}$ can be written
$$
\psi_{IJ}(t)=B_{1,IJ}\int_{0}^{t}U(t)U(s)^{*}\phi_{2,IJ}(s)~ds
$$
where $\phi_{2,IJ}=B_{2,IJ}^{*}(\cdot)\varphi_{\infty,J}(\cdot)\in L^{2}(\rz;L^{2}_{z_{1}}L^{r'_{\sigma}}_{x})$ 
satisfies
\begin{eqnarray*}
  && \phi_{2,IJ}=0\quad\text{if}~\inf J\geq \sup I\,,\\
\text{and}&&
\|\phi_{2,IJ}\|_{L^{2}(\rz,dt;L^{2}_{z_{1}}L^{r'_{\sigma}}_{x})}\leq
|J|^{1/2}\beta_{2,IJ}\|\varphi_{\infty,J}\|_{L^{\infty}(J,dt;L^{2}_{z}L^{2}_{x})}
\leq |J|^{1/2}\beta_{2,IJ}\|\varphi_{\infty}\|_{L^{\infty}(UJ,dt;L^{2}_{z}L^{2}_{x})}\,.
\end{eqnarray*}
The retarded endpoint Strichartz estimate with
$\|B_{1,IJ}\|_{L^{2}_{z}L^{2}_{x}\leftarrow
  L^{2}_{z}L^{r_{\sigma}}_{x}}\leq \beta_{1,IJ}$ implies
$$
\|\psi_{IJ}\|_{L^{2}(I,dt;L^{2}_{z}L^{2}_{x})}\lesssim
1_{]0,+\infty[}(\sup I-\inf J)\beta_{1,IJ}\beta_{2,IJ}|J|^{1/2}\|\varphi_{\infty}\|_{L^{\infty}(UJ,dt;L^{2}_{z}L^{2}_{x})}
$$
and
therefore 
$$
\|\psi_{IJ}\|_{L^{1}(I,dt;L^{2}_{z}L^{2}_{x})}\lesssim \left[1_{]0,+\infty[}(\sup I-\inf J)|I|^{1/2}\beta_{1,IJ}\beta_{2,IJ}|J|^{1/2}\right]\|\varphi_{\infty}\|_{L^{\infty}(UJ,dt;L^{2}_{z}L^{2}_{x})}\,.
$$
The finiteness of $\sum_{I\in \mathcal{I},J\in
 \mathcal{J}}\left[ 1_{]0,+\infty[}(\sup I-\inf J)|I|^{1/2}\beta_{1,IJ}\beta_{2,IJ}|J|^{1/2}\right]$
ensures that $\varphi_{1,I}=\sum_{J\in \mathcal{J}}\psi_{IJ}$ belongs to
$L^{1}(I,dt;L^{2}_{z}L^{2}_{x})$ and finally
\begin{align*}
\|\varphi_{1}\|_{L^{1}(UI,dt;L^{2}_{z}L^{2}_{x})}&=\sum_{I\in
  \mathcal{I}}\|\varphi_{1,I}\|_{L^{1}(I,dt;L^{2}_{z}L^{2}_{x})}
\\
&\lesssim
\sum_{I\in \mathcal{I},J\in
  \mathcal{J}}\left[ 1_{]0,+\infty[}(\sup I-\inf J) 
|I|^{1/2} \beta_{1,IJ}\beta_{2,IJ}|J|^{1/2}\right]
\|\varphi_{\infty}\|_{L^{\infty}(UJ,dt;L^{2}_{z}L^{2}_{x})}\,.
\end{align*}
\end{proof}
\subsection{Fixed point in weighted spaces}
\label{sec:fixweight}
In this section, we apply the
general framework of Strichartz estimates for evolution equations in
the spaces 
$$
F^{2}=L^{2}(Z',\mathbf{dz'};\Gamma
(L^{2}(\rz^{d},dy;\cz)))=L^{2}(Z',\mathbf{dz'};\cz)\oplus L^{2}_{\mathrm{sym}}(\mathcal{R}\times
Z',d\mu\otimes \mathbf{dz'};L^{2}(\rz^{d},dy_{G};\cz))\,.
$$
The measured space of parameters $(Z',\mathbf{dz'})$ will be specified
later and by following the notations of Definition~\ref{de:LpxLqy} and
\eqref{eq:L2tzL2z} for the application of Strichartz estimates, we
write
\begin{eqnarray}
  && Z_{0}=Z'\quad,\quad Z_{1}=\mathcal{R}\times Z'\quad,\quad
     \mathbf{dz_{0}}=\mathbf{dz'}\quad,\quad
     \mathbf{dz_{1}}=\mu\otimes \mathbf{dz'}
\nonumber
\\
&& X_{0}=\left\{0\right\}\quad,\quad X_{1}=\rz^{d}\quad,\quad
   \mathbf{dx_{0}}=\delta_{0}\quad,\quad \mathbf{dx}=dy_{G}\,,
\nonumber\\
&&
F_{2}=L^{2}_{z,\text{sym}}L^{2}_{y_{G}}=L^{2}_{z_{0}}\oplus
   L^{2}_{z_{1},\text{sym}}L^{2}_{y_{G}}=L^{2}_{z_{0}}\oplus L^{2}_{(Y',z'),\text{sym}}L^{2}_{y_{G}}\,,
\label{eq:F2L2L2}
\end{eqnarray}
where the second variable $x\in \mathcal{X}=X_{0}\sqcup X_{1}$ has
been replaced by $y_{G}$ in order to recall its link with the center
of mass on the non vacuum sector.\\
We will use the $L^{p}_{y_{G}}$\,,$1\leq p\leq +\infty$\,, version
$$
L^{2}_{z,\text{sym}}L^{p}_{y_{G}}=L^{2}_{z_{0}}\oplus
L^{2}_{(Y',z'),\text{sym}}L^{p}_{y_{G}}\qquad\text{with}~ z_{1}=(Y',z')\,.
$$
In all the above identities the subscript $_{\text{sym}}$ refers to
the symmetry for the relative variable $Y'\in \mathcal{R}$\,. Because
the symmetry is preserved by all our defined operators, this subscript
will be forgotten when we write estimates.\\
 Only the useful conditions on the ``free
dynamics'' $U(t)$\,, or more precisely $U(t)U(s)^{*}:F^{2}\to F^{2}$
will be specified. Those will be checked for our model later in
Section~\ref{sec:conseqStri}.
The free dynamics or more precisely $U(t)U(s)^{*}: F^{2}\to F^{2}$ is assumed to preserve the number of
particles 
$$[U(t)U(s)^{*},N]=0
$$
 with the following decomposition:
\begin{eqnarray}
\label{eq:UtUs}
  &&U(t)U(s)^{*}=(K_{0}(t,z')\overline{K_{0}(s,z')}\times_{z'})\oplus
     (U_{1}(t,Y',z')U_{1}^{*}(s,Y',z')\times_{(Y',z')})\\
\label{eq:F2UtUs}
\text{in}&&
F^{2}=\underbrace{L^{2}(Z',\mathbf{dz'};\cz)}_{=L^{2}_{z_{0}}\quad \text{(vacuum)}}\oplus 
\underbrace{L^{2}_{\mathrm{sym}}(\mathcal{R}\times
Z',d\mu\otimes \mathbf{dz'};L^{2}(\rz^{d},dy_{G};\cz))}_{=L^{2}_{(Y',z'),\text{sym}}L^{2}_{y_{G}}}\,,
\end{eqnarray}
where $\times_{z'}$ or $\times_{(Y',z')}$ stands for the pointwise
multiplication. So $U_{1}(t,Y',z')U_{1}^{*}(s,Y',z')$ is a one
particle operator acting in the $y_{G}$-variable, parametrized by
$z_{1}=(Y',z')$ and we add the
following conditions which make 
 the results of
Subsection~\ref{sec:endpoint} relevant: 
\begin{itemize}
\item The measured space $(X_{1},\mathbf{dx_{1}})$ is nothing but
  $(\rz^{d},dy_{G})$ in the center of mass variable and the
  $z_{1}=(Y',z')$-dependent one particle operators
  $U_{1}(t,z_{1}):\mathfrak{h}_{\rmin}\to L^{2}(\rz^{d},dy_{G};\cz)$ and its
  adjoint are assumed to satisfy the estimate
  \eqref{eq:hypStri1}\eqref{eq:hypStri2} with $\sigma>1$\,. Remember
  $r'_{\sigma}=\frac{2\sigma}{\sigma+1}$ and $r_{\sigma}=\frac{2\sigma}{\sigma-1}$\,.
\item The additional assumption
of Proposition~\ref{pr:StriDT} concerned with the dense subset $D$ is
also assumed for $U_{1}(t,z_{1})$\,.
\item The vacuum component $K_{0}$ belongs to $L^{\infty}(\rz\times
  Z', dt\otimes \mathbf{dz'};\cz)$\,. 
\end{itemize}

The interaction terms will be 
$$B^{*}_{t,s}=c_{1}(t,s)e^{\alpha(t,s)
  N}\sqrt{h}a_{G}^{*}(V_{1})e^{-\alpha'(t,s)
  N}\quad\text{and}\quad B_{t,s}=c_{2}(t,s)\sqrt{h}e^{\alpha(t,s)N}a_{G}(V_{2})e^{-\alpha'(t,s)N}
$$ with $V_{1},V_{2}\in
L^{r'_{\sigma}}(\rz^{d},dy;\cz)$ (complex valued
$V$ are allowed here) and where
$c_{1}$\,, $c_{2}$\,, $\alpha$ and $\alpha'$ are real  measurable
functions of $(t,s)\in [0,T]^{2}$ with $\alpha-\alpha'<0$\,. Those
will be specified further and we shall check the estimates
\eqref{eq:Bts*hyp}\eqref{eq:Btshyp}.
Because $Z_{0}=Z'$ corresponds to the vacuum sector, $N=0$\,, on which
$a_{G}(V)$ vanishes while the range of $a_{G}(V)^{*}$ lies in the non
vacuum sector $N\geq 1$\,, the range $B^{*}_{t,s}$ lies naturally in
$L^{2}_{z_{1}}L^{r'_{\sigma}}_{y_{G}}$, $z_{1}=(Y',z')$\,, once the proper estimates are checked while
it adjoints $B_{t,s}$ sends $L^{2}_{z_{1}}L^{r_{\sigma}}_{y_{G}}$ into
$L^{2}_{z}L^{2}_{y_{G}}$ and is naturally extended by $0$ on
the vacuum sector $L^{2}_{z_{0}}$\,.\\
We will consider the following system
\begin{alignat}{2}
  \label{eq:dynuinfty}
  u^{h}_\infty(t)&=
-i\int_{0}^{t}U(t)U^{*}(s) \left(\sqrt{h}a_{G}^{*}(V_{1})u^{h}_\infty(s)+\sqrt{h} u^{h}_2(s)+u^h_1(s)\right)\,ds &+f^{h}_\infty(t) \\
u_2^{h}(t)&= 
-i\int_{0}^{t} a_{G}(V_{2})U(t)U(s)^{*} \sqrt{h} u_2^{h}(s)~ds 
&+f_2^{h}(t) 
\label{eq:dynu2} \,,\\
u_1^{h}(t)&=
 -i\int_{0}^{t} a_{G}(V_{2})U(t)U(s)^{*} \left( h a_{G}^{*}(V_{1})
 u_\infty^{h}(s) + \sqrt{h} u_1^{h}(s) \right) ~ds 
 &+f_1^{h}(t)\,. 
\label{eq:dynu1}
\end{alignat}
written shortly as
\begin{equation}
\forall q\in\{\infty,2,1\},\quad u_q^h = \sum_{p\in\{\infty,2,1\} }  L_{qp}(u^h_p) \, + \, f_q^h
\label{eq:dynuqp}
\end{equation}
or
\begin{equation}
            \begin{pmatrix}
              u_\infty^{h}\\u_2^{h}\\u_1^h
            \end{pmatrix}
=L
            \begin{pmatrix}
              u_\infty^{h}\\u_2^{h}\\u_1^h
            \end{pmatrix}
+
            \begin{pmatrix}
              f_\infty^{h}\\f_2^{h}\\f_1^h
            \end{pmatrix}
\,,\quad L=
  \begin{pmatrix}
    L_{\infty\,\infty}&L_{\infty2}&L_{\infty1}\\
    0&L_{22}&0\\
    L_{1\infty}&0&L_{11}
  \end{pmatrix}
\,.\label{eq:defL}
\end{equation}
This system will be studied in spaces with the number weight $e^{\alpha N}$ and
we will use the following functional spaces.
\begin{definition}
\label{de:E0Talpha} For $T>0$\,, $h\in
]0,h_{0}[$\,, $I_{T}^{h}$ denotes the interval
$I_{T}^{h}=]-T/h,T/h[$\,.\\
Fix $\alpha_{0},\alpha_{1}\in\rz$\,, $\alpha_{0}<\alpha_{1}$ and set
$M_{\alpha01}=\frac{\max(e^{\alpha_{1}},e^{-\alpha_{0}})}{2}\geq 1/2$\,.\\
Assume $V_{1},V_{2}\in L^{r'_{\sigma}}(\rz^{d},dy;\cz)$ with
$\max(\|V_{1}\|_{L^{r'_{\sigma}}},\|V_{2}\|_{L^{r'_{\sigma}}})< C_{V}$\,.\\
For a parameter $\gamma>0$ and $\alpha\in [\alpha_{0},\alpha_{1}[$ set
$$
T_{\alpha}=\gamma(\alpha_{1}-\alpha)\,.
$$
The space $\mathcal{E}^{h}_{\alpha_{0},\alpha_{1},\gamma}$ is the set of
$(e^{-\alpha_{0}N}L^{2}_{z,\rmsym}L^{2}_{y_{G}})^{3}$-valued
measurable 
functions $I_{T_{\alpha_{0}}}^{h}\ni t\mapsto
\begin{pmatrix}
  u_\infty(t)\\u_2(t)\\u_1(t)
\end{pmatrix}
$ such that for all $\alpha$ in $[\alpha_{0},\alpha_{1}[$\,,
\begin{align*}
 |t|^{-1/2}u_\infty & \in
L^{\infty}(I^{h}_{T_{\alpha}},dt;e^{-\alpha N}L^{2}_{z}L^{2}_{y_{G}})\;,\\
u_2 & \in L^{2}_{\rmloc}(I^{h}_{T_{\alpha}},dt;e^{-\alpha N}L^{2}_{z}L^{2}_{y_{G}})\;,\\
 |t|^{-1/2} u_1 & \in L^{1}_{\rmloc}(I^{h}_{T_{\alpha}},dt;e^{-\alpha N}L^{2}_{z}L^{2}_{y_{G}})\;.
\end{align*}
and $M(u_\infty,u_2,u_1)<+\infty$ with 
\begin{align}
  \label{eq:defMinftyu}
  M(u_\infty,u_2,u_1)&=M_\infty(u_\infty)+M_2(u_2)+M_1(u_1)\,,\\
\label{eq:defMu0}
M_\infty(u_\infty)&=\sup_{\alpha_{0}\leq \alpha<\alpha_{1}}
\left\|\left(\frac{T_{\alpha}-|ht|}{|ht|}\right)^{1/2}e^{\alpha
  N}u_\infty \right\|_{L^{\infty}(I_{T_{\alpha}}^{h};L^{2}_{z}L^{2}_{y_{G}})}\,,\\
\label{eq:defM2u}
M_2(u_2)&=
\frac{1}{M_{\alpha01}C_{V}\gamma^{1/2}}\sup_{\substack{\alpha_{0}\leq\alpha<\alpha_{1}\\
\tau\in  [0,T_{\alpha}[ }}\sqrt{T_{\alpha}-\tau}
\left\| e^{\alpha N}u_2 \right\|_{L^{2}(I^{h}_{\tau};L^{2}_{z}L^{2}_{y_{G}})}\,,\\
\label{eq:defM1u}
M_1(u_1)&=
\frac{1}{M_{\alpha01}C_{V}\gamma^{1/2}}\sup_{\substack{\alpha_{0}\leq\alpha<\alpha_{1}\\
\tau\in  [0,T_{\alpha}[ }}\sqrt{T_{\alpha}-\tau}
\left\| \frac{e^{\alpha N}u_1}{\sqrt{|ht|}} \right\|_{L^{1}(I^{h}_{\tau};L^{2}_{z}L^{2}_{y_{G}})}\,.
\end{align}
\end{definition}
Endowed with the norm $M(u_\infty,u_2,u_1)$\,,
$\mathcal{E}^{h}_{\alpha_{0},\alpha_{1},\gamma}$ is a Banach space for all
$h\in ]0,h_{0}[$\,. The $\alpha$-dependent time domain
$I^{h}_{T_{\alpha}}$ where weighted $L^{\infty}_{t}$, $L^{2}_{t}$ and $L^{1}_{t}$ norms are
evaluated is illustrated in Figure~\ref{fig:Tad}.
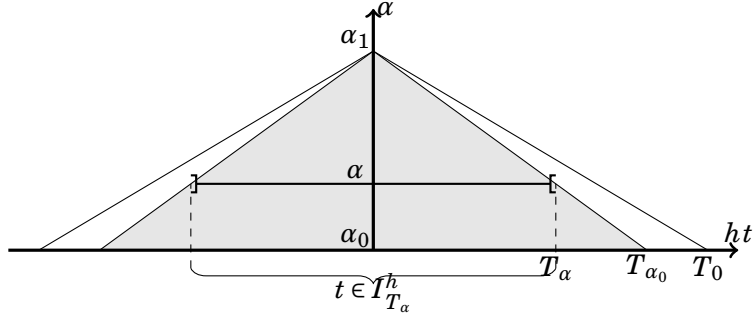
\begin{figure}[h!]
\centering
\begin{tikzpicture}[scale=0.8]
\fill[gray!20] (-4.5,0)--(4.5,0)--(0,3.3) -- cycle;
\draw[very thick,->] (0,-0)--(0,4);
\draw (0,3.3)--(5.5,0);
\draw(0,3.3)--(-5.5,0);
\draw (0,3.3)--(4.5,0);
\draw(0,3.3)--(-4.5,0);
\draw[very thick, ->] (-6,0)--(6,0);
\draw[thick,{]-[}] (-3,1.1)--(3,1.1);
\draw[dashed] (3,-0.3)--(3,1.1);
\draw[dashed] (-3,-0.3)--(-3,1.1);
\node[](T) at (5.5,-0.3) {$T_{0}$};
\node[](Td) at (4.5,-0.3){$T_{\alpha_{0}}$};
\node[](Tad) at (3,-0.3){$T_{\alpha}$};
\node[](a0) at (-0.3,0.2){$\alpha_{0}$};
\node[](a1) at (-0.3,3.5){$\alpha_{1}$};
\node[](aa) at (0.2,4){$\alpha$};
\node[](a) at (-0.3,1.3){$\alpha$};
\node[](t) at (6,0.3){$ht$};
\node[](It) at (0,-0.7) {$t\in I^{h}_{T_{\alpha}}$};
\draw [decorate,decoration={brace,amplitude=0.2cm,mirror}](-3,-0.3)--(3,-0.3);
\end{tikzpicture}
\caption{\label{fig:Tad} The time interval $I_{T_{\alpha}}^{h}=\left]-\frac{\gamma(\alpha_{1}-\alpha)}{h},\frac{\gamma(\alpha_{1}-\alpha)}{h}\right[$ according to
 $\alpha$\,.}
\end{figure}
\\
The constants $C_{V}>0$ and
$M_{\alpha01}=\max(e^{\alpha_{1}},e^{-\alpha_{0}})/2\geq 1/2$ were 
chosen so that Proposition~\ref{pr:expaLp} applied with
$q'=r'_{\sigma}$ and $p'=2$\,, gives
\begin{eqnarray*}
  &&
\|e^{\alpha
  N}a_{G}^{*}(V)e^{-\alpha'N}\varphi\|_{L^{2}_{z_{1}}L^{r'_{\sigma}}_{y_{G}}}\leq
\frac{C_{V}e^{\alpha'}}{2\sqrt{\alpha'-\alpha}}
\|\varphi\|_{L^{2}_{z}L^{2}_{y_{G}}}\leq \frac{M_{\alpha01}C_{V}}{\sqrt{\alpha'-\alpha}}\|\varphi\|_{L^{2}_{z}L^{2}_{y_{G}}}
\\
&&
\|e^{\alpha
  N}a_{G}(V)e^{-\alpha'N}\varphi\|_{L^{2}_{z}L^{2}_{y_{G}}}\leq
\frac{C_{V}e^{-\alpha}}{2\sqrt{\alpha'-\alpha}}\|\varphi\|_{L^{2}_{z_{1}}L^{r_{\sigma}}_{y_{G}}}\leq
   \frac{M_{\alpha01}C_{V}}{\sqrt{\alpha'-\alpha}}\|\varphi\|_{L^{2}_{z}L^{r_{\sigma}}_{y_{G}}}\,,
\end{eqnarray*}
for all $\alpha,\alpha'\in [\alpha_{0},\alpha_{1}[$\,,
$\alpha<\alpha'$\,.\\
Finally the normalization of~\eqref{eq:defM2u} and~\eqref{eq:defM1u}
was chosen in order to make the
contraction statement simple.
\begin{proposition}
\label{pr:contracStri}
Assume that the free dynamics
 $U_{1}(t,z_{1}):\mathfrak{h}_{in}\to
L^{2}(\rz^{d},dy_{G};\cz)$ satisfies
\eqref{eq:hypStri1}\eqref{eq:hypStri2} (uniformly w.r.t.~$z\in Z$)  with $\sigma>1$ and the
additional existence of the dense subset~$D$ assumed in Proposition~\ref{pr:StriDT}.\\
Let $h_{0}>0$\,, $\alpha_{0},\alpha_{1}\in\rz$\,, $\alpha_{0}<\alpha_{1}$ and
$V_{1},V_{2}\in L^{r'_{\sigma}}(\rz^{d},dy;\cz)$ be
fixed. The positive constants $M_{\alpha01},C_{V}$\,, the space
$\mathcal{E}_{\alpha_{0},\alpha_{1},\gamma}^{h}$ and its norm $M$ are the
ones of Definition~\ref{de:E0Talpha}. 
By choosing the parameter $\gamma>0$ small enough the linear
operator $L$ given by \eqref{eq:defL} is a contraction of the Banach
space $(\mathcal{E}_{\alpha_{0},\alpha_{1},\gamma}^{h},M)$ for all $h\in
]0,h_{0}[$ 
and the system
\eqref{eq:defL}, explicitely written
\eqref{eq:dynuinfty}\eqref{eq:dynu2}\eqref{eq:dynu1}, admits a unique solution for any
$(f^{h}_{\infty},f^{h}_{2},f^{h}_{1})\in \mathcal{E}_{\alpha_{0},\alpha_{1},\gamma}^{h}$\,.\\
More precisely there exists a constant $C_{d,U}$ determined by the
dimension $d$ and the free dynamics $U$\,, given by the pair $K_{0}$
and $U_{1}$\,,
such that 
$$
\forall h\in ]0,h_{0}[\,,\quad
\|L\|_{\mathcal{L}(\mathcal{E}_{\alpha_{0},\alpha_{1},\gamma}^{h})}\leq C_{d,U}M_{\alpha01}C_{V}\gamma^{1/2}\,.
$$
Taking e.g. $\gamma=\frac{1}{2C_{d,U}^{2}M_{\alpha01}^{2}C_{V}^{2}}$
ensures
$\|L\|_{\mathcal{L}(\mathcal{E}^{h}_{\alpha_{0,\alpha_{1},\gamma}})}\leq
\frac{1}{2}$ so that the solution to \eqref{eq:defL} satisfies
$$
M(u_{\infty}^{h},u_{2}^{h}, u_{1}^{h})\leq 2M(f_{\infty}^{h},f_{2}^{h},f_{1}^{h})\,.
$$
\end{proposition}
\begin{proof}
 The non-vanishing entries of $L
 \begin{pmatrix}
   u_{\infty}\\u_{2}\\u_{1}
 \end{pmatrix}
$, namely
$$
L_{\infty\,\infty}(u_{\infty})\quad,\quad L_{\infty 2}(u_{2})\quad,\quad 
L_{\infty 1}(u_{1})\quad,\quad L_{22}(u_{2})\quad,\quad
L_{11}(u_{1})\quad\text{and}\quad
L_{1\infty}(u_{\infty})
$$
 will be
 considered separately in this order of increasing difficulty. Additionally the symmetry $t\mapsto -t$ allows  us to
 restrict the analysis to
 $t\geq 0$\,, that is $t\in [0,\frac{T_{\alpha}}{h}[$ for $\alpha\in
 [\alpha_{0},\alpha_{1}[$\,. Accordingly $I_{T}^{h}$ is, in this
 proof,  the restricted
 interval $[0,\frac{T}{h}[$\,.\\
We use like in Section~\ref{sec:StriCM}
 the symbol $\lesssim$ for
 inequalities with constants which depend only on the dimension $d$
 and the free dynamics $U$\,.\\
\noindent$\mathbf{L_{\infty\,\infty}(u_\infty)}:$ For this term and  up to the square root and the parameter $h\in
]0,h_{0}[$\,, we follow exactly the method of \cite{Nir} for Cauchy-Kowalevski theorem.
 Write for $t\in ]0,T_{\alpha}/h[$\,, $ht\in
   ]0,T_{\alpha}[$\,, $\alpha<\alpha_{1}-\frac{ht}{\gamma}$\,,
and
$$
\left(\frac{T_{\alpha}-ht}{ht}\right)^{1/2}e^{\alpha
  N}L_{\infty\,\infty}(u_\infty)(t)=-i\int_{0}^{T_{\alpha}/h}U(t)U(s)^{*}B^{*}_{t,s}
\left(\frac{T_{\alpha_{s}}-hs}{hs}\right)^{1/2}e^{\alpha_{s}N}u_\infty(s)~ds
$$ 
with 
\begin{equation}
     B^{*}_{t,s}=1_{s<t}\left(\frac{T_{\alpha}-ht}{ht}\right)^{1/2}
     e^{\alpha
     N}\sqrt{h}a_{G}^{*}(V)e^{-\alpha_{s}N}\left(\frac{hs}{T_{\alpha_{s}}-hs}\right)^{1/2}\,,\label{eq:defB*(t,s)}
\end{equation}
and $\alpha<\alpha_{s}<\alpha_{1}-\frac{hs}{\gamma}$.
Hence $hs<T_{\alpha_{s}}$ and
\begin{equation}
  \label{eq:borneIu}
\left(\frac{T_{\alpha_{s}}-hs}{hs}\right)^{1/2}
\|e^{\alpha_{s}N}u_\infty(s)\|_{L^{2}_{z}L^{2}y_{G}}
\leq M_\infty(u_{\infty})
\end{equation}
while $\alpha<\alpha_{s}$ implies  that
$\|B_{t,s}^{*}\|=\|B_{t,s}^{*}
\|_{L^{2}_{z}L^{r_{\sigma}'}_{y_{G}}\leftarrow
  L^{2}_{z}L^{2}_{y_{G}}}$ satisfies
$$
\|B_{t,s}^{*}\|^{2}\leq h1_{s<t}
\frac{M_{\alpha01}^{2}C_{V}^{2}}{(\alpha_{s}-\alpha)}\frac{(T_{\alpha}-ht)(hs)}{ht(T_{\alpha_{s}}-hs)}
=h 1_{s'<t'}\frac{M_{\alpha 01}^{2}C_{V}^{2}}{(\alpha_{s'/h}-\alpha)}\frac{(T_{\alpha}-t')s'}{t'(T_{\alpha_{s'/h}}-s')}\,,
$$
by setting $s'=hs$\,, $t'=ht$\,. By choosing
$$
\alpha_{s}=\frac{\alpha_{1}+\alpha-hs/\gamma}{2}=\frac{\alpha_{1}+\alpha-s'/\gamma}{2}\,,
$$
we obtain
\begin{eqnarray*}
  &&\gamma(\alpha_{s}-\alpha)=\frac{\gamma(\alpha_{1}-\alpha)-s'}{2}=\frac{T_{\alpha}-s'}{2} \,,\\
&&T_{\alpha_{s'/h}}=\gamma(\alpha_{1}-\alpha_{s'/h})=\frac{\gamma(\alpha_{1}-\alpha)+s'}{2}\quad,\quad
   T_{\alpha_{s'/h}}-s'=\frac{T_{\alpha}-s'}{2}\,,
\end{eqnarray*}
and
$$
\frac{(T_{\alpha}-t')s'}{(\alpha_{s'/h}-\alpha)t'(T_{\alpha_{s'/h}}-s')}=4\gamma\frac{(T_{\alpha}-t)s'}{t'(T_{\alpha}-s')^{2}} \,.
$$
This yields
\begin{equation}
  \label{eq:borneIB}
\int_{0}^{T_{\alpha}/h}\|B^{*}_{t,s}\|^{2}~ds\leq 4\gamma M_{\alpha01}^2 C_{V}^{2}
\frac{T_{\alpha}-t'}{t'}\int_{0}^{t'}\frac{s'}{(T_{\alpha}-s')^{2}}~ds'\leq
4\gamma M_{\alpha01}^{2}C_{V}^{2}\,.
\end{equation}
The inequalities \eqref{eq:borneIu} and \eqref{eq:borneIB} combined
with the
inequality \eqref{eq:A*stT} with $n=1$ of Proposition~\ref{pr:StriDT}
imply
\begin{equation}
  \label{eq:estimI}
\left\|
\left(\frac{T_{\alpha}-ht}{ht}\right)^{1/2}e^{\alpha N}L_{\infty\,\infty}(u_{\infty})\right\|_{L^{\infty}([0,T_{\alpha}/h];L^{2}_{z}L^{2}_{y_{G}})}\lesssim 2\gamma^{1/2}M_{\alpha01}C_{V}M_\infty(u_\infty)\,.
\end{equation}
\noindent$\mathbf{L_{\infty2}(u_2)}:$ 
The Cauchy-Schwarz inequality applied to
$$
\left(\frac{T_{\alpha}-ht}{ht}\right)^{1/2}e^{\alpha
  N}L_{\infty2}(u_2)(t)=-i\sqrt{T_{\alpha}-ht}
\frac{1}{\sqrt{t}}
\int_{0}^{t}
U(t)U(s)^{*}e^{\alpha
N}u_2(s)
~ds\,,
$$
imply
\begin{align*}
\left\|\left(\frac{T_{\alpha}-ht}{ht}\right)^{1/2}e^{\alpha N}L_{\infty2}(u_2)(t)
\right\|_{L^{2}_{z}L^{2}_{y_{G}}}
&\leq
\sqrt{T_{\alpha}-ht} \frac{1}{\sqrt{t}}\|e^{\alpha
  N} u_2(s) \|_{L^{2}_{z}L^{1}([0,t];L^{2}_{y_{G}})}
\\
&\leq
\sqrt{T_{\alpha}-ht}
\| e^{\alpha N} u_2(s) \|_{L^{2}([0,t];L^{2}_{z}L^{2}_{y_{G}})}\\
&\leq
\sup_{\tau \in ]0,T_{\alpha}[}
\sqrt{T_{\alpha}-\tau}
\| e^{\alpha N} u_2(s) \|_{L^{2}([0,\tau/h];L^{2}_{z}L^{2}_{y_{G}})}\,.
\end{align*}
Taking the supremum over $\alpha\in[\alpha_0,\alpha_1[$ yields
\begin{equation}
  \label{eq:estimII}
M_\infty(L_{\infty2}(u_2)) \lesssim
M_{\alpha01}C_{V}\gamma^{1/2}M_2(u_2)\,. 
\end{equation}
\noindent$\mathbf{L_{\infty1}(u_1)}:$ The expression
$$
\left(\frac{T_{\alpha}-ht}{ht}\right)^{1/2}e^{\alpha
  N}L_{\infty1}(u_1)(t)=-i\sqrt{T_{\alpha}-ht}
\int_{0}^{t}
\frac{\sqrt{hs}}{\sqrt{ht}}
\left[
U(t)U(s)^{*}e^{\alpha
N}\frac{1}{\sqrt{hs}}u_1(s)
\right]~ds\,,
$$
gives
\begin{align*}
\left\|\left(\frac{T_{\alpha}-ht}{ht}\right)^{1/2}e^{\alpha N}L_{\infty1}(u_1)(t)
\right\|_{L^{2}_{z}L^{2}_{y_{G}}}
&\leq
\sqrt{T_{\alpha}-ht}\|e^{\alpha
  N}\frac{u_1(s)}{\sqrt{hs}}\|_{L^{1}([0,t];L^{2}_{z}L^{2}_{y_{G}})}
\\
&\leq 
\sup_{\tau \in ]0,T_{\alpha}[}
\sqrt{T_{\alpha}-\tau}
\|\frac{u_1(s)}{\sqrt{hs}}\|_{L^{1}([0,\frac{\tau}{h}];L^{2}_{z}L^{2}_{y_{G}})}\\
&\leq M_{\alpha01}C_{V}\gamma^{1/2}M_1(u_1)
\end{align*}
and
\begin{equation}
  \label{eq:estimLInfty1}
\|\left(\frac{T_{\alpha}-ht}{ht}\right)^{1/2}e^{\alpha
  N}L_{\infty1}(u_1)\|_{L^{\infty}([0,\frac{T_{\alpha}}{h}];
L^{2}_{z}L^{2}_{y_{G}})}\leq
M_{\alpha01}C_{V}\gamma^{1/2}M_1(u_1)\,. 
\end{equation}

The entries $L_{22}(u_{2})$, $L_{11}(u_{1})$ and finally 
$L_{1\infty}(u_{\infty})$ require some
additional techniques. The proof,
done in several steps for each of them, relies on a dyadic partition 
of the interval
$[0,T_{\alpha}[$ around $T_{\alpha}$\,. 
In the two cases of $L_{22}(u_{2})$ and $L_{11}(u_{1})$\,, the norms $M_{2}(\varphi)$ and $M_{1}(\varphi)$ are
transformed into equivalent norms corresponding to this dyadic
partition, the proof being given
in Lemma~\ref{le:equivnorms2} below. Finally the entry 
$L_{1\infty}(u_{\infty})$ is treated via  dyadic partitions around $T_{\alpha}$ and
$0$ and happens to be a direct application of Proposition~\ref{pr:L2L1Linf}.\\
\medskip

\noindent\textbf{Splitting the interval $\mathbf{[0,T[}$.}
Fix $\alpha\in [\alpha_{0},\alpha_{1}[$ and therefore
$T=T_{\alpha}$\,. The intervals $J^{n}_{T}$ are defined for
$n\in\nz$ by
\begin{eqnarray*}
&& J^{n}_{T}=T+2^{-n}[-T,-T/2[=[(1-2^{-n})T, (1-2^{-n-1})T[\,,\\
  && J^{\leq n_{0}}_{T}=\ccup_{n\leq n_{0}}J^{n}_{T}\, \,
     \text{for}~n_0\in \nz ,
\end{eqnarray*}
so that $[0,T]=\ccup_{n\in\nz}J^{n}_{T}=J^{\leq n_{0}}_{T}\cup
(\ccup_{n>n_{0}}J^{n}_{T})$\,, see Figure \ref{fig:JnT}.
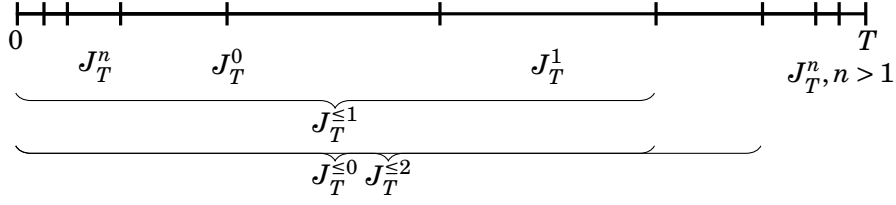
\begin{figure}[h!]
\centering
\begin{tikzpicture}[scale=0.7]
\draw[very thick,|-|] (0,0)--(16,0);
\draw[very thick,-|](0,0)--(8,0);
\draw[very thick,-|](2,0)--(4,0);
\draw[very thick,-|](1,0)--(2,0);
\draw[very thick,|-|](1/2,0)--(1,0);
\draw[very thick,|-](12,0)--(14,0);
\draw[very thick,|-](14,0)--(15,0);
\draw[very thick,|-|](15,0)--(15.5,0);
\node[](J0-) at (4,-1) {$J^{0}_{T}$};
\node[](T) at (16,-0.5) {$T$};
\node[](0) at (0,-0.5) {$0$};
\node[](J0+) at (10,-1) {$J^{1}_{T}$};
\node[](Jn>0) at (15.5,-1.2) {$J^{n}_{T},n>1$};
\node[](Jn<0) at (1.5,-1) {$J^{n}_{T}$};
\draw [decorate,decoration={brace,amplitude=0.2cm,mirror}](0,-1.5)--(12,-1.5);
\node[](J0) at (6,-2.1) {$J^{\leq 1}_{T}$};
\draw [decorate,decoration={brace,amplitude=0.2cm,mirror}](0,-2.5)--(12,-2.5);
\node[](Jleq0) at (6,-3.1) {$J^{\leq 0}_{T}$};
\draw [decorate,decoration={brace,amplitude=0.2cm,mirror}](0,-2.5)--(14,-2.5);
\node[](Jleq1) at (7,-3.1) {$J^{\leq 2}_{T}$};
\end{tikzpicture}
\caption{\label{fig:JnT} The time intervals $J^{n}_{T}$\,, 
 $n\in\nz$\,, with length $\frac{T}{2^{n+1}}$}
\end{figure}
\\
With the exponents
$$
\alpha'_{0}=\frac{\alpha_{1}+6\alpha}{7}\quad\text{and}\quad
\alpha'_{n}=\frac{\alpha_{1}+(2^{n+2}-1)\alpha}{2^{n+2}}~\text{for}~n\geq 1
$$
we note that
\begin{eqnarray*}
  && J^{\leq
     2}_{T_{\alpha'_{0}}}=\frac{7}{8}T_{\alpha'_{0}}=\frac{7}{8}\frac{6}{7}T_{\alpha}=\frac{3}{4}T_{\alpha}=J^{\leq
     1}_{T_{\alpha}}\,,\\
\text{and~for}~n\geq 1&&
T_{\alpha'_{n}}=T_{\alpha}-\frac{1}{2}\frac{T_{\alpha}}{2^{n+1}}=(1-2^{-n-2})T_{\alpha}\,.
\end{eqnarray*}
By taking $\delta_{n}=\frac{T_{\alpha}}{2^{n+2}}$ and
$2\delta_{n}=\frac{T_{\alpha}}{2^{n+1}}$ for $n> 1$\,,
 we obtain in particular
 \begin{eqnarray*}
   &&
      J^{n}_{T_{\alpha}}=[T_{\alpha}-4\delta_{n},T_{\alpha}-2\delta_{n}[=[T_{\alpha'_{n}}-3\delta_{n},T_{\alpha'_{n}}-\delta_{n}[\quad\text{with}~\delta_{n}\leq\frac{T_{\alpha'_{n}}}{12}~(n>1)
\end{eqnarray*}
as summarized in Figure~\ref{fig:triangle}.
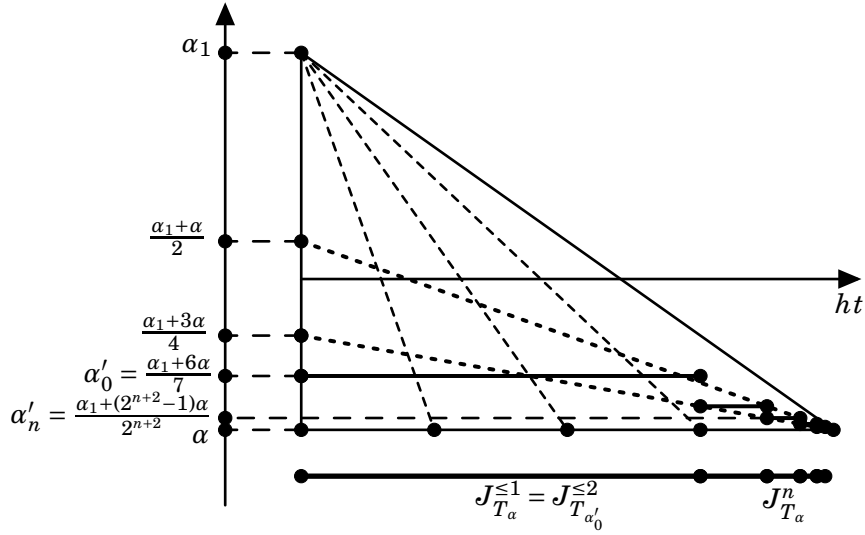
\begin{figure}[h!]
\centering
\begin{tikzpicture}[line cap=round,line join=round,>=triangle 45,x=0.5cm,y=0.5cm]
\clip(-6.1391229567038756,-7.00572963900362) rectangle (19.347979087556062,8.541057458848234);
\draw [line width=1.pt,black] (2.,6.)-- (16.,-4.);
\draw [line width=1.pt,black] (16.,-4.)-- (2.,-4.);
\draw [line width=1.pt,black] (2.,-4.)-- (2.,6.);
\draw [line width=1.pt,dashed] (2.,6.)-- (5.5,-4.);
\draw [line width=1.pt,dashed] (2.,6.)-- (9.,-4.);
\draw [line width=1.pt,dashed] (2.,6.)-- (12.5,-4.);
\draw [line width=1.5pt] (2.,-2.571428571428571)-- (12.5,-2.5714285714285716);
\draw [line width=1.5pt,black, loosely dotted] (2.,1.)-- (16.,-4.);
\draw [line width=1.5pt,black, loosely dotted] (2.,-1.5)-- (16.,-4.);
\draw [line width=1.5pt] (12.5,-3.375)-- (14.25,-3.375);
\draw [line width=1.5pt] (14.25,-3.6875)-- (15.125,-3.6875);
\draw [line width=1.5pt] (15.125,-3.84375)-- (15.5625,-3.84375);
\draw [line width=2.pt] (15.5625,-3.921875)-- (15.78125,-3.921875);
\draw [line width=2.pt] (2.,-5.22833)-- (12.5,-5.22833);
\draw [line width=2.pt] (12.5,-5.22833)-- (14.25,-5.22833);
\draw [line width=2.pt] (14.25,-5.22833)-- (15.125,-5.22833);
\draw [line width=2.pt] (15.125,-5.22833)-- (15.5625,-5.22833);
\draw [line width=2.pt] (15.5625,-5.22833)-- (15.78125,-5.22833);
\draw [line width=1.pt,dash pattern=on 6pt off 6pt] (2.,1.)-- (0.,1.);
\draw [line width=1.pt,dash pattern=on 6pt off 6pt] (2.,-4.)-- (0.,-4.);
\draw [line width=1.pt,dash pattern=on 6pt off 6pt] (2.,6.)-- (0.,6.);
\draw [line width=1.pt,dash pattern=on 6pt off 6pt] (2.,-1.5)-- (0.,-1.5);
\draw [line width=1.pt,dash pattern=on 6pt off 6pt] (2.,-2.571428571428571)-- (0.,-2.571428571428571);
\draw [line width=1.pt,dash pattern=on 6pt off 6pt] (14.25,-3.6875)-- (0.,-3.68529496840551);
\draw [->,line width=1.pt] (0.,-6.) -- (0.,7.3402362810479005);
\draw [->,line width=1.pt] (2.,0.) -- (16.735955709719306,0.);
\draw [fill,black] (16.,-4.) circle (2.5pt);
\draw [fill, black] (2.,-2.571428571428571) circle (2.5pt);
\draw [fill, black] (9.,-4.) circle (2.5pt);
\draw [fill, black] (5.5,-4.) circle (2.5pt);
\draw [fill, black] (12.5,-4.) circle (2.5pt);
\draw [fill, black] (12.5,-2.5714285714285716) circle (2.5pt);
\draw [fill, black] (2.,1.) circle (2.5pt);
\draw [fill, black] (2.,-1.5) circle (2.5pt);
\draw [fill, black] (12.5,-3.375) circle (2.5pt);
\draw [fill, black] (14.25,-3.375) circle (2.5pt);
\draw [fill, black] (14.25,-3.6875) circle (2.5pt);
\draw [fill, black] (15.125,-3.6875) circle (2.5pt);
\draw [fill, black] (15.125,-3.84375) circle (2.5pt);
\draw [fill, black] (15.5625,-3.84375) circle (2.5pt);
\draw [fill, black] (15.5625,-3.921875) circle (2.5pt);
\draw [fill, black] (15.78125,-3.921875) circle (2.5pt);
\draw [fill, black] (2.,-5.22833) circle (2.5pt);
\draw [fill, black] (12.5,-5.22833) circle (2.5pt);
\draw [fill, black] (12.5,-5.22833) circle (2.5pt);
\draw [fill, black] (14.25,-5.22833) circle (2.5pt);
\draw [fill, black] (14.25,-5.22833) circle (2.5pt);
\draw [fill, black] (15.125,-5.22833) circle (2.5pt);
\draw [fill, black] (15.125,-5.22833) circle (2.5pt);
\draw [fill, black] (15.5625,-5.22833) circle (2.5pt);
\draw [fill, black] (15.5625,-5.22833) circle (2.5pt);
\draw [fill, black] (15.78125,-5.22833) circle (2.5pt);
\draw [fill, black] (0.,1.) circle (2.5pt);
\draw [fill, black] (2.,-4.) circle (2.5pt);
\draw [fill, black] (0.,-4.) circle (2.5pt);
\draw [fill, black] (2.,6.) circle (2.5pt);
\draw [fill, black] (0.,6.) circle (2.5pt);
\draw [fill, black] (2.,-1.5) circle (2.5pt);
\draw [fill, black] (0.,-1.5) circle (2.5pt);
\draw [fill, black] (2.,-2.571428571428571) circle (2.5pt);
\draw [fill, black] (0.,-2.571428571428571) circle (2.5pt);
\draw [fill, black] (0.,-3.68529496840551) circle (2.5pt);
\draw[black] (-0.7523831134354492,6.135007074656876) node {$\alpha_{1}$};
\draw[black] (-1.2689116534694953,1.1008093477334178) node {$\frac{\alpha_{1}+\alpha}{2}$};
\draw[black] (-1.350258136056798,-1.3916120758904509) node {$\frac{\alpha_{1}+3\alpha}{4}$};
\draw[black] (-2.0258136056798,-2.5020968685941547) node {$\alpha'_{0}=\frac{\alpha_{1}+6\alpha}{7}$};
\draw[black] (-3.011399737418642,-3.5632267816221384) node {$\alpha'_{n}=\frac{\alpha_{1}+(2^{n+2}-1)\alpha}{2^{n+2}}$};
\draw[black] (-0.6515543536397259,-4.180658190278997) node {$\alpha$};
\draw[black] (14.719987878023154,-6.0) node {$J^{n}_{T_{\alpha}}$};
\draw[black] (16.405898961495858,-0.65386936205176753) node {$ht$};
\draw[black] (8.298720058616872,-6.) node
{$J^{\leq 1}_{T_{\alpha}}=J^{\leq 2}_{T_{\alpha'_{0}}}$};
\end{tikzpicture}
\caption{\label{fig:triangle} The exponent $\alpha_{0}'$ is
 determined by $\frac{7}{8}T_{\alpha'_{0}}=\frac{3}{4}T_{\alpha}$
 while for $n>1$\,, $\alpha_{n}'$ is determined by
 $T_{\alpha'_{n}}=(1-\frac{2^{-n-1}}{2})T_{\alpha}=(1-2^{-n-2})T_{\alpha}$\,.}
\end{figure}
\\
The equivalence of norms
\begin{equation}
  \label{eq:equivN21toN24}
  \kappa_{2}^{-1}N_{2,1}(\varphi)\leq  N_{2,i}(\varphi) \leq \kappa_{2}N_{2,1}(\varphi)\,, \quad
2\leq i\leq 4 \,,
\end{equation}
for some universal constant $\kappa_{2}>1$ 
is proved in Lemma~\ref{le:equivnorms2}
for
\begin{align}
\label{eq:defN1_2}
  N_{2,1}(\varphi)&=\sup_{\tau\in
                    [0,T[}\sqrt{T-\tau}\left\|\varphi\right\|_{L^{2}([0,\frac{\tau}{h}];L^{2}_{z}L^{2}_{y_{G}})} \,, \\
\label{eq:defN2_2}
N_{2,2}(\varphi)&=
\sqrt{T}\left\|\varphi\right\|_{L^{2}(h^{-1}J^{\leq 1}_{T};L^{2}_{z}L^{2}_{y_{G}})}+\sup_{\delta\in
                  ]0,T/8]}\sqrt{\delta}\left\|\varphi\right\|_{L^{2}(h^{-1}[T-2\delta,T-\delta];L^{2}_{z}L^{2}_{y_{G}})} \,, \\
\label{eq:defN3_2}
N_{2,3}(\varphi)&=
\sqrt{T}\sup_{n\in\nz}
2^{-n/2}\left\|\varphi\right\|_{L^{2}(h^{-1}J^{n}_{T};L^{2}_{z}L^{2}_{y_{G}})}\,,\\
\label{eq:defN4_2}
N_{2,4}(\varphi)&=
\sqrt{T}\left\|\varphi\right\|_{L^{2}(h^{-1}J^{\leq
  2}_{T};L^{2}_{z}L^{2}_{y_{G}})}+\sup_{\delta\in
                  ]0,T/12]}\sqrt{\delta}\left\|\varphi\right\|_{L^{2}(h^{-1}[T-3\delta,T-\delta];L^{2}_{z}L^{2}_{y_{G}})}\,.
\end{align}

\noindent$\mathbf{L_{22}(u_2)}:$ 
For $\alpha \in [\alpha_{0},\alpha_{1}[$\,, we seek an upper bound of
$N_{2,1}(\varphi)$ (with $T=T_\alpha$) for
$$
\varphi(t)=e^{\alpha N}L_{22}(u_2)(t)=
-i\int_{0}^{t}e^{\alpha N}\sqrt{h}a_{G}(V_{2})U(t)U(s)^{*}u_2(s)~ds\,.
$$
By the equivalence of norms $N_{2,1}$ and $N_{2,3}$ this is the same as
finding an upper bound for
$$
\sqrt{T_{\alpha}} \, 2^{-n/2}\|\varphi\|_{L^{2}(h^{-1}J^{n}_{T_{\alpha}};L^{2}_{z}L^{2}_{y_{G}})}
$$
 uniformly in both $\alpha \in [\alpha_0,\alpha_1[$ and $n\geq 0$\,,
 or equivalently for
$$
\sqrt{T_{\alpha}}
\|\varphi
\|_{L^{2}(h^{-1}J^{\leq 1}_{T_{\alpha}};L^{2}_{z}L^{2}_{y_{G}})}\quad\text{and}\quad
\sqrt{T_{\alpha}}2^{-n/2}\|\varphi\|_{L^{2}(h^{-1}J^{n}_{T_{\alpha}};L^{2}_{z}L^{2}_{y_{G}})}~(n>1)\,,
$$
with the same uniformity.\\
For $t\in h^{-1}J^{\leq 1}_{T_{\alpha}}$ we write
$$
\sqrt{T_{\alpha}}\varphi(t)=-i\sqrt{T_{\alpha}}e^{\alpha
  N}a_{G}(V)e^{-\alpha'_{0}N} \int_{s<t}U(t)U(s)^{*}\sqrt{h}\,w_{1}(s)~ds
$$
with 
$$
w_{1}(s)=e^{\alpha'_{0}N}1_{h^{-1}J^{\leq 1}_{T_{\alpha}}}(s)u_{2}(s)=
e^{\alpha'_{0}N}1_{h^{-1}J^{\leq 2}_{T_{\alpha'_{0}}}}(s)\,u_{2}(s)\,.
$$
Then Proposition~\ref{pr:expaLp}, the retarded Strichartz
estimate~\eqref{eq:retardStri} and the Cauchy-Schwarz inequality yield
\begin{align*}
  \sqrt{T_{\alpha}}\|\varphi\|_{L^{2}(h^{-1}J^{\leq
  1}_{T_{\alpha}};L^{2}_{z}L^{2}_{y_{G}})}
&\lesssim
  \sqrt{T_{\alpha}}\frac{C_{V}M_{\alpha01}}{\sqrt{\alpha'_{0}-\alpha}}\|\int_{s<t}U(t)U(s)^{*}\sqrt{h}w_{1}(s)~ds\|_{L^{2}_{z}L^{2}_{t}(h^{-1}J^{\leq
  1}_{T_{\alpha}};L^{r_{\sigma}}_{y_{G}})}
\\
&\lesssim
C_{V}M_{\alpha
  01}\sqrt{\gamma}\|\sqrt{h}w_{1}\|_{L^{2}_{z}L^{1}(h^{-1}J^{\leq
  2}_{T_{\alpha'_{0}}};L^{2}_{y_{G}})}\\
&\lesssim C_{V}M_{\alpha
  01}\sqrt{\gamma}\sqrt{T_{\alpha'_{0}}}\|w_{1}\|_{L^{2}_{z}L^{2}_{t}(h^{-1}J^{\leq
  2}_{T_{\alpha'_{0}}};L^{2}_{y_{G}})}\\
&\lesssim C_{V}M_{\alpha
  01}\sqrt{\gamma}\sqrt{T_{\alpha'_{0}}}\|e^{\alpha'_{0}}u_{2}\|_{L^{2}_{t}(h^{-1}J^{\leq
  2}_{T_{\alpha'_{0}}};L^{2}_{z}L^{2}_{y_{G}})}\,.
\end{align*}
The equivalence between the norms $N_{2,1}$ and $N_{2,4}$ implies
\begin{equation}
  \label{eq:M2L22_1}
\sqrt{T_{\alpha}}\|\varphi\|_{L^{2}(h^{-1}J^{\leq
    1}_{T_{\alpha}};L^{2}_{z}L^{2}_{y_{G}})}\lesssim
C_{V}^{2}M_{\alpha 01}^{2}\gamma M_{2}(u_{2})\,.
\end{equation}
For $t\in h^{-1}J^n_{T_\alpha}$\,, $n>1$\,, write
\begin{align*}
\sqrt{T_{\alpha}} \, 2^{-n/2} \varphi(t)
= &-i\sqrt{T_{\alpha}}2^{-n/2}e^{\alpha
  N}a_{G}(V)e^{-\alpha'_{0}N}\int_{s<t}U(t)U(s)^{*}\sqrt{h} \, w_{1}(s)~ds
\\
&-i\sqrt{T_{\alpha}} 2^{-n/2} \underbrace{\sum_{m=2}^n   e^{\alpha N} a_G(V_2) e^{-\alpha'_m N}
\int_{s<t} U(t)U(s)^*\sqrt{h} \, w_m(s)~ds}_{=\tilde{\varphi}}
\end{align*}
with for $m\geq 2$
\[w_m(s) 
= 1_{h^{-1}J^{m}_{T_{\alpha}}}(s)
e^{\alpha'_{m}N}u_2(s)
= 1_{h^{-1}\left[T_{\alpha'_{m}}-3\delta_{m},T_{\alpha'_{m}}-\delta_{m}\right]}(s)
             e^{\alpha'_{m}N}u_2(s)\,.\]
The first term is actually  estimate as we did for \eqref{eq:M2L22_1}
with the additional factor $2^{-n/2}\leq 1$\,. It suffices to consider
the application of  Proposition~\ref{pr:expaLp}, the retarded Strichartz
estimate~\eqref{eq:retardStri} and the Cauchy-Schwarz inequality to 
\begin{align}
\sqrt{T_{\alpha}} \,& 2^{-n/2}
\|\tilde{\varphi}\|_{L^{2}(h^{-1}J^{n}_{T_{\alpha}};L^{2}_{z}L^{2}_{y_{G})}} 
\nonumber \\ 
 &\lesssim \sqrt{T_{\alpha}} \, 2^{-n/2}  \sum_{m=2}^n  \frac{C_{V} \, M_{\alpha 01}}{\sqrt{\alpha'_m-\alpha}}
 \left\| \int_{s<t} U(t)U(s)^* \sqrt{h} w_m(s)~ds \right\|_{L^{2}_{z}L^2_t(h^{-1}J^{n}_{T_{\alpha}};L^{r_\sigma}_{y_{G}})}  \nonumber \\ 
 &\lesssim C_{V} \, M_{\alpha 01} \, \sqrt{\gamma} \, 2^{-n/2} \sum_{m=2}^n 2^{m/2} 
 \left\| \sqrt{h} w_m \right\|_{L^{2}_{z}L^1_t(h^{-1}J^{m}_{T_{\alpha}};L^{2}_{y_G})}  \nonumber \\ 
 &\lesssim C_{V} \, M_{\alpha 01} \sqrt{\gamma}  \, 2^{-n/2} \sum_{m=2}^n  \sqrt{T_\alpha}
 \left\| w_m \right\|_{L^{2}_{z}L^2_t(h^{-1}J^{m}_{T_{\alpha}};L^{2}_{y_{G}})} \,.\label{eq:UnifEstiDecoupage}
\end{align}
Thanks to the equivalence of the norms $N_{2,1}$ and $N_{2,4}$ (with
$T=T_{\alpha'_m}$), we obtain for $m\geq 2$
\begin{align}
\sqrt{T_\alpha} \left\| w_m \right\|_{L^{2}_{z}L^2_t(h^{-1}J^{m}_{T_{\alpha}};L^{2}_{y_G})} 
&= 2^{\frac{m+2}{2}}\sqrt{\delta_m} \left\| e^{\alpha'_{m}N}u_2(s)\right\|_{L^2_t(h^{-1}\left[T_{\alpha'_{m}}-3\delta_{m},T_{\alpha'_{m}}-\delta_{m}\right];L^{2}_{z}L^{2}_{y_G})} \nonumber \\
&\lesssim 2^{m/2} C_{V} \, M_{\alpha 01} \, \sqrt{\gamma} M_2(u_2) \,.\label{eq:M2L22_2}
\end{align}
Putting together \eqref{eq:UnifEstiDecoupage} and
\eqref{eq:M2L22_2} gives
\begin{align*}\sqrt{T_{\alpha}} \, 2^{-n/2}\|\varphi\|_{L^{2}(h^{-1}J^{n}_{T_{\alpha}};L^{2}_{z}L^{2}_{y_{G}})} 
&\lesssim 2^{-n/2}\sum_{m=0}^n 2^{m/2} C_{V}^2 \, M_{\alpha 01}^2 \, \gamma \, M_2(u_2) \\
&\lesssim  C_{V}^2 \, M_{\alpha 01}^2 \, \gamma \, M_2(u_2)
\end{align*}
which, combined with \eqref{eq:M2L22_1} and the normalization of
$M_{2}(L_{22}(u_{2}))$\,, yields
\begin{equation}\label{eq:EstimL22}
M_2(L_{22}(u_2)) \lesssim C_{V} \, M_{\alpha 01} \, \sqrt{\gamma} \, M_2(u_2) \,.
\end{equation}

\medskip

The estimate of $L_{11}(u_{1})$ starts with the same decomposition of
the interval $[0,T/h]$ with the norms
\begin{align}
\label{eq:defN11}
  N_{1,1}(\varphi)&=\sup_{\tau\in
  [0,T[}\sqrt{T-\tau}\left\|\frac{\varphi(t)}{\sqrt{ht}}\right\|_{L^{1}([0,\frac{\tau}{h}];L^{2}_{z}L^{2}_{y_{G}})} \,, \\
\label{eq:defN12}
N_{1,2}(\varphi)&=
\left\|\left(\frac{T}{ht}\right)^{1/2}\varphi\right\|_{L^{1}(h^{-1}J^{\leq
  1}_{T};L^{2}_{z}L^{2}_{y_{G}})}+\sup_{\delta\in
                  ]0,T/8]}\left(\frac{\delta}{T}\right)^{1/2}\left\|\varphi\right\|_{L^{1}(h^{-1}[T-2\delta,T-\delta];L^{2}_{z}L^{2}_{y_{G}})} \,, \\
\label{eq:defN13}
N_{1,3}(\varphi)&=
\left\|\left(\frac{T}{ht}\right)^{1/2}\varphi\right\|_{L^{1}(h^{-1}J^{\leq
  1}_{T};L^{2}_{z}L^{2}_{y_{G}})}+\sup_{n>1}
2^{-n/2}\left\|\varphi\right\|_{L^{1}(h^{-1}J^{n}_{T};L^{2}_{z}L^{2}_{y_{G}})}\,,\\
\label{eq:defN14}
N_{1,4}(\varphi)&=
\left\|\left(\frac{T}{ht}\right)^{1/2}\varphi\right\|_{L^{1}(h^{-1}J^{\leq
  2}_{T};L^{2}_{z}L^{2}_{y_{G}})}+\sup_{\delta\in
                  ]0,T/12]}\left(\frac{\delta}{T}\right)^{1/2}\left\|\varphi\right\|_{L^{1}(h^{-1}[T-3\delta,T-\delta];L^{2}_{z}L^{2}_{y_{G}})}\,.
\end{align}
Those norms are equivalent according to 
\begin{equation}
  \label{eq:equivN11toN15}
\kappa_{1}^{-1}N_{1,1}(\varphi)\leq  N_{1,i}(\varphi) \leq \kappa_{1}N_{1,1}(\varphi)\quad,
2\leq i\leq 4
\end{equation}
with a universal constant $\kappa_{1}> 1$\,. See
Lemma~\ref{le:equivnorms2} for the proof.

\noindent$\mathbf{L_{11}(u_1)}$\textbf{-Step~1, Decomposition of $L_{11}(u_1)$:}
For $\alpha \in [\alpha_{0},\alpha_{1}[$\,, we seek an upper bound of
$N_{1,1}(\varphi)$ for
$$
\varphi(t)=e^{\alpha N}L_{11}(u_1)(t)=
-i\int_{0}^{t}e^{\alpha N}\sqrt{h}a_{G}(V_{2})U(t)U(s)^{*}u_1(s)~ds\,.
$$
By the equivalence of norms $N_{1,1}$ and $N_{1,3}$ this is the same
as finding a uniform upper bound for
$$
\left\|\left(\frac{T_{\alpha}}{ht}\right)^{1/2}\varphi\right\|_{L^{1}(h^{-1}J_{T_{\alpha}}^{\leq
  1};L^{2}_{z}L^{2}_{y_{G}})}\quad \text{and}\quad
2^{-n/2}\|\varphi\|_{L^{1}(h^{-1}J^{n}_{T_{\alpha}};L^{2}_{z}L^{2}_{y_{G}})}\;\text{for}\; n>1\,.
$$
Setting $\psi_{1}(t)=\left(\frac{T_{\alpha}}{ht}\right)^{1/2}1_{h^{-1}J^{\leq 1}_{T_{\alpha}}}(t)\varphi(t)$ and, for $n>1$, 
$\psi_{n}(t)=2^{-n/2}1_{h^{-1}J^{n}_{T_{\alpha}}}(t)\varphi(t)$ gives
\begin{align*}
  \psi_{1}(t)&=-i\int_{0}^{t}\left(\frac{T_{\alpha}}{ht}\right)^{1/2}e^{\alpha
                 N}\sqrt{h}a_{G}(V_{2})U(t)U(s)^{*}1_{h^{-1}J^{\leq
                 1}_{T_{\alpha}}}(s)u_1(s)~ds\,, \quad t\in h^{-1}J^{\leq 1}_{T_{\alpha}}\,,
\end{align*}
and, for $n>1$,
\begin{align*}
\psi_{n}(t)&=
-i\int_{0}^{t}2^{-n/2}e^{\alpha
                 N}\sqrt{h}a_{G}(V_{2})U(t)U(s)^{*}1_{h^{-1}J^{\leq
                 1}_{T_{\alpha}}}(s)u_1(s)~ds\\
&\quad
-i\sum_{1< m \leq n}\int_{0}^{t}\sqrt{h}2^{-n/2}e^{\alpha
                 N}a_{G}(V_{2})U(t)U(s)^{*}1_{h^{-1}J^{m}_{T_{\alpha}}}(s)u_1(s)~ds\,,\quad
   t\in h^{-1}J^{n}_{T_{\alpha}}\,.
\end{align*}
This allows to rewrite the above decomposition as
\begin{align*}
 \psi_{1}(t) \stackrel{t\leq\frac{3T_{\alpha}}{4h}}{=} & -i\int_{0}^{\frac{3T_{\alpha}}{4h}}
\underbrace{1_{[0,t]}(s)\left(\frac{T_{\alpha}}{T_{\alpha'_{0}}}\right)^{1/2}\left(\frac{hs}{ht}\right)^{1/2}e^{\alpha
                N}\sqrt{h}a_{G}(V_{2})e^{-\alpha'_{0}N}}_{B_{11}(t,s)}U(t)U(s)^{*}w_{1}(s)~ds
\,,
\end{align*}
and, for $n>1$,
\begin{align*}
\psi_{n}(t)\stackrel{t\in h^{-1}J^{n}_{T_{\alpha}}}{=} &
-i \int_{0}^{\frac{T_{\alpha}}{h}}
\underbrace{1_{[0,\frac{3T_{\alpha}}{4h}]}(s)2^{-n/2}e^{\alpha
                 N}\sqrt{h}a_{G}(V_{2})e^{-\alpha'_{0}N}\left(\frac{hs}{T_{\alpha'_{0}}}\right)^{1/2}}_{B_{n1}(t,s)}
U(t)U(s)^{*}
w_{1}(s)~ds\\
&
-i\sum_{m=
   2}^{n}\int_{0}^{\frac{T_{\alpha}}{h}}
\underbrace{1_{[0,t]\cap h^{-1}J^{m}_{T_{\alpha}}}(s)2^{-(n-m)/2}e^{\alpha
                 N}\sqrt{h}a_{G}(V_{2})e^{-\alpha'_{m}N}}_{B_{nm}(t,s)}U(t)U(s)^{*}
w_{m}(s)~ds
\,,
\end{align*}
\begin{align*}
\text{with} \quad & w_{1}(s)=1_{h^{-1}J^{\leq
                 2}_{T_{\alpha'_{0}}}}(s)\left(\frac{T_{\alpha'_{0}}}{hs}\right)^{1/2}e^{\alpha'_{0}N}u_1(s)\\
\text{and} \quad &
             w_{m}(s)\stackrel{m> 1}{=}2^{-m/2}1_{h^{-1}J^{m}_{T_{\alpha}}}(s)
e^{\alpha'_{m}N}u_1(s)=2^{-m/2}1_{\left[\frac{T_{\alpha'_{m}}-3\delta_{m}}{h},\frac{T_{\alpha'_{m}}-\delta_{m}}{h}\right]}(s)
             e^{\alpha'_{m}N}u_1(s)\,.
\end{align*}
Proposition~\ref{pr:StriDT} tells us
\begin{align*}
  \|\psi_{1}\|_{L^{1}(h^{-1}J^{\leq
  1}_{T_{\alpha}};L^{2}_{z}L^{2}_{y_{G}})}
\lesssim &
  \left(\sup_{s\in[0,\frac{3T_\alpha}{4h}]}\int_{0}^{\frac{3T_\alpha}{4h}}
\|B_{11}(t,s)\|^{2}~dt \right)^{1/2}\|w_{1}
\|_{L^{1}(h^{-1}J^{\leq 2}_{T_{\alpha'_{0}}};L^{2}_{z}L^{2}_{y_{G}})} \,, 
\end{align*}
and, for $n>1$,
\begin{align*}
\|\psi_{n}\|_{L^{1}(h^{-1}J^{n}_{T_{\alpha}};L^{2}_{z}L^{2}_{y_{G}})}
\lesssim &
 \left(\sup_{s\in[0,\frac{3T_\alpha}{4h}]}\int_{h^{-1}J^{n}_{T_{\alpha}}}\|B_{n1}(t,s)\|^{2}~dt\right)^{1/2}
\|w_{1}
\|_{L^{1}(h^{-1}J^{\leq 2}_{T_{\alpha'_{0}}};L^{2}_{z}L^{2}_{y_{G}})}\\
&+\sum_{m=2}^{n}
\left(\sup_{s\in h^{-1}J^{m}_{T_{\alpha}}}\int_{h^{-1}J^{n}_{T_{\alpha}}}\|B_{nm}(t,s)\|^{2}~dt\right)^{1/2}\|w_{m}\|_{L^{1}([\frac{T_{\alpha'_{m}}-3\delta_{m}}{h},\frac{T_{\alpha'_{m}}-\delta_{m}}{h}];L^{2}_{z}L^{2}_{y_{G}})}\,.
\end{align*}
From the comparison between the norms $N_{1,1}$ and $N_{1,4}$ we know
\begin{align*}
 \|w_{1}\|_{L^{1}(h^{-1}J^{\leq
     2}_{T_{\alpha'_{0}}};L^{2}_{z}L^{2}_{y_{G}})}
&\lesssim
\sup_{\tau\in
     [0,T_{\alpha'_{0}}[}\sqrt{T_{\alpha'_{0}}-\tau}\left\|
     \frac{e^{\alpha'_{0}N}u_1}{\sqrt{ht}}\right\|_{L^{1}([0,\frac{\tau}{h}];L^{2}_{z}L^{2}_{y_{G}})}\\
&\lesssim
      M_{\alpha01}C_{V}\gamma^{1/2}M_1(u_1)\,,
\end{align*}
while for $m>1$\,,
\begin{align*}
\|w_{m}\|_{L^{1}([\frac{T_{\alpha'_{m}}-3\delta_{m}}{h},\frac{T_{\alpha'_{m}}-\delta_{m}}{h}];L^{2}_{z}L^{2}_{y_{G}})}&\lesssim
\left(\frac{T_{\alpha'_{m}}}{\delta_{m}}\right)^{1/2}2^{-m/2}
\sup_{\tau\in   [0,T_{\alpha'_{m}}[}\sqrt{T_{\alpha'_{m}}-\tau}\|\frac{e^{\alpha'_{m}N}u_1}{\sqrt{ht}}\|_{L^{1}([0,\frac{\tau}{h}];L^{2}_{z}L^{2}_{y_{G}})}\\
&\lesssim 
       \left(\frac{T_{\alpha}}{T_{\alpha}2^{-m-2}}\right)^{1/2}2^{-m/2}M_{\alpha01}C_{V}\gamma^{1/2}M_1(u_1)\\
&\lesssim M_{\alpha01}C_{V}\gamma^{1/2}M_1(u_1)\,.
\end{align*}
We have proved
\begin{multline}
\frac{\sup_{\tau\in
  [0,T_{\alpha}[}\sqrt{T_{\alpha}-\tau}\left\|\frac{e^{\alpha
  N}L_{11}(u_1)}{\sqrt{ht}}\right\|_{L^{1}([0,\frac{\tau}{h}];L^{2}_{z}L^{2}_{y_{G}})}}{M_{\alpha01}C_{V}\gamma^{1/2}M_1(u_1)} \\
\lesssim \left(\sup_{s\in[0,\frac{3T_\alpha}{4h}]}\int_{0}^{\frac{3T_\alpha}{4h}}
\|B_{11}(t,s)\|^{2}~dt \right)^{1/2}
+\sup_{n\geq
   1}\left(\sup_{s\in[0,\frac{3T_\alpha}{4h}]}\int_{h^{-1}J^{n}_{T_{\alpha}}}\|B_{n1}(t,s)\|^{2}~dt\right)^{1/2}
\\
\label{eq:Bnm}
\hspace{-2cm}+\sup_{n> 1}\sum_{m=2}^{n}\left(\sup_{s\in h^{-1}J^{m}_{T_{\alpha}}}\int_{h^{-1}J^{n}_{T_{\alpha}}}\|B_{nm}(t,s)\|^{2}~dt\right)^{1/2}\,.
\end{multline}
It remains to estimate every term of the above right-hand side.\\
\noindent$\mathbf{L_{11}(u_1)}$\textbf{-Step~2, Estimate for $B_{11}$:}
The expression
$$
B_{11}(t,s)=1_{[0,t]}(s)\left(\frac{T_{\alpha}}{T_{\alpha'_{0}}}\right)^{1/2}\left(\frac{hs}{ht}\right)^{1/2}e^{\alpha
  N}\sqrt{h}a_{G}(V_{2})e^{-\alpha'_{0}N}
$$
implies, with $T_{\alpha}=\gamma(\alpha_{1}-\alpha)=7\gamma(\alpha_{0}'-\alpha)$\,,
$$
\left\|B_{11}(t,s)\right\|^{2}\leq
\frac{T_{\alpha}}{T_{\alpha'_{0}}} \frac{7\gamma}{T_\alpha} M_{\alpha 01}^{2} C_{V}^{2}h1_{[0,t]}(s)\frac{hs}{ht}
\leq 7
\gamma M_{\alpha01}^{2}C_{V}^{2} \frac{4hs}{3T_\alpha}  \frac{1_{[0,t]}(s)}{t}\,.
$$
We obtain
\begin{equation*}
\int_{0}^{\frac{3T_{\alpha}}{4h}}\|B_{11}(t,s)\|^{2}~dt \leq 7  \gamma M_{\alpha01}^{2}C_{V}^{2} \frac{4hs}{3T_\alpha}  \ln(\frac{3T_{\alpha}}{4hs})\,.
\end{equation*}
and
\begin{equation}
  \label{eq:B00}
\left(\sup_{s\in[0,\frac{3T_\alpha}{4h}]}\int_{0}^{\frac{3T_\alpha}{4h}}
\|B_{11}(t,s)\|^{2}~dt \right)^{1/2}\lesssim
\gamma^{1/2}M_{\alpha01}C_{V}\,.
\end{equation}
\noindent$\mathbf{L_{11}(u_1)}$\textbf{-Step~3, Estimate for
  $B_{n1}$\,, $n>1$:}
From
$$
B_{n1}(t,s)=
1_{[0,\frac{3T_{\alpha}}{4h}]}(s)2^{-n/2}e^{\alpha
                 N}\sqrt{h}a_{G}(V_{2})e^{-\alpha'_{0}N}\left(\frac{hs}{T_{\alpha'_{0}}}\right)^{1/2}
$$
we deduce with
$\alpha'_{0}-\alpha=\frac{\alpha_{1}-\alpha}{7}=\frac{T_{\alpha}}{7\gamma}$
and $T_{\alpha_{0}'}=\frac{6 T_{\alpha}}{7}$\,,
$$
\|B_{n1}(t,s)\|^{2}\leq
1_{[0,\frac{3T_{\alpha}}{4h}]}(s)2^{-n}\frac{h
M_{\alpha01}^{2}C_{V}^{2}
}{(\alpha'_{0}-\alpha)}\left(\frac{3T_{\alpha}/4}{T_{\alpha'_{0}}}\right)
\leq 1_{[0,\frac{3T_{\alpha}}{4h}]}(s) \frac{2^{-n}7\gamma h}{T_{\alpha}}M_{\alpha01}^{2}C_{V}^{2}
\,.
$$
With
$J^{n}_{T_{\alpha}}=[(1-2^{-n})T_{\alpha},(1-2^{-n-1})T_{\alpha}[$
for $n> 1$ we obtain
$$
\int_{h^{-1}J^{n}_{T_{\alpha}}}\|B_{n0}(t,s)\|^{2}~dt\leq 
2^{-n-1}T_{\alpha}\times \frac{2^{-n}7\gamma }{T_{\alpha}}
M_{\alpha01}^{2}C_{V}^{2}
\leq
\frac{7\gamma}{4}
M_{\alpha01}^{2}C_{V}^{2}\,.
$$
and
\begin{equation}
  \label{eq:Bn0}
\sup_{n>
   1}\left(\sup_{s\in[0,\frac{3T_\alpha}{4h}]}\int_{h^{-1}J^{n}_{T_{\alpha}}}\|B_{n0}(t,s)\|^{2}~dt\right)^{1/2}\lesssim
 \gamma^{1/2}M_{\alpha01}C_{V}\,.
\end{equation}
\noindent$\mathbf{L_{11}(u_1)}$\textbf{-Step~4, Estimate for the
  $B_{nm}$'s, $n,m>1$:}
From
$$
B_{nm}(t,s)=1_{[0,t]\cap h^{-1}J^{m}_{T_{\alpha}}}(s)2^{-(n-m)/2}e^{\alpha
                 N}\sqrt{h}a_{G}(V_{2})e^{-\alpha'_{m}N}
$$
and $\alpha'_{m}-\alpha=2^{-(m+2)}(\alpha_{1}-\alpha)=\frac{2^{-(m+2)}T_{\alpha}}{\gamma}$\,,
we deduce

$$
\|B_{nm}(t,s)\|^{2}\leq 1_{[0,t]\cap
  h^{-1}J^{m}_{T_{\alpha}}}(s)2^{-(n-m)}\frac{h2^{m+2}\gamma}{T_{\alpha}}M_{\alpha01}^{2}C_{V}^{2}\,.
$$
Using again that the length of $J^{n}_{T_{\alpha}}$ is $2^{-(n+1)}T_{\alpha}$\,, we
get
$$
\sup_{s\in h^{-1}J^{m}_{T_{\alpha}}}
\int_{h^{-1}J^{n}_{T_{\alpha}}}\|B_{nm}(t,s)\|^{2}~dt
\leq 2\gamma 2^{-2(n-m)}M_{\alpha01}^{2}C_{V}^{2}
$$
and
\begin{equation}
  \label{eq:Bnm1}
\sup_{n\geq 1}\sum_{m=1}^{n}\left(\sup_{s\in
    h^{-1}J^{m}_{T_{\alpha}}}\int_{h^{-1}J^{n}_{T_{\alpha}}}\|B_{nm}(t,s)\|^{2}~dt\right)^{1/2}\lesssim
\gamma^{1/2}M_{\alpha01}C_{V}
\,.
\end{equation}
\noindent$\mathbf{L_{1\infty}(u_{\infty})}$\textbf{-Step~1,
  Decomposition of $\mathbf{L_{1\infty}(u_{\infty})}$:}\\ Compared
with the decomposition of $\mathbf{L_{22}(u_{2})}$ and $\mathbf{L_{11}(u_{1})}$\,, an additional
dyadic decomposition has to be done around $0$ in order to absorb the
weight $\frac{1}{\sqrt{ht}}$ properly and to use
Proposition~\ref{pr:L2L1Linf}. Decompose now
$[0,T]=\cup_{n\in\zz}J^{n}_{T}$ where $J^{0}_{T}$ is now the interval
$[T/4,T/2[$ and $
J^{n}_{T}=2^{n}J^{0}_{T}$ for $n<0$\,, according to figure
\ref{fig:JnTneg}. In particular, the interval previously
denoted by $J^{0}_{T}$ is now $J^{\leq 0}_{T}$ while $J^{\leq
  n_{0}}_{T}$ is not changed for $n_{0}>0$\,.
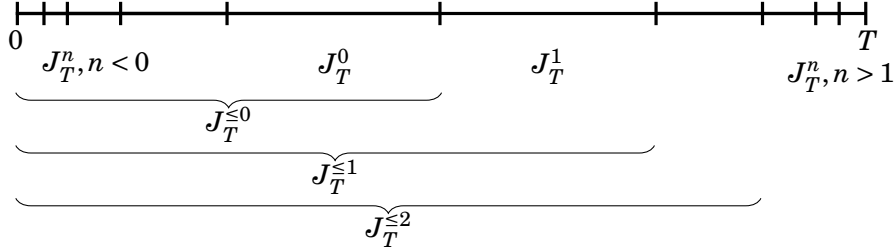
\begin{figure}[h!]
\centering
\begin{tikzpicture}[scale=0.7]
\draw[very thick,|-|] (0,0)--(16,0);
\draw[very thick,-|](0,0)--(8,0);
\draw[very thick,-|](2,0)--(4,0);
\draw[very thick,-|](1,0)--(2,0);
\draw[very thick,|-|](1/2,0)--(1,0);
\draw[very thick,|-](12,0)--(14,0);
\draw[very thick,|-](14,0)--(15,0);
\draw[very thick,|-|](15,0)--(15.5,0);
\node[](J0-) at (6,-1) {$J^{0}_{T}$};
\node[](T) at (16,-0.5) {$T$};
\node[](0) at (0,-0.5) {$0$};
\node[](J0+) at (10,-1) {$J^{1}_{T}$};
\node[](Jn>0) at (15.5,-1.2) {$J^{n}_{T},n>1$};
\node[](Jn<0) at (1.5,-1) {$J^{n}_{T}, n<0$};
\draw [decorate,decoration={brace,amplitude=0.2cm,mirror}](0,-2.5)--(12,-2.5);
\node[](J0) at (6,-3.1) {$J^{\leq 1}_{T}$};
\draw [decorate,decoration={brace,amplitude=0.2cm,mirror}](0,-1.5)--(8,-1.5);
\node[](Jleq0) at (4,-2.1) {$J^{\leq 0}_{T}$};
\draw [decorate,decoration={brace,amplitude=0.2cm,mirror}](0,-3.5)--(14,-3.5);
\node[](Jleq1) at (7,-4.1) {$J^{\leq 2}_{T}$};
\end{tikzpicture}
\caption{\label{fig:JnTneg} The time intervals $J^{n}_{T}$\,, 
 $n\in\zz$\,.}
\end{figure}

We seek an upper bound of
$N_{1,1}(\varphi)$ for
$$
\varphi(t)=e^{\alpha N}L_{1\infty}(u_{\infty})(t)=
-h\int_{0}^{t}e^{\alpha N}a_{G}(V_{2})U(t)U(s)^{*}a_{G}(V_{1})^{*}u_{\infty}(s)~ds\,.
$$
By the equivalence of norms $N_{1,1}$ and $N_{1,3}$ this is equivalent to
proving a uniform upper bound for
$$
\left\|\left(\frac{T_{\alpha}}{ht}\right)^{1/2}\varphi\right\|_{L^{1}(h^{-1}J_{T_{\alpha}}^{\leq
  1};L^{2}_{z}L^{2}_{y_{G}})}\quad \text{and}\quad
2^{-n/2}\|\varphi\|_{L^{1}(h^{-1}J^{n}_{T_{\alpha}};L^{2}_{z}L^{2}_{y_{G}})}\;\text{for}\; n>1\,.
$$
But the dyadic decomposition around $0$ says
$$
\left\|\left(\frac{T_{\alpha}}{ht}\right)^{1/2}\varphi\right\|_{L^{1}(h^{-1}J_{T_{\alpha}}^{\leq
  1};L^{2}_{z}L^{2}_{y_{G}})}
\leq 2\sum_{n\leq
  1}\|2^{-\frac{n+1}{2}}\varphi\|_{L^{1}(h^{-1}J^{n}_{T_{\alpha}};L^{2}_{z}L^{2}_{y_{G}})}
=2\|\sum_{n\leq
  1}2^{-\frac{n+1}{2}}1_{h^{-1}J^{n}_{T_{\alpha}}}(t)\varphi\|_{L^{1}(h^{-1}J^{\leq  1}_{T_{\alpha}};L^{2}_{z}L^{2}_{y_{G}})}\,.
$$
\noindent$\mathbf{L_{1\infty}(u_{\infty})}$\textbf{-Step~2,
Estimate on $\mathbf{h^{-1}J^{n\leq 1}_{T_{\alpha}}}$:}\\ 
We write $\varphi_{1}=\sum_{n\leq
  1}2^{-\frac{n+1}{2}}1_{h^{-1}J^{n}_{T_{\alpha}}}(t)e^{\alpha N}L_{1\infty}(u_{\infty})=\sum_{n\leq
  1}\varphi_{1,n}(t)$ where 
\begin{align*}  
\varphi_{1,n}(t)=&-h\sum_{m=-\infty}^{1}2^{-\frac{n+1}{2}}1_{h^{-1}J^{n}_{T_{\alpha}}}(t)
\times
\\
&\quad\int_{0}^{t}e^{\alpha
N}a_{G}(V_{2})e^{-\frac{\alpha+\alpha'_{0}}{2}N}U(t)U(s)^{*}
e^{\frac{\alpha+\alpha'_{0}}{2}N}a^{*}_{G}(V_{1})e^{-\alpha'_{0}N}
1_{h^{-1}J^{m}_{T_{\alpha}}}(s)e^{\alpha'_{0}N}u_{\infty}^{h}(s)~ds\\
=&-h\sum_{m=-\infty}^{1}1_{h^{-1}J^{n}_{T_{\alpha}}}
\int_{0}^{t}B_{1n}U(t)U(s)^{*}B_{2m}^{*}(s)\varphi_{\infty,m}(s)~ds
\end{align*}
with
\begin{eqnarray*}
  && B_{1n}=2^{-\frac{n+1}{2}}e^{\alpha
     N}a_{G}(V_{2})e^{-\frac{\alpha+\alpha'_{0}}{2}N}\quad,\\
&&
     \|B_{1n}\|_{L^{2}_{z}L^{2}_{y_{G}}\leftarrow L^{2}_{z}L^{r_{\sigma}}_{y_{G}}}\lesssim
     \frac{M_{\alpha 01}\|V_{2}\|_{L^{r'_{\sigma}}}}{\sqrt{\alpha_{0}'-\alpha}}2^{-n/2}\lesssim
     \frac{M_{\alpha 01}C_{V}\gamma^{1/2}}{T_{\alpha_{1}}^{1/2}}2^{-n/2}\,,\\
&&
 B_{2m}^{*}=e^{\frac{\alpha+\alpha'_{0}}{2}N}a_{G}^{*}(V_{1})e^{-\alpha'_{0}N}1_{h^{-1}J^{m}_{T_{\alpha}}}(s)
 \frac{\sqrt{hs}}{\sqrt{T_{\alpha}-hs}}\quad,\\
&&
 \|B_{2m}^{*}\|_{L^{2}_{z}L^{r'_{\sigma}}_{y_{G}}\leftarrow L^{2}_{z}L^{2}_{y_{G}}}
 \lesssim
 \frac{M_{\alpha 01}\|V_{1}\|_{L^{r'_{\sigma}}}}{
 \sqrt{\alpha'_{0}-\alpha}}2^{m/2}
 \lesssim
   \frac{M_{\alpha 01}C_{V}\gamma^{1/2}}{T_{\alpha}^{1/2}}2^{m/2}\,,\\
&&
   \varphi_{\infty,m}(s)=1_{h^{-1}J^{m}_{T_{\alpha}}}(s)\varphi_{\infty}(s)\quad,\quad
   \varphi_{\infty}(s)=e^{\alpha'_{0}N}\frac{\sqrt{T_{\alpha}-hs}}{\sqrt{hs}}u_{\infty}(s)\,.
\end{eqnarray*}
By noticing
$$
|h^{-1}J^{n}_{T_{\alpha}}|\leq T_{\alpha} h^{-1}2^{n}
$$
the upper bound of Proposition~\ref{pr:L2L1Linf} gives
\begin{eqnarray*}
\|\varphi_{1}\|_{L^{1}(h^{-1}J^{\leq 1}_{T_{\alpha}};L^{2}_{z}L^{2}_{y_{G}})}
&\lesssim &
\left[\sum_{-\infty\leq m\leq n\leq
  1}2^{n/2}(M_{\alpha 01}C_{V}\gamma^{1/2}2^{-n/2})(M_{\alpha 01}C_{V}\gamma^{1/2}2^{m/2})\right]
\left\|\varphi_{\infty}\right\|_{L^{\infty}(h^{-1}J^{\leq 1}_{T_{\alpha}};L^{2}_{z}L^{2}_{y_{G})}}
\\
&\lesssim& M_{\alpha 01}^{2}C_{V}^{2}\gamma M_{\infty}(u_{\infty})\,.
\end{eqnarray*}
We proved
\begin{equation}
  \label{eq:L1inf1}
\|\frac{1}{\sqrt{ht}}e^{\alpha
  N}L_{1\infty}(u_{\infty})\|_{L^{1}(h^{-1}J^{\leq 1}_{T_{\alpha}};L^{2}_{z}L^{2}_{y_{G}})}\lesssim
M_{\alpha 01}C_{V}^{2}\gamma M_{\infty}(u_{\infty})\,.
\end{equation}
\noindent$\mathbf{L_{1\infty}(u_{\infty})}$\textbf{-Step~3,
Estimate on $\mathbf{h^{-1}J^{n}_{T_{\alpha}}}$, $\mathbf{n>1}$:}\\ 
Write $\varphi_{1}(t)=2^{-n/2}1_{h^{-1}J^{n}_{T_{\alpha}}}(t)e^{\alpha
  N}L_{1\infty}(u_{\infty})$\,, where
\begin{align*}  
\varphi_{1}(t)=&-h\sum_{m=-\infty}^{1}2^{-\frac{n}{2}}1_{h^{-1}J^{n}_{T_{\alpha}}}(t)
\times
\\
&\quad\int_{0}^{t}e^{\alpha
N}a_{G}(V_{2})e^{-\frac{\alpha+\alpha'_{0}}{2}N}U(t)U(s)^{*}
e^{\frac{\alpha+\alpha'_{0}}{2}N}a_{G}^{*}(V_{1})e^{-\alpha'_{0}N}
1_{h^{-1}J^{m}_{T_{\alpha}}}(s)e^{\alpha'_{0}N}u_{\infty}^{h}(s)~ds
\\
-&h\sum_{m=2}^{n}2^{-\frac{n}{2}}1_{h^{-1}J^{n}_{T_{\alpha}}}(t)\times\\
&\quad\int_{0}^{t}e^{\alpha
N}a_{G}(V_{2})e^{-\frac{\alpha+\alpha'_{m}}{2}N}U(t)U(s)^{*}
e^{\frac{\alpha+\alpha'_{m}}{2}N}a_{G}^{*}(V_{1})e^{-\alpha'_{m}N}
1_{h^{-1}J^{m}_{T_{\alpha}}}(s)e^{\alpha'_{m}N}u_{\infty}^{h}(s)~ds\\
=&-h\sum_{m=-\infty}^{1}1_{h^{-1}J^{n}_{T_{\alpha}}}(t)\times
\int_{0}^{t}B_{1n}U(t)U(s)^{*}B_{2m}^{*}(s)\varphi_{\infty,m}(s)~ds\\
&-h\sum_{m=2}^{n}1_{h^{-1}J^{n}_{T_{\alpha}}}(t)\times
\int_{0}^{t}B_{1nm}U(t)U(s)^{*}B_{2nm}^{*}(s)\varphi_{\infty,m}(s)~ds\,.
\end{align*}
The family $\mathcal{I}$ of Proposition~\ref{pr:L2L1Linf} is made here
of the single interval $h^{-1}J^{n}_{T_{\alpha}}$ while the family 
$\mathcal{J}=\left\{h^{-1}J^{m}_{T_{\alpha}},m\leq n\right\}$ is
splitted in two parts $m\leq 1$ and $2\leq m\leq n$\,. In the last two
lines the notations correspond to
\begin{eqnarray*}
  && B_{1n}=2^{-\frac{n}{2}}e^{\alpha
     N}a_{G}(V_{2})e^{-\frac{\alpha+\alpha'_{0}}{2}N}\quad,\\
&&
     \|B_{1n}\|_{L^{2}_{z}L^{2}_{y_{G}}\leftarrow L^{2}_{z}L^{r_{\sigma}}_{y_{G}}}\lesssim
     \frac{M_{\alpha 01}\|V_{2}\|_{L^{r'_{\sigma}}}}{\sqrt{\alpha_{0}'-\alpha}}2^{-n/2}\lesssim
     \frac{M_{\alpha 01}C_{V}\gamma^{1/2}}{T_{\alpha_{1}}^{1/2}}2^{-n/2}\,,\\
m\leq 1&&
 B_{2m}^{*}=e^{\frac{\alpha+\alpha'_{0}}{2}N}a_{G}^{*}(V_{1})e^{-\alpha'_{0}N}1_{h^{-1}J^{m}_{T_{\alpha}}}(s)
 \frac{\sqrt{hs}}{\sqrt{T_{\alpha}-hs}}\quad,\\
m\leq 1&&
 \|B_{2m}^{*}\|_{L^{2}_{z}L^{r'_{\sigma}}_{y_{G}}\leftarrow L^{2}_{z}L^{2}_{y_{G}}}
 \lesssim
 \frac{M_{\alpha 01}\|V_{1}\|_{L^{r'_{\sigma}}}}{
 \sqrt{\alpha'_{0}-\alpha}}2^{m/2}
 \lesssim
   \frac{M_{\alpha 01}C_{V}\gamma^{1/2}}{T_{\alpha}^{1/2}}2^{m/2}\,,\\
m\geq 2 &&B_{1nm}=2^{-\frac{n}{2}}e^{\alpha
     N}a_{G}(V_{2})e^{-\frac{\alpha+\alpha'_{m}}{2}N}\quad,\\
m\geq 2&&
     \|B_{1nm}\|_{L^{2}_{z}L^{2}_{y_{G}}\leftarrow L^{2}_{z}L^{r_{\sigma}}_{y_{G}}}\lesssim
     \frac{M_{\alpha 01}\|V_{2}\|_{L^{r'_{\sigma}}}}{\sqrt{\alpha_{m}'-\alpha}}2^{-n/2}\lesssim
     \frac{M_{\alpha 01}C_{V}\gamma^{1/2}}{T_{\alpha_{1}}^{1/2}}2^{m/2-n/2}\,
\\
m\geq 2&&B_{2nm}^{*}=e^{\frac{\alpha+\alpha'_{m}}{2}N}a_{G}^{*}(V_{1})e^{-\alpha'_{m}N}1_{h^{-1}J^{m}_{T_{\alpha}}}(s)
 \frac{\sqrt{hs}}{\sqrt{T_{\alpha}-hs}}\quad,\\
\\
m\geq 2&&
 \|B_{2mn}^{*}\|_{L^{2}_{z}L^{r'_{\sigma}}_{y_{G}}\leftarrow L^{2}_{z}L^{2}_{y_{G}}}
 \lesssim
 \frac{M_{\alpha 01}\|V_{1}\|_{L^{r'_{\sigma}}}}{
 \sqrt{\alpha'_{m}-\alpha}}2^{m/2}
 \lesssim
   \frac{M_{\alpha 01}C_{V}\gamma^{1/2}}{T_{\alpha}^{1/2}}2^{m}\,,\\
\\
&&
   \varphi_{m}(s)=1_{h^{-1}J^{m}_{T_{\alpha}}}(s)\varphi_{\infty}(s)\quad,\\
&&
   \varphi_{\infty}(s)=1_{h^{-1}J^{\leq 1}_{T_{\alpha}}}(s)e^{\alpha'_{0}N}\frac{\sqrt{T_{\alpha}-hs}}{\sqrt{hs}}u_{\infty}(s)+\sum_{m=2}^{\infty}e^{\alpha'_{m}N}1_{h^{-1}J^{m}_{T_{\alpha}}}(s)\frac{\sqrt{T_{\alpha}-hs}}{\sqrt{hs}}u_{\infty}(s)\,.
\end{eqnarray*}
The size of the intervals are estimated respectively by
$|h^{-1}J^{n}_{T^{\alpha}}|\lesssim h^{-1}2^{-n}T_{\alpha}$ and
$$
|h^{-1}J^{m}_{T_{\alpha}}|\lesssim
h^{-1}2^{m}T_{\alpha}\quad\text{for}~m\leq 1\quad,\quad
|h^{-1}J^{m}_{T_{\alpha}}|\lesssim
h^{-1}2^{-m}T_{\alpha}\quad\text{for}\quad m\geq 2\,.
$$
Proposition~\ref{pr:L2L1Linf} gives
\begin{align*}
\frac{\|\varphi_{1}\|_{L^{1}(h^{-1}J^{n}_{T_{\alpha}};L^{2}_{z}L^{2}_{y_{G}})}}{
\|\varphi_{\infty}\|_{L^{\infty}(h^{-1}J^{\leq n}_{T_{\alpha}};L^{2}_{z}L^{2}_{y_{G}})}}
\lesssim&\left[ \sum_{-\infty\leq m\leq 1}2^{-n/2}(M_{\alpha
          01}C_{V}\gamma^{1/2}2^{-n/2})(M_{\alpha 01}C_{V}\gamma^{1/2}2^{m/2})2^{m/2}\right]
\\
&+\left[\sum_{m=2}^{n}2^{-n/2}(M_{\alpha
  01}C_{V}\gamma^{1/2}2^{m/2-n/2})(M_{\alpha 01}C_{V}\gamma^{1/2}2^{m})2^{-m/2}\right]
\\
&\lesssim M_{\alpha 01}C_{V}^{2}\gamma\,.
\end{align*}
With 
$$
\|\varphi_{\infty}\|_{L^{\infty}(h^{-1}J^{\leq  n}_{T_{\alpha}};L^{2}_{z}L^{2}_{y_{G}})}\leq 
\|\varphi_{\infty}
\|_{L^{\infty}([0,T_{\alpha}/h[;L^{2}_{z}L^{2}_{y_{G}})}\leq M_{\infty}(u_{\infty})\,,
$$
we have proved
\begin{equation}
  \label{eq:L1inf2}
\sup_{n\geq 1}2^{-n/2}\|e^{\alpha
  N}L_{1\infty}(u_{\infty})\|_{L^{1}(h^{-1}J^{n}_{T_{\alpha}};L^{2}_{z}L^{2}_{y_{G}})}\lesssim M_{\alpha 01}C_{V}^{2}\gamma M_{\infty}(u_{\infty})\,.
\end{equation}
\noindent\textbf{Conclusion.}
From \eqref{eq:estimI}, \eqref{eq:estimII} and \eqref{eq:estimLInfty1} we deduce
\begin{equation}
  \label{eq:estimMInfty}
M_\infty(L_{\infty\,\infty}(u_\infty)+L_{\infty2}(u_2)+L_{\infty1}(u_1))\lesssim \gamma^{1/2}
M_{\alpha01}C_{V}M(u_\infty,u_2,u_1)\,. 
\end{equation}
Combining~\eqref{eq:Bnm}, \eqref{eq:B00}, \eqref{eq:Bn0}, \eqref{eq:Bnm1}, and taking the supremum over $\alpha\in[\alpha_0,\alpha_1[$ yields 
\[M_1(L_{11}(u_1))\lesssim \gamma^{1/2}
M_{\alpha01}C_{V}M_1(u_1)\,,\]
while \eqref{eq:EstimL22} says
$$
M_{2}(L_{22}(u_{2}))\lesssim C_{V}M_{\alpha 01}\sqrt{\gamma}M_{2}(u_{2})\,.
$$
Finally the upper bounds \eqref{eq:L1inf1},\eqref{eq:L1inf2} combined,
firstly 
with the equivalence of norms $N_{11}$ and $N_{31}$\,, and secondly the
normalization of \eqref{eq:defM1u} of $M_{1}$ yields
\[M_1(L_{1\infty}(u_\infty))\lesssim \gamma^{1/2}
M_{\alpha01}C_{V}M(u_\infty,u_2,u_1)\,.\]
The sum of all those inequalities is
$$
M(L(u_\infty,u_2,u_1))\lesssim \gamma^{1/2}M_{\alpha01}C_{V}M(u_\infty,u_2,u_1)\,,
$$
which means that there exists a constant $C_{d,U}$ determined by the
dimension $d$ and the free dynamics $U$ such that
$$
\|L\|_{\mathcal{L}(\mathcal{E}_{\alpha_{0},\alpha_{1},\gamma}^{h})}\leq C_{d,U}M_{\alpha01}C_{V}\gamma^{1/2}\,.
$$ 
\end{proof}

\begin{lemma}
\label{le:equivnorms2}
The norms $N_{p,1},N_{p,2},N_{p,3},N_{p,4}$ defined in
\eqref{eq:defN1_2}\eqref{eq:defN2_2}\eqref{eq:defN3_2}\eqref{eq:defN4_2}
for $p=2$ (resp.
\eqref{eq:defN11}\eqref{eq:defN12}\eqref{eq:defN13}\eqref{eq:defN14}
for $p=1$)
are equivalent according to \eqref{eq:equivN21toN24}
(resp. \eqref{eq:equivN11toN15}). 
\end{lemma}
\begin{proof}
We forget the notation $L^{2}_{z}L^{2}_{y_{G}}$ 
because it is a time integration issue and it can be done with any Banach
space valued functions.\\
With the  Definition \eqref{eq:defN3_2} of $N_{2,3}(\varphi)$\,, the equality
$$
\|\varphi\|_{L^{2}(h^{-1}J^{\leq 1}_{T})}
=\left(\|\varphi\|^{2}_{L^{2}(h^{-1}J^{0}_{T})}
+\|\varphi\|^{2}_{L^{2}(h^{-1}J^{1}_{T})}\right)^{1/2}
$$
allows to replace $N_{2,3}(\varphi)$ by the equivalent norm
$$
\sqrt{T}\|\varphi\|_{L^{2}(h^{-1}J^{\leq 1}_{T})}+\sqrt{T}\sup_{n>1}2^{-n/2}\|\varphi\|_{L^{2}(h^{-1}J^{n}_{T})}
$$
For $p=1$\,, the inequality
$$
\forall t\in [\frac{3T}{4h},\frac{T}{h}[,\quad
\frac{1}{T}\leq \frac{1}{ht}\leq \frac{4}{3T} 
$$
allows to replace the second term of the definitions \eqref{eq:defN12}
\eqref{eq:defN13}\eqref{eq:defN14} of $N_{1,2}$\,, $N_{1,3}$ and
$N_{1,4}$\,, respectively by
\begin{eqnarray*}
  &&\sup_{\delta\in
     ]0,T/8]}\sqrt{\delta}\left\|\frac{\varphi}{\sqrt{ht}}\right\|_{L^{1}(h^{-1}[T-2\delta,T-\delta[)}\,\\
&&
\sup_{n>1}\sqrt{T}2^{-n/2}\left\|\frac{\varphi}{\sqrt{ht}}\right\|_{L^{1}(h^{-1}J^{n}_{T})}\,,\\
&&
\sup_{\delta\in ]0,T/12]}\sqrt{\delta}\left\|\frac{\varphi}{\sqrt{ht}}\right\|_{L^{1}(h^{-1}[T-3\delta,T-\delta[)}\,.
\end{eqnarray*}
Additionally the $\sup_{\tau\in [0,T[}$ in the definitions
\eqref{eq:defN1_2}\eqref{eq:defN11} can be replaced by $\sup_{\tau\in
  [3T/4,T[}$\,.
We are thus led to compare the norms, for $p=1,2$\,,
\begin{align*}
 N_{p,1,T,h}(\varphi)&= \sup_{\tau\in
                         [3T/4,T[}\sqrt{T-\tau}\left\|\frac{\varphi}{(ht)^{1/p-1/2}}\right\|_{L^{p}([0,\frac{\tau}{h}])} \,, \\
N_{p,2,T,h}(\varphi)&=\sqrt{T}
\left\|\frac{\varphi}{(ht)^{1/p-1/2}}\right\|_{L^{p}(h^{-1}J^{\leq
                        1}_{T})}+\sup_{\delta\in]0,T/8]}\sqrt{\delta}
\left\|\frac{\varphi}{(ht)^{1/p-1/2}}\right\|_{L^{p}(h^{-1}[T-2\delta,T-\delta])}\,,\\
N_{p,3,T,h}(\varphi)&=
\sqrt{T}\left\|\frac{\varphi}{(ht)^{1/p-1/2}}\right\|_{L^{p}(h^{-1}J^{\leq
                        1}_{T})}
+\sqrt{T}\sup_{n>1}2^{-n/2}\left\|\frac{\varphi}{(ht)^{1/p-1/2}}\right\|_{L^{p}(h^{-1}J^{n}_{T})}\,,\\
N_{p,4,T,h}(\varphi)&=
\sqrt{T}
\left\|\frac{\varphi}{(ht)^{1/p-1/2}}\right\|_{L^{p}(h^{-1}J^{\leq
                        2}_{T})}+\sup_{\delta\in]0,T/12]}\sqrt{\delta}\left\|\frac{\varphi}{(ht)^{1/p-1/2}}\right\|_{L^{p}(h^{-1}[T-3\delta,T-\delta])}\,.\\
\end{align*}
The elementary homogeneity of those expressions gives
$$
N_{p,i,T,h}(\varphi)=\frac{T}{h^{1/p}}N_{p,i,1,1}(\tilde{\varphi})\quad\text{with}~\tilde{\varphi}(t)=\varphi(ht)\quad\text{for}~p=1,2\quad\text{and}~1\leq
i\leq4\,,
$$
and it suffices to consider the case $T=h=1$ while setting $\psi=\frac{\tilde{\varphi}}{t^{1/p-1/2}}$\,.\\
For $\tau\in [3/4,1[$ the identity
$$
\left\|\psi\right\|_{L^{p}([0,\tau])}
=\left(\left\|\psi\right\|^{p}_{L^{p}(J^{\leq
      1}_{1})}
+\left\|\psi\right\|^{p}_{L^{p}([3/4,\tau])}
\right)^{1/p}
$$
reduces the comparison of $N_{p,1,1,1}(\tilde{\varphi})$\,, $N_{p,2,1,1}(\tilde{\varphi})$ and
$N_{p,3,1,1}(\tilde{\varphi})$ to the comparison of
\begin{align*}
A_{1}(\psi)&=\sup_{\tau\in [3/4,1[}\sqrt{1-\tau}\left\|\psi\right\|_{L^{p}([3/4,\tau])}\,,\\
A_{2}(\psi)&=\sup_{\delta\in]0,1/8]}\sqrt{\delta}
\left\|\psi\right\|_{L^{p}[1-2\delta,1-\delta])}\,,\\
A_{3}(\psi)&=\sup_{n>1}2^{-n/2}\left\|\psi\right\|_{L^{p}(J^{n}_{1})}\,.
\end{align*}
Taking $\tau=1-\delta$\,, $\delta\leq 1/8$\,, in $A_{1}(\psi)$ and $\delta=2^{-n-1}$\,, $n>1$\,, in $A_{2}(\psi)$ gives
$$
A_{3}(\psi)\leq \sqrt{2}A_{2}(\psi)\leq\sqrt{2} A_{1}(\psi)\,.
$$
For $\tau\in[3/4,1[$ there exists $n_{\tau}>1$ such that $\tau\in
[1-2^{-n_{\tau}},1-2^{-n_{\tau}-1}[=J^{n_{\tau}}_{1}$ and
$$
\|\psi\|^{p}_{L^{p}([3/4,\tau])}=\sum_{n=2}^{n_{\tau}}\|\psi\|^{p}_{L^{p}(J^{n}_{1})}\leq
\sum_{n=2}^{n_{\tau}}2^{np/2}(2^{-n/2}\|\psi\|_{L^{p}(J^{n}_{1})})^{p}
\leq\frac{2^{p(n_{\tau}+1)/2}}{2^{p/2}-1}A_{3}(\psi)^{p}\,.
$$
The inequality
$$
(1-2^{-n_{\tau}})\leq\tau\quad\text{or}\quad  \sqrt{1-\tau}\leq 2^{-n_{\tau}/2}\,,
$$
while taking the supremum over $\tau\in [3/4,1[$\,, implies
$$
A_{1}(\psi)\leq \frac{\sqrt{2}}{(2^{p/2}-1)^{1/p}} A_{3}(\psi)\,.
$$
We have proved the equivalence 
$$
\kappa_{p,1}^{-1}N_{p,1}(\varphi)\leq N_{p,i}(\varphi)\leq \kappa_{p,1}N_{p,1}(\varphi)\quad\text{for}~p=1,2,~i=2,3\,,
$$
with a universal constant $\kappa_{p,1}>1$\,.\\
It now suffices to compare $N_{p,2}$ and $N_{p,4}$ or equivalently
$N_{p,2,1,1}(\tilde{\varphi})$ and $N_{p,4,1,1}(\tilde{\varphi})$
written with $\psi=\frac{\tilde{\varphi}}{t^{1/p-1/2}}$
\begin{align*}
 N_{p,2,1,1}(\tilde{\varphi}) &=\|\psi\|_{L^{p}(J^{\leq
                                  1}_{1})}+\sup_{\delta\in
                                  ]0,1/8]}\sqrt{\delta}\|\psi\|_{L^{p}([1-2\delta,1-\delta])}=:B_{2}(\psi) \,, \\
N_{p,4,1,1}(\tilde{\varphi})&=
\|\psi\|_{L^{p}(J^{\leq 2}_{1})}+\sup_{\delta\in ]0,1/12]}\sqrt{\delta}\|\psi\|_{L^{p}([1-3\delta,1-\delta])}=:B_{4}(\psi)\,.
\end{align*}
For the first terms of $B_{2}(\psi)$ and $B_{4}(\psi)$\,,
$$
\|\psi\|^{p}_{L^{p}(J^{\leq
                                  1}_{1})}\leq
\|\psi\|^{p}_{L^{p}(J^{\leq
                                  2}_{1})}
=
\|\psi\|^{p}_{L^{p}(J^{\leq 1}_{1})}+
\|\psi\|^{p}_{L^{p}(J^{2}_{1})}
$$
gives
$$
\|\psi\|_{L^{p}(J^{\leq
                                  1}_{1})}\leq
\|\psi\|_{L^{p}(J^{\leq
                                  2}_{1})}
\leq 
\|\psi\|_{L^{p}(J^{\leq 1}_{1})}+\sup_{\delta\in]0,1/8]}\|\psi\|_{L^{p}([1-2\delta,1-\delta])}\,.
$$
For the second terms of $B_{2}(\psi)$ and $B_{4}(\psi)$\,,
$$
\sqrt{\delta}\|\psi\|_{L^{p}([1-3\delta,1-3\delta/2])}\leq
\sqrt{\delta}\|\psi\|_{L^{p}([1-3\delta,1-\delta])}\leq
\sqrt{\delta}\|\psi\|_{L^{p}([1-3\delta,1-3\delta/2])}+
\sqrt{\delta}\|\psi\|_{L^{p}([1-2\delta,1-\delta])}
$$
leads to 
$$
(2/3)^{1/2}
\sup_{\delta\in
  ]0,1/8]}\sqrt{\delta}\|\psi\|_{L^{p}([1-3\delta,1-\delta])}\leq
\sup_{\delta\in
  ]0,1/12]}\sqrt{\delta}\|\psi\|_{L^{p}([1-2\delta,1-\delta])}
\leq 2\sup_{\delta\in ]0,1/8]}\sqrt{\delta}\|\psi\|_{L^{p}([1-3\delta,1-\delta])}\,.
$$
Adding the two terms yields the equivalences
\begin{eqnarray*}
  && \kappa_{p,2}^{-1}B_{2}(\psi)\leq B_{4}(\psi)\leq \kappa_{p,2}
B_{2}(\psi)\\
\text{and}
&&
\kappa_{p,2}^{-1}N_{2,2}(\varphi)\leq N_{2,4}(\varphi)\leq \kappa_{p,2}N_{2,2}(\varphi)
\end{eqnarray*}
for a universal constant $\kappa_{p,2}>1$\,. 
The proof ends by taking $\kappa_{p}=\kappa_{p,1}\kappa_{p,2}>1$\,.
\end{proof}

\section{Consequences of Strichartz estimates for our model problem }
\label{sec:conseqStri}
The general results of
Section~\ref{sec:StriCM} are applied to our model problem presented in
Subsection~\ref{sec:ourproblem}.

\subsection{Validity of the general hypotheses and main result}
\label{sec:valid}
Let us consider \eqref{eq:dynaleaFou} 
\begin{equation}
  \label{eq:dynaleaFou2}
\begin{cases}
i\partial_{t}\hat{f}^{h}=(\xi-d\Gamma(D_{y}))^{2}\hat{f}^{h}+\sqrt{h}[a(V)+a^{*}(V)]\hat{f}\,,\\
\hat{f}^{h}(t=0)=\hat{f}^{h}_{0}\,,
\end{cases}
\end{equation}
where $\hat f^{h}(t)\in L^{2}(\rz^{d}\times Z'', \frac{d\xi}{(2\pi)^{d}}\otimes \mathbf{dz''};
\Gamma(L^{2}(\rz^{d},dy;\cz)))$\,, $\xi$ is the Fourier variable of
$x\in \rz^{d}$ and $z''\in Z''$ is a parameter,
e.g. $L^{2}(Z'',\mathbf{dz''})=L^{2}(\rz^{d},\frac{d\xi}{(2\pi)^{d}};\cz)\otimes
\Gamma(L^{2}(\rz^{d},dy;\cz))$ when we want to
handle the evolution of Hilbert-Schmidt operators on
$L^{2}(\rz^{d},\frac{d\xi}{(2\pi)^{d}};\cz)\otimes
\Gamma(L^{2}(\rz^{d},dy;\cz))$ as described in the end of
Section~\ref{sec:ourproblem}. 
Our complete parameter is thus
$$
z'=(\xi,z')\in \rz^{d}\times Z''=Z'
$$
and remember the writing introduced in Definition~\ref{de:LpxLqy} and
specified in 
\eqref{eq:L2tzL2z} and \eqref{eq:F2UtUs}
\begin{eqnarray*}
  &&
L^{2}(\rz^{d}\times
Z'',\frac{d\xi}{(2\pi)^{d}}\otimes\mathbf{dz};\Gamma(L^{2}(\rz^{d}, dy;\cz)))
=L^{2}_{z,\text{sym}}L^{2}_{y_{G}}=\underbrace{L^{2}_{z_{0}}}_{\text{vacuum}}\oplus L^{2}_{z_{1},\text{sym}}L^{2}_{y_{G}}.
\\
\text{with}&& Z_{0}=Z'\quad,\quad Z_{1}=\mathcal{R}\otimes Z'\,,
\end{eqnarray*}
where the subscript $_{\text{sym}}$ refers to the symmetry for the
relative coordinate variable $Y'\in \mathcal{R}$\,.\\
Using the center of mass variable (see Section~\ref{sec:centermass}) by setting $t\mapsto u_{G}^{h}(t)=U_{G}^{-1}\hat{f}^{h}(t)$, \eqref{eq:dynaleaFou2} becomes
\begin{equation}
   \label{eq:dynaleauG}
 \begin{cases}
       i\partial_{t}u_{G}^{h} = (\xi-D_{y_{G}})^{2}u_{G}^{h}+\sqrt{h}[a^{*}_{G}(V)+a_{G}(V)]u_{G}^{h}\,,\\
 u_{G}^{h}(t=0) = u_{G,0}^{h}\,.
\end{cases}
\end{equation}
In this context, the free dynamics $U(t)$ involved in \eqref{eq:UtUs} acts simply on $L^{2}_{z}L^{2}_{y_{G}}$ and equals
$$
U(t)=K_{0}(t,z')\oplus U_{1}(t,Y',z')=
(e^{-it|\xi|^{2}}\times_{z'})\oplus (e^{-it(\xi-D_{y_{G}})^{2}}\times_{(Y',z')})
$$
where we recall  $z_{0}=z'\in Z'$ and $z_{1}=(Y',z')\in \mathcal{R}\times Z'$\,.
Because $\|e^{\pm
  i\xi\cdot y}\varphi\|_{L^{p}_{y}}=\|\varphi\|_{L^{p}_{y}}$ for all $1\leq
p\leq +\infty$ and $e^{-itD_{y}^{2}}=e^{it\Delta_{y}}$ satisfies 
\begin{eqnarray*}
  && \|e^{it\Delta_{y}}f\|_{L^{2}_{y}}\leq \|f\|_{L^{2}_{y}}\,,\\
&& \|e^{it\Delta_{y}}(e^{is\Delta_{y}})^{*}g\|_{L^{\infty}_{y}}
=\|e^{i(t-s)\Delta_{y}}g\|_{L^{\infty}_{y}}\leq
   \frac{\|g\|_{L^{1}_{y}}}{(4\pi)^{d/2}|t-s|^{d/2}}\quad t\neq s\,,
\end{eqnarray*}
the assumption~\eqref{eq:hypStri1}\eqref{eq:hypStri2} are satisfied for
$U_{1}(t,z_{1})$\,, $z_{1}=(Y',z')$\,, as soon as $d\geq 3$ with $\sigma=\frac{d}{2}>1$\,,
uniformly with respect to $z_{1}\in \mathcal{R}\times Z'$\,.\\
The dense  subset $D$ in $L^{2}_{z_{1}}L^{2}_{y_{G}}$ such that
$D\subset
L^{2}(Z_{1},\mathbf{dz_{1}};L^{r_{\sigma}}(\rz^{d},dy_{G};\cz))$\,, with
$r^{\sigma}=\frac{2d}{d-2}$ and $d\geq 3$ here,  is simply 
$D=L^{2}(Z_1,dz_1;H^{\mu}(\rz^{d}))$ with $\mu>d/2$\,. Remember that
the dense subset $D$ was introduced in Proposition~\ref{pr:StriDT} for
the dense \emph{a priori} definition of the operator $A_{T}$ on
$L^{1}([0,T];L^{2}_{z,x})$ (see Remark~\ref{re:DL1} and the proof of Proposition~\ref{pr:StriDT}).\\
Below are reviewed assumptions on $V$:
\begin{enumerate}
\item If $V\in L^{\frac{2d}{d+2}}(\rz^{d},dy;\rz)$\,, the
  assumptions of Proposition~\ref{pr:contracStri} are satisfied with
  $C_{V}=1+\|V\|_{L^{\frac{2d}{d+2}}}>0$ and $r'_{\sigma}=\frac{2\sigma}{\sigma+1}=\frac{2d}{d+2}$\,.
\item If $V\in H^{2}(\rz^{d};\rz)$ then \eqref{eq:dynaleaFou2} (or~\eqref{eq:dynaleauG}) defines a  unitary dynamics with a rather well understood domain of its generator in $L^{2}(\rz^{d}\times Z'', \frac{d\xi}{(2\pi)^{d}}\otimes \mathbf{dz''};
\Gamma(L^{2}(\rz^{d},dy;\cz)))\simeq L^{2}_{z}L^{2}_{y_{G}}$\,.
\end{enumerate}
We will always assume~$V\in L^{r_\sigma^\prime}$ in the sequel, and depending on the statement we might assume that~$V\in H^2$ or not.

If $V\in H^{2}(\rz^{d};\rz)$, the unique solution $t\mapsto u_{G}^{h}(t)=U_{G}^{-1}\hat{f}^{h}(t)\in
\mathcal{C}^{0}(\rz;L^{2}_{z}L^{2}_{y_{G}})$ to~\eqref{eq:dynaleauG} satisfies
\begin{equation}
  \label{eq:uGmild}
u^{h}_{G}(t)=U(t)u^{h}_{G,0}-i\int_{0}^{t}U(t)U(s)^{*}\sqrt{h}[a_{G}^{*}(V)+a_{G}(V)]u^{h}_{G}(s)~ds\,,
\end{equation}
and we will now seek for a solution of this equation using the fixed point method developed in Section~\ref{sec:fixweight}, for~$V\in L^{r_\sigma^\prime}$ but not necessarily~$V\in H^{2}(\rz^{d};\rz)$.

If $(u^h_{\infty},u^h_{2},u^h_{1})$  solves
\begin{alignat}{2}
  \label{eq:dynuinfty_expl}
  u^{h}_\infty(t)&=
-i\int_{0}^{t}U(t)U^{*}(s) \left(\sqrt{h}a_{G}^{*}(V)u^{h}_\infty(s)+\sqrt{h} u^{h}_2(s)+u^h_1(s)\right)\,ds &+f^{h}_\infty(t)\,, \\
u_2^{h}(t)&= 
-i\int_{0}^{t} a_{G}(V)U(t)U(s)^{*} \sqrt{h} u_2^{h}(s)~ds 
&+f_2^{h}(t) 
\label{eq:dynu2_expl} \,,\\
u_1^{h}(t)&=
 -i\int_{0}^{t} a_{G}(V)U(t)U(s)^{*} \left( h a_{G}^{*}(V)
 u_\infty^{h}(s) + \sqrt{h} u_1^{h}(s) \right) ~ds 
 \,,& 
\label{eq:dynu1_expl}
\end{alignat}
with
\begin{align}
\label{eq:rhsfh}
f^{h}_\infty(t) =-i&\int_{0}^{t}U(t)U(s)^{*}
a_{G}^{*}(V) \sqrt{h}U(s)u_{G,0}^h~ds \,,\\
\label{eq:rhsgh}
f^{h}_2(t)=-i\,a_{G}(V)&\int_{0}^{t}U(t)U(s)^{*}a_{G}^{*}(V) \sqrt{h}U(s)u_{G,0}^h~ds + a_G(V) U(t) u^h_{G,0}\,,
\end{align}
written shortly as
\begin{equation}
            \begin{pmatrix}
              u_\infty^{h}\\u_2^{h}\\u_1^h
            \end{pmatrix}
=L
            \begin{pmatrix}
              u_\infty^{h}\\u_2^{h}\\u_1^h
            \end{pmatrix}
+
            \begin{pmatrix}
              f_\infty^{h}\\f_2^{h}\\0
            \end{pmatrix}
\,,\quad L=
  \begin{pmatrix}
    L_{\infty\,\infty}&L_{\infty2}&L_{\infty1}\\
    0&L_{22}&0\\
    L_{1\infty}&0&L_{11}
  \end{pmatrix}
\,, \label{eq:dynutvtexp}
\end{equation}
then $u^h_G(t) = u_\infty(t) + U(t) u^h_{G,0}$ will yield a solution to~\eqref{eq:uGmild}. 

Actually, with $u^h_G(t) = u^h_\infty(t) + U(t) u^h_{G,0}$, applying $a_G(V)$ to~\eqref{eq:dynuinfty_expl} on the one hand, and summing~$\sqrt{h}\times$\eqref{eq:dynu2_expl} and~\eqref{eq:dynu1_expl} on the other hand yields $\sqrt{h} a_G(V) u_G^h = u_1^{h}+ \sqrt{h}u_2^{h}$, which inserted in~\eqref{eq:dynuinfty_expl} provides~\eqref{eq:uGmild}.

\begin{theorem}
\label{th:applStri}
Assume $d\geq 3$ and $V\in L^{\frac{2d}{d+2}}(\rz^{d},dy;\rz)$ with 
$$
\|V\|_{L^{\frac{2d}{d+2}}}<C_{V}\,.$$
Assume that there exists $\alpha_{1}>0$ and $C_{\alpha_{1}}>0$ such that 
$$
\forall h\in ]0,h_{0}[\,,\quad \|e^{\alpha_{1}N}u^{h}_{G,0}\|_{L^{2}_{z}L^{2}_{y_{G}}}\leq C_{\alpha_{1}}\,.
$$
There exists a constant $C_{d}>0$ depending on the dimension $d\geq
3$\,, such that when $\gamma>0$ is chosen such that
$$
2\|L\|_{\mathcal{L}(\mathcal{E}^h_{-\alpha_{1},\alpha_{1},\gamma})}\leq C_{d}e^{\alpha_{1}}C_{V}\gamma^{1/2}\leq 1\,,
$$
the function $u^h_G(t) = u^h_\infty(t) + U(t) u^h_{G,0}$ with $(u^h_\infty,u^h_2,u^h_1)$ the unique solution to \eqref{eq:dynutvtexp} in $(\mathcal{E}_{-\alpha_{1},\alpha_{1},\gamma}^{h},M)$ satisfies
\begin{align}
 \label{eq:estLinf}
\forall t\in I_{T_{\alpha}}^{h}\,,\quad
\left\|e^{\alpha N}[u_{G}^{h}(t)-U(t)u_{G,0}^{h}]\right\|_{L^{2}_{z}L^{2}_{y_{G}}}
\leq C_{d}C_{V}e^{\alpha_{1}}C_{\alpha_{1}}\gamma^{1/2}\frac{\sqrt{|ht|}}{\sqrt{T_{\alpha}-|ht|}}\,,
\end{align}
with
\begin{equation}
T_{\alpha}=\gamma(\alpha_{1}-\alpha)
\end{equation}
for all $\alpha\in [0,\alpha_{1}[$ and all $h\in]0,h_{0}[$\,.

If, moreover, $V\in H^{2}(\rz^{d};\rz)$, then
$u_{G}^{h}$ is the only solution to $\eqref{eq:dynaleauG}$ in 
$\mathcal{C}^{0}(I_{T_{0}}^{h};L^{2}_{z}L^{2}_{y_{G}})$\,.
\end{theorem}
\begin{proof}
  We take $\alpha_{0}=-\alpha_{1}$ where $\alpha_{1}>0$ is fixed. The
  constant $M_{\alpha01}$ of Definition~\ref{de:E0Talpha} is nothing
  but
$$
M_{\alpha01}=\frac{e^{\alpha_{1}}}{2}\,.
$$
Accordingly to Definition~\ref{de:E0Talpha}, for a fixed $\gamma>0$ the time scale $T_{\alpha}$ is given by
$T_{\alpha}=\gamma(\alpha_{1}-\alpha)$ for all $\alpha\in
[-\alpha_{1},\alpha_{1}[$\,.
Proposition~\ref{pr:contracStri} tells us that the condition
$$
C_{d,U}\frac{e^{\alpha_{1}}}{2}C_{V}\gamma^{1/2}\leq \frac{1}{2}
$$
where $C_{d,U}=C_{d}$ is determined by the dimension $d\geq 3$ here,
ensures that the operator $L$ is a contraction in
$\mathcal{E}_{-\alpha_{1},\alpha_{1},\gamma}^{h}$ for all $h\in ]0,h_{0}[$:
$$
\|L\|_{\mathcal{L}(\mathcal{E}_{-\alpha_{1},\alpha_{1},{\gamma}}^{h})}\leq \frac{1}{2}\,.
$$
If $M(f_\infty^{h},f_2^{h},0)<\infty$, then the system \eqref{eq:dynutvtexp} admits a unique
solution in $\mathcal{E}_{-\alpha_{1},\alpha_{1},\gamma}^{h}$ for all $h\in
]0,h_{0}[$ with 
$$
M(u_\infty^{h},u_2^{h},u_1^{h})\leq 2M(f_\infty^{h},f_2^{h},0) \,.
$$
It remains to check two things:
\begin{itemize}
\item the right-hand side $(f_\infty^{h},f_2^{h},0)$ given by
\eqref{eq:rhsfh}\eqref{eq:rhsgh} belongs to
$\mathcal{E}_{-\alpha_{1},\alpha_{1},\gamma}^{h}$ and to estimate
$M(f_\infty^{h},f_2^{h},0)$\,;
\item the unique solution $(u_\infty^{h},u_2^{h},u_1^{h})$ to \eqref{eq:dynutvtexp}
  yields after setting $u_{G}^{h}(t)=u_\infty^{h}(t)+U(t)u_{G,0}^{h}$
  the unique solution to \eqref{eq:dynalea} in $\mathcal{C}^{0}(]-\frac{T_{0}}{h},\frac{T_{0}}{h}[;L^{2}_{z,y_{G}})$\,.
\end{itemize}
The first step is simpler than what we did for
Proposition~\ref{pr:contracStri}. Let us start with
\begin{align*}
e^{\alpha
  N}f^{h}_{\infty}(t)=&-i\int_{0}^{t}U(t)U(s)^{*}e^{\alpha
  N}\sqrt{h}a_{G}^{*}(V)e^{-\alpha_{1}N}U(s)e^{\alpha_{1}N}u_{G,0}^{h}~ds\,\\
=&\int_{0}^{t}U(t)U(s)^{*}F(s)~ds
\end{align*}
with $F(s)=-i1_{[0,t]}(s)e^{\alpha
  N}\sqrt{h}a_{G}^{*}(V)e^{-\alpha_{1}N}U(s)e^{\alpha_{1}N}u_{G,0}^{h}$\,.
By Proposition~\ref{pr:expaLp} we know that
$$
\|F\|_{L^{2}_{s}L^{2}_{z_{1}}L^{r'_{\sigma}}_{y_{G}}}\leq \frac{C_{V}e^{\alpha_{1}}}{2\sqrt{\alpha_{1}-\alpha}}|ht|^{1/2}\|e^{\alpha_{1}N}u_{G,0}^{h}\|_{L^{2}_{z}L^{2}_{y_{G}}}
\leq \frac{C_{V}e^{\alpha_{1}}\gamma^{1/2}}{\sqrt{T_{\alpha}}}C_{\alpha_{1}}|ht|^{1/2}\,.
$$
A direct application of the retarded 
endpoint Strichartz estimate \eqref{eq:retardStri} 
yields
$$
\left(\frac{T_{\alpha}-|ht|}{|ht|}\right)^{1/2}\|e^{\alpha
    N}f^{h}_{\infty}(t)\|_{L^{2}_{z}L^{2}_{y_{G}}}
\lesssim C_{V}e^{\alpha_{1}}C_{\alpha_{1}}\gamma^{1/2}\,.
$$
and by taking the supremum over $|ht|<T_{\alpha}$\,,
\begin{equation}
  \label{eq:estimfh}
M(f_{\infty}^{h},0)\lesssim C_{V}e^{\alpha_{1}}C_{\alpha_{1}}\gamma^{1/2}\,.
\end{equation}
For 
\[
f^{h}_{2,1}(t)=-i\,a_{G}(V)\int_{0}^{t}U(t)U(s)^{*}a_{G}^{*}(V) \sqrt{h}U(s)u_{G,0}^h~ds\,,\]
the Proposition~\ref{pr:expaLp} and the retarded Strichartz
estimate~\eqref{eq:homStri} give
\begin{align*}
\sqrt{T_\alpha-\tau} &\| e^{\alpha N}f^{h}_{2,1}\|_{L^2_t(I^h_\tau;L^2_{z}L^{2}_{y_G})}\\
&\lesssim
\sqrt{T_\alpha-\tau} \frac{C_{V}e^{\alpha_1}}{\sqrt{\alpha_1 - \alpha}} 
\left\|\int_{0}^{t} U(t)U(s)^{*}e^{\frac{\alpha+\alpha_{1}}{2}N}a_{G}^{*}(V)
\sqrt{h} U(s) u_{G,0}^h~ds \right\|_{L^2_{z_{1}}L^2_t(I^h_\tau;L^{r_\sigma}_{y_G})}\\
&\lesssim
\sqrt{T_\alpha-\tau} \frac{C_{V}e^{\alpha_1}}{\sqrt{\alpha_1 - \alpha}} 
\left\|e^{\frac{\alpha+\alpha_{1}}{2}N}a_{G}^{*}(V)
\sqrt{h} U(s) u_{G,0}^h \right\|_{L^2_{z_{1}}L^2_s(I^h_\tau;L^{r'_\sigma}_{y_G})}\,,
\end{align*}
where here $r'_{\sigma}=\frac{2d}{d+2}$ and $r_{\sigma}=\frac{2d}{d-2}$\,.\\
Then using Proposition~\ref{pr:expaLp} again, the square integrability
of 1 on $I^h_\tau$ and 
the boundedness of $U(s)$ in the $L^2$ norm,
\begin{align*}
\sqrt{T_\alpha-\tau} \| e^{\alpha N}f^{h}_{2,1}\|_{L^2_t(I^h_\tau;L^2_{z}L^{2}_{y_G})}
&\lesssim
\sqrt{T_\alpha-\tau} \frac{C_{V}^2 e^{2\alpha_1}}{\alpha_1 - \alpha} 
\left\|e^{\alpha_{1}N}
\sqrt{h} U(s) u_{G,0}^h \right\|_{L^2_{z}L^2_s(I^h_\tau;L^{2}_{y_G})}\\
&\lesssim
C_{V}^2 e^{2\alpha_1}\gamma \frac{\sqrt{T_\alpha-\tau}\sqrt{\tau}}{T_\alpha} 
\left\|e^{\alpha_{1}N} U(s) u_{G,0}^h \right\|_{L^\infty_s(I^h_\tau;L^{2}_{z}L^{2}_{y_G})}\\
&\lesssim
C_{V}^2 e^{2\alpha_1}\gamma  
\left\|e^{\alpha_{1}N} u_{G,0}^h \right\|_{L^{2}_{z}L^{2}_{y_G}}
\end{align*}
By taking the supremum w.r.t.~$\alpha\in [-\alpha_{1},\alpha_{1}[$ and
dividing by $C_{V}e^{\alpha_{1}}\gamma^{1/2}/2$ we obtain
\begin{equation}
  \label{eq:estimgh}
M_2(f^{h}_{2,1})\lesssim C_{V}e^{\alpha_{1}}C_{\alpha_{1}}\gamma^{1/2}\,.
\end{equation}

It remains to control 
$$
f_{2,2}^h (t) = a_G(V)U(t)u^h_{G,0}.
$$
 For
$-\alpha_1 \leq \alpha < \alpha_1$ and $0\leq \tau < T_\alpha$,
 Proposition~\ref{pr:expaLp} and the homogeneous Strichartz estimate~\eqref{eq:homStri} yield
\begin{align}
\sqrt{T_\alpha-\tau} &\|e^{\alpha N}a_G(V) U(t)u^h_{G,0}\|_{L^2(I^h_\tau;L^2_{z}L^{2}_{y_G})} \nonumber \\
&\lesssim 
\sqrt{T_\alpha-\tau} \frac{C_{V} \, e^{\alpha_1}}{\sqrt{\alpha_1-\alpha}}\| U(t)e^{\alpha_1 N}u^h_{G,0}\|_{L^2_{z_{1}}L^2(I^h_\tau;L^{r_\sigma}_{y_G})} \nonumber \\
&\lesssim C_{V} \, e^{\alpha_1} \, \sqrt{\gamma}
  \|e^{\alpha_1 N} u^h_{G,0}\|_{L^2_{z}L^{2}_{y_G}} \label{eq:aG(v)U(t)u0}\,.
\end{align}
Taking the supremum over $\tau\in [0,T_{\alpha}[$\,, $\alpha\in
[-\alpha,\alpha_{1}[$ and dividing by
$C_{V}e^{\alpha_{1}}\gamma^{1/2}/2$ gives
$$
M(0,f_{2,2}^{h},0)\lesssim C_{\alpha_{1}}
$$
It can be improved by rewriting the system
$$
\begin{pmatrix}
  u^{h}_{\infty}\\
u^{h}_{2}\\
u^{h}_{1}
\end{pmatrix}
=(\mathrm{Id}-L)^{-1}
\begin{pmatrix}
  f^{h}_{\infty}\\
f^{h}_{2,1}+f^{h}_{2,2}\\
0
\end{pmatrix}
=
\begin{pmatrix}
 0\\
f_{2,2}^{h}\\
0
\end{pmatrix}
+(\mathrm{Id}-L)^{-1}
\begin{pmatrix}
  f^{h}_{\infty}\\
f^{h}_{2,1}\\
0
\end{pmatrix}
+(\mathrm{Id}-L)^{-1}L
\begin{pmatrix}
  0\\f^{h}_{2,2}\\
0
\end{pmatrix}
$$
from which we deduce
$$
M(u^{h}_{\infty},u_{2}^{h}-f^{h}_{2,2},u_{1}^{h})\lesssim
C_{V}e^{\alpha_{1}}\sqrt{\gamma}\left[M(f_{\infty}^{h},f_{2,1}^{h},0)+M(0,f_{2,2}^{h},0)\right] \,.
$$
The inequalities \eqref{eq:estimfh}, \eqref{eq:estimgh} and \eqref{eq:aG(v)U(t)u0} prove that
$(f_\infty^{h},f_2^{h},0)\in \mathcal{E}_{-\alpha_{1},\alpha_{1},\gamma}^{h}$ and thus
$$
M(u_\infty^{h},u_2^{h}-f_{2,2}^{h},u_1^{h})\leq 2M(f_\infty^{h},f_2^{h},0)\lesssim C_{V}e^{\alpha_{1}}C_{\alpha_{1}}\gamma^{1/2}\,.
$$
By possibly enlarging the constant $C_{d}>0$\,, the above inequality
becomes
$$
M(u_\infty^{h},u_2^{h}-f_{2,2}^{h},u_1^{h})\leq C_{d}C_{V}C_{\alpha_{1}}e^{\alpha_{1}}\gamma^{1/2}\,.
$$  
We have finished the proof as soon as we can identify 
$$
u^{h}_\infty(t)=u_{G}^{h}(t)-U(t)u_{G,0}^{h}
$$
for $t\in I_{T_{\alpha}}^{h}$ and $\alpha\in [0,\alpha_{1}[$\,.
For $t\in I_{T_{0}}^{h}$\,, the function $u_G^h(t) = u^{h}_\infty(t)+U(t)u_{G,0}^{h}$
belongs to $\mathcal{C}^{0}(I_{T_{0}}^{h};L^{2}_{z}L^{2}_{y_{G}})$
and satisfies \eqref{eq:uGmild} which is equivalent to
\eqref{eq:dynaleauG}. 
By the existence and uniqueness for \eqref{eq:uGmild} or
\eqref{eq:dynaleauG} in $\mathcal{C}^{0}(I_{T_{0}}^{h};L^{2}_{z}L^{2}_{y_{G}})$ when $V\in  H^{2}(\rz^{d};\rz)$\,,
$u_{G}^{h}$ is the unique solution to \eqref{eq:uGmild} or
\eqref{eq:dynaleauG} in $\mathcal{C}^{0}(I_{T_{0}}^{h};L^{2}_{z}L^{2}_{y_{G}})$\,.
\end{proof}

\subsection{Consequences of Theorem~\ref{th:applStri}}
\label{sec:corollaries}

Let us work now with a general initial time $t_{0}$\,, specified later, and consider \eqref{eq:dynaleauG}
\begin{equation}
   \label{eq:dynaleauGt0}
 \begin{cases}
       i\partial_{t}u_{G}^{h} = (\xi-D_{y_{G}})^{2}u_{G}^{h}+\sqrt{h}[a^{*}_{G}(V)+a_{G}(V)]u_{G}^{h}\,,\\
 u_{G}^{h}(t=t_{0})=u_{G,t_{0}}^{h}\,,
\end{cases}
 \end{equation}
with the solution 
$u_{G}^{h}(t)=u^{h}_{G}(t'+t_{0})=U(t')u^{h}_{G,t_{0}}+u^{h}_{\infty}$ in the framework of
Theorem~\ref{th:applStri}. For simplicity and because we work
definitely in the framework of \eqref{eq:dynaleauGt0} we use here
$U(t)U(s)^{*}=U(t-s)$\,.
 Remember that
$(u_{\infty}^{h},u_{2}^{h},u_{1}^{h})$ solves 
\begin{equation}
\label{eq:systt0}
\begin{pmatrix}
  u_{\infty}^{h}\\
u_{2}^{h}\\
u_{1}^{h}
\end{pmatrix}
=\underbrace{
\begin{pmatrix}
  L_{\infty\,\infty}&L_{\infty 2}&L_{\infty 1}\\
0&L_{22}&0\\
L_{1\infty}&0&L_{11}
\end{pmatrix}}_{=L}
\begin{pmatrix}
  u_{\infty}^{h}\\
u_{2}^{h}\\
u_{1}^{h}
\end{pmatrix}
+
\begin{pmatrix}
  f_{\infty}^{h}\\
f_{2}^{h}\\
0
\end{pmatrix}
\end{equation}
with
\begin{align*}
L_{\infty\,\infty}(\varphi)(t') &=-i\int_{0}^{t'}U(t'-s)\sqrt{h}a_{G}^{*}(V)\varphi(s)~ds\,,\\
L_{\infty
   1}(\varphi)(t') &= -i\int_{0}^{t'}U(t'-s)\varphi(s)~ds\,, \quad
   L_{\infty 2}(\varphi)(t')=\sqrt{h}L_{\infty 1}(\varphi)(t')\,,\\
L_{qq}(\varphi)(t') &= -i\int_{0}^{t'}a_{G}(V)U(t'-s)\sqrt{h}\varphi(s)~ds\,,\quad q\in\{2,1\}\,, \\
L_{1\infty}(\varphi)(t') &= -ih\int_{0}^{t'}a_{G}(V)U(t'-s)a_{G}^{*}(V)\varphi(s)~ds\,,
\end{align*}
and
\begin{align*}
   f_{\infty}^{h}(t')&= -i\int_{0}^{t'}U(t'-s)a_{G}^{*}(V)\sqrt{h}U(s)u_{G,t_{0}}^{h}~ds\,,\\
f_{2}^{h}(t')&=
\underbrace{-ia_{G}(V)\int_{0}^{t'}U(t'-s)a_{G}^{*}(V)\sqrt{h}U(s)u_{G,t_{0}}^{h}~ds}_{f^{h}_{2,1}(t')}+
\underbrace{a_{G}(V)U(t')u^{h}_{G,t_{0}}}_{f^{h}_{2,2}(t')}\,.
\end{align*}
Theorem~\ref{th:applStri} provides a framework in which $L$ is a
contraction and we will use it twice while inverting
$$
\begin{pmatrix}
  u_{\infty}^{h}\\
u_{2}^{h}\\
u_{1}^{h}
\end{pmatrix}
=(\mathrm{Id}-L)^{-1}
\begin{pmatrix}
  f_{\infty}^{h}\\
f_{2}^{h}\\
0
\end{pmatrix}
=\begin{pmatrix}
  0\\
f_{2,2}^{h}\\
0
\end{pmatrix}
+
(\mathrm{Id}-L)^{-1}
\begin{pmatrix}
  f_{\infty}^{h}\\
f_{2,1}^{h}\\
0
\end{pmatrix}
+
(\mathrm{Id}-L)^{-1}L
\begin{pmatrix}
  0\\
f_{2,2}^{h}\\
0
\end{pmatrix}
$$
and then using the Neumann expansion
$(\mathrm{Id}-L)^{-1}=\sum_{k=0}^{\infty}L^{k}$ for different values
of $t_{0}$ and of the parameter $\gamma$ in Theorem~\ref{th:applStri}.
The following result is an easy consequence of
Theorem~\ref{th:applStri}.
\begin{proposition}
\label{pr:applStri} Assume that the initial datum $u_{G,0}^{h}$ for
$t_{0}=0$ in \eqref{eq:dynaleauGt0} satisfies the uniform bound
$\|e^{2\alpha_{1}N}u^{h}_{G,0}\|_{L^{2}_{z}L^{2}_{y_{G}}}\leq
C_{\alpha_{1}}$ for all $h\in ]0,h_{0}[$\,. Then there exists
$\hat{T}_{\alpha_{1}}>0$ and $\tilde{C}_{\alpha_{1}}>0$\,, $\delta_{\alpha_{1}}>0$\,, such that
\begin{description}
\item[a)] The following weighted estimate
$$
\|e^{\alpha_{1}N}u_{G}^{h}(t)\|_{L^{2}_{z}L^{2}_{y_{G}}}\leq
\tilde{C}_{\alpha_{1}}
$$
holds true for all $t\in
I_{\hat{T}_{\alpha_{1}}}^{h}=]-\frac{\hat{T}_{\alpha_{1}}}{h},\frac{\hat{T}_{\alpha_{1}}}{h}[$
and all $h\in ]0,h_{0}[$\,.
\item[b)] For $t_{0}\in I_{\hat{T}_{\alpha_{1}}}^{h}$\,,
  $u_{G}^{h}(t_{0}+\delta/h)$
  admits in $e^{-\frac{\alpha_{1}}{2}N}L^{2}_{z}L^{2}_{y_{G}}$ 
the following asymptotic expansion in terms of $\delta\in [-\delta_{1},\delta_{1}]$\,,
  \begin{multline*}
u_{G}^{h}(t_{0}+\delta/h)=
\underbrace{U(\delta/h)u^{h}_{G}(t_{0})}_{\mathcal{O}(1)}
-i\underbrace{\sqrt{h}\int_{0}^{\delta/h}U(\delta/h-s)[a_{G}^{*}(V)+a_{G}(V)]U(s)u_{G}^{h}(t_{0})~ds}_{\mathcal{O}(|\delta|^{1/2})}\\
-\underbrace{h\int_{0}^{\delta/h}\int_{0}^{s}U(\delta/h-s)[a_{G}^{*}(V)+a_{G}(V)]U(s-s')[a_{G}^{*}(V)+a_{G}(V)]U(s')u_{G}^{h}(t_{0})~ds'~ds}_{\mathcal{O}(|\delta|)}
\\
\hspace{8cm}
+\mathcal{O}(|\delta|^{3/2})
  \end{multline*}
where  $v(h,\delta)=\mathcal{O}(|\delta|^{k/2})$\,, $k=0,1,2,3$\,, means
$\|e^{\frac{\alpha_{1}}{2}N}v(h,\delta)\|_{L^{2}_{z}L^{2}_{y_{G}}}\leq \tilde{C}_{\alpha_{1}}|\delta|^{k/2}$ uniformly with respect to
$h\in ]0,h_{0}[$ and $t_{0}\in I_{\hat{T}_{\alpha_{1}}}^{h}$\,.
\end{description}
\end{proposition}
\begin{proof}
\textbf{a)} Fix $\alpha_{1}>0$ and apply Theorem~\ref{th:applStri} with
$\alpha_{1}$ replaced by $2\alpha_{1}$\,. There exists
$\gamma=\gamma_{1}>0$\,, determined by $\alpha_{1}$\,, $C_{12}(V)$ and the dimension
$d\geq 3$\,, such that the operator $L$ is a contraction in
$\mathcal{E}^h_{-2\alpha_{1},2\alpha_{1},\gamma_{1}}$\,. The
system~\eqref{eq:systt0} for $t_{0}=0$ admits a unique solution with the norm $M$ in
$\mathcal{E}^h_{-2\alpha_{1},2\alpha_{1},\gamma_{1}}$ estimated by
\begin{equation}
  \label{eq:estimsyst}
M(u^{h}_{\infty},u^{h}_{2},u^{h}_{1})\lesssim C_{\alpha_{1}}
\end{equation}
and the solution $u_{G}^{h}$ to \eqref{eq:dynalea} equals
$$
u_{G}^{h}(t)=U(t)u^{h}_{G,0}+u_{\infty}^{h}(t)\,.
$$
With
$T_{\alpha_{1}}=\gamma_{1}(2\alpha_{1}-\alpha_{1})=\gamma_{1}\alpha_{1}$\,,
the estimate \eqref{eq:estimsyst} says in particular
$$
\forall t\in I^{h}_{T_{\alpha_{1}}}\,,\quad
\|e^{\alpha_{1}N}u_{\infty}^{h}(t)\|_{L^{2}_{z}L^{2}_{y_{G}}}\lesssim C_{\alpha_{1}}\frac{\sqrt{|ht|}}{\sqrt{T_{\alpha_{1}}-|ht|}}\,.
$$
By taking $\hat{T}_{\alpha_{1}}=\frac{T_{\alpha_{1}}}{2}$ with
$|ht|\leq \frac{T_{\alpha_{1}}}{2}$ when $t\in
I_{\hat{T}_{\alpha_{1}}^{h}}$ and with 
$$
\|e^{\alpha_{1}N}U(t)u^{h}_{G,0}\|_{L^{2}_{z}L^{2}_{y_{G}}}\leq
\|e^{\alpha_{1}}u^{h}_{G,0}\|_{L^{2}_{z}L^{2}_{y_{G}}}\leq C_{\alpha_{1}}\,,
$$
we finally obtain
$$
\forall t\in I_{\hat{T}_{\alpha_{1}}}^{h}\,,\quad
\|e^{\alpha_{1}N}u^{h}_{G}(t)\|_{L^{2}_{z}L^{2}_{y_{G}}}\lesssim \tilde{C}_{\alpha_{1}}\,,
$$
for $\tilde{C}_{\alpha_{1}}$ large enough.\\
\noindent\textbf{b)} With \textbf{a)} the initial datum
$u^{h}_{G,t_{0}}=u^{h}_{G}(t_{0})$ of \eqref{eq:dynaleauGt0} fulfils
the assumptions of Theorem~4.1 after time translation
$t'=t-t_{0}$ and where $t'\in I_{T}^{h}$ means $t\in t_{0}+I_{T}^{h}$ \,. For any  $\gamma>0$ small enough and by setting
$T_{\alpha}=\gamma(\alpha_{1}-\alpha)$ for $\alpha\in [0,\alpha_{1}]$
we know that the system \eqref{eq:systt0} satisfies
$$
\|L\|_{\mathcal{L}(\mathcal{E}^h_{-\alpha_{1},\alpha_{1},\gamma})}\lesssim
\gamma^{1/2}\quad,\quad M(f^{h}_{\infty},f^{h}_{2,1},0)
\lesssim C_{\alpha_{1}}\gamma^{1/2}\quad,\quad
M(0,f^{h}_{2,2},0)\lesssim C_{\alpha_{1}}\,,
$$
while $u_{G}^{h}(t'+t_{0})=U(t')u^{h}_{G}(t_{0})+u^{h}_{\infty}(t')$ for
$t'\in I^{h}_{T_{\alpha}}$\,.\\
In particular 
$$
\begin{pmatrix}
  u_{\infty}^{h}\\
u_{2}^{h}\\
u_{1}^{h}
\end{pmatrix}
=
\begin{pmatrix}
  0\\
f^{h}_{2,2}\\
0
\end{pmatrix}
+
(\mathrm{Id}-L)^{-1}
\begin{pmatrix}
  f^{h}_{\infty}\\
f^{h}_{2,1}\\
0
\end{pmatrix}
+(\mathrm{Id}-L)^{-1}
L
\begin{pmatrix}
  0\\
f^{h}_{2,2}\\
0
\end{pmatrix}
$$
leads to
$$
\begin{pmatrix}
  u_{\infty}^{h}\\
u_{2}^{h}\\
u_{1}^{h}
\end{pmatrix}
=
\begin{pmatrix}
  f^{h}_{\infty}+L_{\infty
    2}(f^{h}_{2,2})+L_{\infty\,\infty}(f^{h}_{\infty}+L_{\infty
    2}(f^{h}_{2,2}))+L_{\infty 2}(f^{h}_{2,1}+L_{22}(f^{h}_{2,2}))\\
f^{h}_{2,2}+f^{h}_{2,1}+L_{22}(f^{h}_{2,2})+
L_{22}(f^{h}_{2,1}+L_{22}(f^{h}_{2,2}))\\
L_{1\infty}(f^{h}_{\infty}+L_{\infty 2}(f^{h}_{2,2}))
\end{pmatrix}
+\mathcal{O}(\gamma^{3/2})
$$
in $\mathcal{E}^h_{-\alpha_{1},\alpha_{1},\gamma}$\,.
By using the first line with $\alpha=\frac{\alpha_{1}}{2}$\,, and by setting
\begin{multline*}
v^{h}(t')=
U(t')u^{h}_{G}(t_{0})
+[f^{h}_{\infty}(t')+L_{\infty2}(f^{h}_{2,2})](t')\\
+L_{\infty\,\infty}[f^{h}_{\infty}+L_{\infty2}(f^{h}_{2,2})](t')
+L_{\infty
  2}[f^{h}_{2,1}+L_{22}(f^{h}_{2,2})](t')
\end{multline*}
the difference $u_{G}^{h}(t_{0}+t')-v^{h}(t')$ satisfies
satisfies
$$
\forall t'\in I^{h}_{T_{\frac{\alpha_{1}}{2}}}\,,\quad \|e^{\frac{\alpha_{1}}{2}N}[u^{h}_{G}(t_{0}+t')-v^{h}(t')]\|_{L^{2}_{z}L^{2}_{y_{G}}}\lesssim \gamma^{3/2}\frac{\sqrt{|ht'|}}{\sqrt{T_{0}-|ht'|}}\,,
$$
where $T_{\frac{\alpha_{1}}{2}}=\frac{\gamma\alpha_{1}}{2}$\,. 
For 
$\delta=\pm \frac{T_{\frac{\alpha_{1}}{2}}}{2}=\pm\frac{\gamma\alpha_{1}}{4}$ we obtain
$$
\|e^{\frac{\alpha_{1}}{2}N}[u^{h}_{G}(t_{0}+\delta/h)-v^{h}(\delta/h)]\|_{L^{2}_{z}L^{2}_{y_{G}}}
=\mathcal{O}(|\delta|^{3/2})\,.
$$
It now suffices to specify all the terms of $v^{h}(\delta/h)$:
\begin{itemize}
\item The first one is nothing but $U(\delta/h)u^{h}_{G}(t_{0})$ with an
  $\mathcal{O}(1)$-norm.
\item The second term
$$
f_{\infty}^{h}(\delta/h)+L_{\infty 2}(f^{h}_{2,2})(\delta/h)
=-i\int_{0}^{\delta/h}U(\delta/h-s)\sqrt{h}[a_{G}^{*}(V)+a_{G}(V)]U(s)u^{h}_{G}(t_{0})~ds
$$
has an $\mathcal{O}(\delta^{1/2})$-norm.
\item All the other terms have an $\mathcal{O}(\delta)$-norm and equal
  \begin{align*}
L_{\infty\,\infty}(f^{h}_{\infty})(\delta/h)&=-h\int_{0}^{\delta/h}\int_{0}^{s}U(\delta/h-s)a_{G}^{*}(V)U(s-s')a_{G}^{*}(V)U(s')u^{h}_{G}(t_{0})~ds'ds \,, \\
L_{\infty\,\infty}(L_{\infty 2}(f^{h}_{2,2}))(\delta/h)
&=
-h\int_{0}^{\delta/h}\int_{0}^{s}U(\delta/h-s)a_{G}^{*}(V)U(s-s')a_{G}(V)U(s')u^{h}_{G}(t_{0})~ds'ds \,, \\
L_{\infty 2}(f^{h}_{2,1})(\delta/h)&=
-h\int_{0}^{\delta/h}\int_{0}^{s}U(\delta/h-s)a_{G}(V)U(s-s')a_{G}^{*}(V)U(s')u^{h}_{G}(t_{0})~ds'ds \,, \\
L_{\infty 2}(L_{22}(f^{h}_{2,2}))(\delta/h)&=
-h\int_{0}^{\delta/h}\int_{0}^{s}U(\delta/h-s)a_{G}(V)U(s-s')a_{G}(V)U(s')u^{h}_{G}(t_{0})~ds'ds\,.
  \end{align*}
\end{itemize}
This ends the proof.
\end{proof}

\section{Semiclassical measures}
\label{sec:evolsemi}
We will check here that semiclassical (or Wigner) measures for our
model problem presented in Section~\ref{sec:ourproblem}
can be
defined simultaneously for all  macroscopic times $t\in
]-\hat{T}_{\alpha_{1}},\hat{T}_{\alpha_{1}}[$\,.\\

\subsection{Framework}
\label{sec:framsemi}
Below are reviewed a few properties of semiclassical measures or
Wigner measures. We
refer the reader e.g.  to
\cite{CdV}\cite{Ger}\cite{GMMP}\cite{HMR}\cite{LiPa}\cite{Sch} for
various presentations of  those now well known
objects.
\begin{description}
\item[a)] The Anti-Wick quantization on $\rz^{d}$ is defined by
$$
a^{\AWick}(hx,D_{x})=\int_{T^{*}\rz^{d}}a(X)~|\varphi_{X}^h\rangle\langle
\varphi_{X}^{h}|~\frac{dX}{(2\pi h)^{d}}
$$
is defined for any $a\in L^{\infty}(T^{*}\rz^{d},dx;\cz)$ with 
$$
\varphi_{X_{0}}^{h}(x)=\frac{h^{d/4}}{\pi^{d/4}}e^{i\xi_{0} \cdot (x-\frac{x_{0}}{2h})}e^{-\frac{h(x-\frac{x_{0}}{h})^{2}}{2}}\,, \quad
X_{0}=(x_{0},\xi_{0})\in T^{*}\rz^{d}\,.
$$
It is a non negative quantization for which
$$
(a\geq 0)\Rightarrow (a^{\AWick}(hx,D_{x})\geq 0)\quad
\text{and}\quad
\|a^{\AWick}(hx,D_{x})\|_{\mathcal{L}(L^{2}(\rz^{d},dx;\cz))}\leq \|a\|_{L^{\infty}}\,.
$$
A natural separable subspace  of $L^{\infty}(T^{*}\rz^{d};\cz)$ is
\begin{eqnarray*}
  &&
\mathcal{C}^{0}_{0}(T^{*}\rz^{d};\cz)=\left\{a\in
  \mathcal{C}^{0}(T^{*}\rz^{d};\cz)\,,
     \lim_{X\to\infty}a(X)=0\right\}\\
\text{resp.}&&
\mathcal{C}^{0}(T^{*}\rz^{d}\sqcup
\left\{\infty\right\};\cz)=\mathcal{C}^{0}_{0}(T^{*}\rz^{d};\cz)\oplus
\cz~1=\left\{a\in \mathcal{C}^{0}(T^{*}\rz^{d};\cz)\,, 
              \lim_{X\to \infty}a(X)\in \cz\right\}\,,
\end{eqnarray*}
endowed with the $\mathcal{C}^{0}$ norm, of which the dual is
the  space $\mathcal{M}_{b}(T^{*}\rz^{d};\cz)$
(resp. $\mathcal{M}_{b}(T^{*}\rz^{d}\sqcup
\left\{\infty\right\};\cz)$) of bounded
 Radon measures on $T^{*}\rz^{d}$
(resp. $T^{*}\rz^{d}\sqcup \left\{\infty\right\}$)\,.
\item[b)] For a bounded family $(\varrho_{h})_{h\in ]0,h_{0}[}$ of
  normal states  $\varrho_{h}\in \mathcal{L}^{1}(L^{2}(\rz^{d},dx;\cz))$\,,
  $\varrho_{h}\geq 0$\,, $\mathrm{Tr}[\varrho_{h}]=1$\,, the 
set of semiclassical measures on $T^{*}\rz^{d}$ (resp. $T^{*}\rz^{d}\sqcup
\left\{\infty\right\}$) is defined as the weak$^{*}$ limit point in
$\mathcal{M}_{b}(T^{*}\rz^{d};\rz_{+})$ (resp. $\mathcal{M}_{b}(T^{*}\rz^{d}\sqcup
\left\{\infty\right\};\rz_{+})$) of
$\frac{\sigma^{\Wick}(\varrho_{h})}{(2\pi h)^{d}}$ with
$$
\sigma^{\Wick}(\varrho_{h})(X)=\langle \varphi^{h}_{X}\,,\,
\varrho_{h}\varphi_{X}^{h}\rangle_{L^{2}(\rz^{d})}=\mathrm{Tr}\left[\varrho_{h}|\varphi_{X}^{h}\rangle
  \langle \varphi_{X}^{h}|\right]\,.
$$
This is extended by linearity for any bounded family 
$(\varrho_{h})_{h\in ]0,h_{0}[}$ in
$\mathcal{L}^{1}(L^{2}(\rz^{d},dx;\cz))$\,.\\
The set of semiclassical measures is denoted by
$$
\mathcal{M}(\varrho_{h},h\in ]0,h_{0}[)\,,
$$
and when $h$ is restricted to a set $\mathcal{E}\subset ]0,h_{0}[$\,,
$0\in \overline{\mathcal{E}}$\,, we use
$$
\mathcal{M}(\varrho_{h}, h\in \mathcal{E})\,.
$$
After recalling
$$
\int_{T^{*}\rz^{d}}a(X) \, \sigma^{\Wick}(\varrho_{h})(X) \, \frac{dX}{(2\pi
   h)^{d}}=\mathrm{Tr}\left[a^{\AWick}(hx,D_{x}) \, \varrho_{h}\right] \,,
$$
semiclassical measures
 $\mu\in \mathcal{M}(\varrho_{h}, h\in
]0,h_{0}[)$  are characterized by the existence of
 a sequence $(h_{k})_{k\in \nz^{*}}$\,,
$h_{k}\in \mathcal{E}$ such that
\begin{eqnarray*}
  && \lim_{k\to \infty}h_{k}=0 \,,  \\
&&
   \lim_{k\to\infty}\mathrm{Tr}\left[a^{\AWick}(h_{k}x,D_{x}) \, \varrho_{h_{k}}\right]=\int_{T^{*}\rz^{d}}a(X)~d\mu(X)\,, \quad
  \forall a\in \mathcal{D} \,, \\
&&\lim_{k\to\infty}\mathrm{Tr}\left[\varrho_{h_{k}}\right]=\mu(T^{*}\rz^{d}\sqcup
\left\{\infty\right\})=\mu(T^{*}\rz^{d})+\mu(\infty)\,,
\end{eqnarray*}
where $\mathcal{D}$ is any dense set of
$\mathcal{C}^{0}_{0}(T^{*}\rz^{d};\cz)$\,.
\item[d)] After choosing
  $\mathcal{D}=\mathcal{C}^{\infty}_{0}(T^{*}\rz^{d};\cz)$ and by
  recalling
  $\|a^{\AWick}(hx,D_{x})-a^{\Weyl}(hx,D_{x})\|=\mathcal{O}(h)$\,, for
  any $a\in S(1,dx^{2}+d\xi^{2})\supset \mathcal{C}^{\infty}_{0}(T^{*}\rz^{d};\cz)$\,,
  semiclassical measures are characterized by
$$
\forall a\in \mathcal{C}^{\infty}_{0}(T^{*}\rz^{d};\cz)\,,\quad
\lim_{k\to \infty}\mathrm{Tr}\left[a^{\Weyl}(h_{k}x,D_{x}) \, \varrho_{h_{k}}\right]=\int_{T^{*}\rz^{d}}a(X)~d\mu(X)\,,
$$
or 
$$
\forall P\in T^{*}\rz^{d}\,,\quad \lim_{k\to \infty}\mathrm{Tr}\left[\tau_{P}^{h_{k}}\varrho_{h_{k}}\right]=\int_{T^{*}\rz^{d}}e^{i(p_{\xi} \cdot x-p_{x} \cdot \xi)}~d\mu(x,\xi)\,,
$$
with 
$$
\tau_{P}^{h}=(e^{i(p_{\xi} \cdot x-p_{x} \cdot \xi)})^{\Weyl}(hx,D_{x})=
e^{i(p_{\xi} \cdot (hx)-p_{x} \cdot D_{x})} \,, \quad P=(p_{x},p_{\xi})\,.
$$
The compactification $T^{*}\rz^{d}\sqcup \left\{\infty\right\}$ is
just a way to count the mass of $(\varrho_{h_{k}})_{k\in\nz^{*}}$
which is not caught by the compactly supported obervables $a\in
\mathcal{C}^{\infty}_{0}(T^{*}\rz^{d};\cz)$\,.
\item[e)] Semiclassical measures are transformed by the dual action of
  semiclassical Fourier integral operator on $a^{\Weyl}(hx,D_{x})$\,,
  $a\in \mathcal{C}^{\infty}_{0}(T^{*}\rz^{d};\cz)$\,.
\item[f)] When $\mathcal{M}(\varrho_{h,1},h\in
  \mathcal{E})=\left\{\mu_{1}\right\}$ and
  $\mathcal{M}(\varrho_{h,2}, h\in
  \mathcal{E})=\left\{\mu_{2}\right\}$ the total variation between
  $\mu_{1}$ and $\mu_{2}$ is estimated by 
$$
|\mu_{2}-\mu_{1}|(\underbrace{T^{*}\rz^{d}}_{\text{or}~T^{*}\rz^{d}\sqcup
  \left\{\infty\right\}})\leq 4\liminf_{h\to
  0}\|\varrho_{h,1}-\varrho_{h,2}\|_{\mathcal{L}^{1}}\,.
$$
\item[g)] When $(\Lambda, d_{\Lambda})$ is a metric space
  and $(\varrho_{h}(\lambda))_{h\in ]0,h_{0}[, \lambda\in \Lambda}$ is
  a bounded family in $\mathcal{L}^{1}(L^{2}(\rz^{d},dx;\cz))$\,,
  semiclassical measures can be defined simultaneously for all $\lambda\in
  \Lambda$\,, if for any sequence $(h_{n})_{n\in\nz}$\,, $\lim_{n\to
    \infty}h_{n}=0^{+}$\,, there exists a subsequence
  $(h_{n_{k}})_{k\in \nz}$ such that
\begin{multline*}
    \forall \lambda\in \Lambda\,, \exists \mu_{\lambda}\in
\mathcal{M}_{b}(T^{*}\rz^{d}\sqcup\{\infty\})\,,\\
\lim_{k\to \infty}\mathrm{Tr}\left[ a^{\AWick}(h_{n_{k}}x,D_{x}) \,
  \varrho_{h_{n_{k}}}(\lambda)\right]=\int_{T^{*}\rz^{d}\sqcup \left\{\infty\right\}}a(X) \, d\mu_{\lambda}(X)\,.
\end{multline*}
By assuming $(\Lambda,d_{\Lambda})$ separable, sufficient conditions
for this property are either
\begin{itemize}
\item For all given $a\in
  \mathcal{C}^{\infty}_{0}(T^{*}\rz^{d};\cz)$\,,
  $\mathrm{Tr}\left[a^{\Weyl}(hx,D_{x}) \, \varrho_{h}(\lambda)\right]$
  is an equicontinuous family of continuous functions from $\Lambda$
  to $\cz$\,, or
\item The map $(P,\lambda)\mapsto \mathrm{Tr}\left[\tau_{P}^{h} \, \varrho_{h}(\lambda)\right]$ is
  an equicontinuous family of continuous functions from
  $T^{*}\rz^{d}\times \Lambda$ to $\cz$\,.
\end{itemize}
For the first characterization, apply a diagonal extraction process for a dense countable subset of
$(\Lambda,d_{\Lambda})$ (and a dense countable subset of
$\mathcal{C}^{0}_{0}(T^{*}\rz^{d})$ lying in
$\mathcal{C}^{\infty}_{0}(T^{*}\rz^{d};\cz)$) and then apply the
various characterisations of elements of
$\mathcal{M}(\varrho_{h}(\lambda)\,, h\in \mathcal{E})$\,.
\end{description}
Like in our problem, semiclassical measures can be defined for
bounded families $\varrho_{h}\in \mathcal{L}^{1}(L^{2}(\rz^{d}\times
Z', dx\otimes \mathbf{dz'};\cz))$ after using observables
$a^{\Weyl}(hx,D_{x})\otimes
\mathrm{Id}_{L^{2}_{z}}$\,.\\
When $(\varrho_{h})_{h\in ]0,h_{0}[}$ is a family of states,
$\varrho_{h}\geq 0$ and $\mathrm{Tr}[\varrho_{h}]=1$\,, the
relationship with the study of pure states can be done in two ways:
\begin{itemize}
\item Firstly by writing a general state as a convex combination of
  pure states, provided that this convex decomposition is explicit
  enough to follow the behaviour as $h\to 0^{+}$\,.
\item Secondly by writing
  $\varrho_{h}=\varrho_{h}^{1/2}\varrho_{h}^{1/2}$ and taking
  $\Psi_{h}=\varrho_{h}^{1/2}\in \mathcal{L}^{2}(L^{2}(\rz^{d}\times
  Z', dx\otimes \mathbf{dz'};\cz))\sim L^{2}(\rz^d \times Z'\times \hat{Z},
  dx\otimes \mathbf{dz'}\otimes \mathbf{d\hat{z}};\cz)$ where $\hat{Z}$
  is another copy of $\rz^{d}\times Z'$ with
  $\mathbf{d\hat{z}}=dx\otimes \mathbf{dz'}$\,. Then
$$
\mathrm{Tr}\left[(a^{\Weyl}(hx,D_{x})\otimes
  \mathrm{Id}_{L^{2}_{z'}})\varrho_{h}\right]=\langle \Psi_{h}\,,
(a^{\Weyl}(hx,D_{x})\otimes \mathrm{Id}_{L^{2}_{z',\hat{z}}})\Psi_{h}\rangle\,.
$$
\end{itemize}
\subsection{Equicontinuity} 
\label{sec:equicont}
The following result, which is the first useful information about
semiclassical measures, before computing them, comes from the
equicontinuity directly deduced from
Proposition~\ref{pr:applStri}. The unitary transforms introduced in
Section~\ref{sec:ourproblem} and Section~\ref{sec:centermass} in order
to transform \eqref{eq:dynalea} into \eqref{eq:dynaleauG} and
$a^{\Weyl}(hx,D_{x})\otimes \mathrm{Id}$ into $a^{\Weyl}(-hD_\xi,
\xi-D_{y_{G}})$ are not recalled here and the results are directly
formulated for the initial problem \eqref{eq:dynalea} and the
semiclassical observables $a^{\Weyl}(hx,D_{x})\otimes \mathrm{Id}$\,.
\begin{proposition}
\label{pr:equicont}
  Assume
$$
V\in L^{r'_{\sigma}}(\rz^{d},dx;\rz)\cap H^{2}(\rz^{d};\rz)\,,\quad
r'_{\sigma}=\frac{2d}{d+2}\,,\quad d\geq 3\,,
$$
and let $U_{\mathcal{V}}(t)=e^{-it(-\Delta_{x}+\sqrt{h}\mathcal{V})}$
like in Subsection~\ref{sec:ourproblem}.\\
Assume that there exists $\alpha_{1}>0$ such that $\varrho_{h}(0)\in
\mathcal{L}^{1}(L^{2}(\rz^{d}\times \Omega ,dx\otimes
\mathcal{G};\cz))$\,, $\varrho_{h}(0)\geq 0$\,,
$\mathrm{Tr}[\varrho_{h}(0)]=1$ satisfies
$$
\exists C_{\alpha_{1}}>0\,,\, 
\forall h\in ]0,h_{0}[\,, \quad
\mathrm{Tr}\left[e^{\alpha_{1}N}\varrho_{h}(0)e^{\alpha_{1}N}\right]\leq C_{\alpha_{1}}\,.
$$
Then there exists $\hat{T}_{\alpha_{1}}>0$ such that elements of
$\mathcal{M}(\varrho_{h}(t), h\in ]0,h_{0}[)$ can be defined
simultaneously for all macroscopic times  $t\in]-\hat{T}_{\alpha_{1}},\hat{T}_{\alpha_{1}}[$
when
$\varrho_{h}(t)=U_{\mathcal{V}}(\frac{t}{h})\, \varrho_{h}(0)\, U_{\mathcal{V}}^{*}(\frac{t}{h})$\,.
\end{proposition}
\begin{proof}
When $U(s)=e^{-is(-\Delta_{x})}$ denotes the free unitary transform\,, the time evolved observable
$U^{*}(\frac{s}{h})[a^{\Weyl}(hx,D_{x})\otimes
\mathrm{Id}_{L^{2}_{\omega}}]U(\frac{s}{h})$ equals
exactly $a^{\Weyl}(hx,D_{x},s)\otimes
\mathrm{Id}_{L^{2}_{\omega}}$ with
$$
a(x,\xi,s)=a(x+2\xi s,\xi)\,.
$$
It is clearly equicontinuous in $h\in ]0,h_{0}[$
 with respect to $s\in
[-\hat{T}_{\alpha_{1}},\hat{T}_{\alpha_{1}}]$ in $\mathcal{L}(L^{2}_{x,\omega})$
for any given  $a\in \mathcal{C}^{\infty}_{0}(T^{*}\rz^{d};\cz)$\,:
$$
\|a^{\Weyl}(hx,D_{x},s)-a^{\Weyl}(hx,D_{x},0)\|_{\mathcal{L}(L^{2}_{x,\omega})}\leq C_{a}|s|\,.
$$
We drop the tensorization with $\mathrm{Id}_{L^{2}}$\,.
With 
\begin{multline*}
  \mathrm{Tr}\left[
  a^{\Weyl}(hx,D_{x})\varrho_{h}(t+\delta)\right]-
 \mathrm{Tr}\left[
  a^{\Weyl}(hx,D_{x})\varrho_{h}(t)\right]\\
= \mathrm{Tr}\left[a^{\Weyl}(hx,D_{x},\delta)U^{*}(\frac{\delta}{h})U_{\mathcal{V}}(\frac{\delta}{h})\varrho_{h}(t)U^{*}_{\mathcal{V}}(\frac{\delta}{h})U(\frac{\delta}{h})\right]
-\mathrm{Tr}\left[a^{\Weyl}(hx,D_{x},0)\varrho_{h}(t)\right]
\end{multline*}
it thus suffices to check, uniformly with respect to  $(h,t)\in
]0,h_{0}[\times ]-\hat{T}_{\alpha_{1}},\hat{T}_{\alpha_{1}}[$\,, the
 estimate
 \begin{equation}
   \label{eq:estimcalL1}
\|U^{*}(\frac{\delta}{h})U_{\mathcal{V}}(\frac{\delta}{h})\varrho_{h}(t)U^{*}_{\mathcal{V}}(\frac{\delta}{h})U(\frac{\delta}{h})-\varrho_{h}(t)\|_{\mathcal{L}^{1}}=o_{\delta\to
  0}(1)\,.
\end{equation}
We now use the decomposition
$\varrho_{h}(0)=\varrho_{h}(0)^{1/2}\varrho_{h}(0)^{1/2}$ and consider
the evolution
$$
U_{\mathcal{V}}(\frac{t}{h})\varrho_{h}(0)^{1/2}\in
\mathcal{L}^{2}(L^{2}(\rz^{d}\times \Omega, dx\otimes
\mathcal{G};\cz))\sim
L^{2}(\rz^{d}\times \Omega\times \hat{Z}, dx\otimes
\mathcal{G}\otimes \mathbf{d\hat{z}};\cz)
$$
with $\hat{Z}=\rz^{d}\times \Omega$\,, $\mathbf{d\hat{z}}=dx\otimes
\mathcal{G}$\,.\\
The estimate \eqref{eq:estimcalL1} is done as soon as
$$
\|U^{*}(\frac{\delta}{h})U_{\mathcal{V}}(\frac{\delta}{h})[U_{\mathcal{V}}(\frac{t}{h})\varrho_{h}(0)^{1/2}]-[U_{\mathcal{V}}(\frac{t}{h})\varrho_{h}(0)^{1/2}]\|_{L^{2}_{x,\omega,\hat{z}}}=o_{\delta\to
0}(1)
$$
uniformly with respect to $(h,t)\in ]0,h_{0}[\times]-\hat{T}_{\alpha_{1}},\hat{T}_{\alpha_{1}}[$\,.\\
This problem is now translated in a problem in
$$\underbrace{L^{2}(\rz^{d}\times \hat{Z}, \frac{d\xi}{(2\pi)^{d}}\otimes
  \mathbf{d\hat{z}};\cz)}_{\text{vacuum}}\oplus
L^{2}_{\text{sym}}(\rz^{d}_{y_{G}}\times Z_{1}; dy_{G}\otimes \mathbf{dz_{1}};\cz)
$$
by the unitary transform
$U_{G}$ associated with the center of mass $y_{G}$ of
Section~\ref{sec:centermass}\,, the translation invariance and its
Fourier variable $\xi\in \rz^{d}$ and the relative coordinates $Y'
\in
\mathcal{R}$\,. The variable $ z_{1}\in Z_{1}$ is nothing but
$z_{1}=(\xi,Y',\hat{z})\in \rz^{d}\times \mathcal{R}\times \hat{Z}$
with $\mathbf{dz_{1}}=\frac{d\xi}{(2\pi)^{d}}\otimes \mu\otimes
\mathbf{d\hat{z}}$\,. The subscript $_{\text{sym}}$ refers to the symmetry in
the variable $Y'\in \mathcal{R}$\,. All the assumptions of
Theorem~\ref{th:applStri} have been checked in
Section~\ref{sec:conseqStri}.
In particular we can use Proposition~\ref{pr:applStri}-b) with 
$$
u_{G}^{h}(\frac{t}{h})=U_{\mathcal{V}}(\frac{t}{h})\varrho_{h}(0)^{1/2}\quad\text{and}\quad
\frac{t}{h}\in I^{h}_{\hat{T}_{\alpha_{1}}} \,.
$$
It says in particular
$$
u_{G}^{h}(\frac{t}{h}+\frac{\delta}{h})=U(\frac{\delta}{h})u_{G}^{h}(\frac{t}{h})+\mathcal{O}(|\delta|^{1/2})\,,
$$
uniformly with respect to $(h,\frac{t}{h})\in ]0,h_{0}[\times I^{h}_{\hat{T}_{\alpha_{1}}}$\,,
and therefore
$$
\|U^{*}(\frac{\delta}{h})U_{\mathcal{V}}(\frac{\delta}{h})
[U_{\mathcal{V}}(\frac{t}{h})\varrho_{h}(0)^{1/2}]
-[U_{\mathcal{V}}(\frac{t}{h})\varrho_{h}(0)^{1/2}]\|_{L^{2}_{x,\omega,\hat{z}}}=\mathcal{O}_{\delta\to
0}(|\delta|^{1/2})
$$
uniformly with respect to $(h,t)\in ]0,h_{0}[\times
]-\hat{T}_{\alpha_{1}},\hat{T}_{\alpha_{1}}[$\,.\\
This ends the proof.
\end{proof}

\section{Approximations}
\label{sec:approx}

With our number estimates stated in Section~\ref{sec:conseqStri},
various approximations can be considered for the general class of
initial data $(\varrho_{h}(0))_{h\in ]0,h_{0}[}$\,, $\varrho_{h}(0)\in
\mathcal{L}^{1}(L^{2}(\rz^{d}\times\Omega, dx\otimes
\mathcal{G};\cz))$\,, $\varrho_{h}(0)\geq 0$\,,
$\mathrm{Tr}[\varrho_{h}(0)]=1$ under the sole additional assumption
$\mathrm{Tr}[e^{\alpha_{1}N}\varrho_{h}(0)e^{\alpha_{1}N}]\leq
C_{\alpha_{1}}$\,. Before computing the evolution of the semiclassical
measures $(\mu_{t})_{t\in ]-\hat{T}_{\alpha_{1}},\hat{T}_{\alpha_{1}}[}$ given by
Proposition~\ref{pr:equicont} (this will be done in a subsequent
article), it provides useful a priori information for them.

\subsection{Truncation with respect to the number operator $N$}
\label{sec:truncN}
For $\varepsilon>0$\,, let 
$\chi_{\varepsilon}:[0,+\infty)\to [0,1]$ be  a decaying function 
such that
\begin{eqnarray}
  \label{eq:moment}
&&\forall k\in \nz\,, \forall \varepsilon\in
   ]0,1[\,,\exists C_{k,\varepsilon}>0 \,,\quad  \sup_{s\in [0,+\infty)}s^{k}\chi_{\varepsilon}(s)\leq
  C_{k,\varepsilon} \,\,, \\
\label{eq:expeps}
&& \forall \alpha_{1}>0\,,\exists C_{\alpha_{1}}>0\,, \forall
   \varepsilon\in ]0,1[\,,\quad \sup_{s\in
  [0,+\infty)}e^{-\alpha_{1}s}(1-\chi_{\varepsilon}(s))\leq
   C_{\alpha_{1}}\times \varepsilon\,.
\end{eqnarray}
Examples
are 
$$
\chi_{\varepsilon}(s)=1_{[0,\varepsilon^{-1}]}(s)\quad\text{and}\quad
\chi_{\varepsilon}(s)=e^{-\varepsilon s}\,.
$$
 Then the operators
$$
a_{G,\varepsilon}(V)=\chi_{\varepsilon}(N)a_{G}(V)\chi_{\varepsilon}(N)\quad,\quad a_{G,\varepsilon}^{*}(V)=\chi_{\varepsilon}(N)a_{G}^{*}(V)\chi_{\varepsilon}(N)
$$
are bounded operators on 
$$
F^{2}=L^{2}(Z',\mathbf{dz'};\Gamma
(L^{2}(\rz^{d},dy;\cz)))=L^{2}_{z,\text{sym}}=L^{2}_{z_{0}}\oplus L^{2}_{z_{1},\text{sym}}L^{2}_{y_{G}}
$$
according to \eqref{eq:F2L2L2}
and $\sqrt{h}(a_{G,\varepsilon}(V)+a_{G,\varepsilon}^{*}(V))$ is an
$\mathcal{O}_{\varepsilon}(\sqrt{h})$ bounded self-adjoint
perturbation of $(\xi-D_{y_{G}})^{2}$\,.
Additionally for $\varepsilon>0$ the estimates of
Proposition~\ref{pr:Lpagstag} hold true when $a_{G}(V)$ and
$a_{G}^{*}(V)$ are replaced by $a_{G,\varepsilon}(V)$ and
$a_{G,\varepsilon}^{*}(V)$\,. Actually, 
\eqref{eq:borne-a(V)} with $n>1$ and
\eqref{eq:borne-a-star(V)} with $n>0$ become
\begin{eqnarray}
  \label{eq:borne-aVeps}
\|a_{G,\varepsilon}(V)f_{G,n}\|_{L^{2}_{z',Y'_{n-1}}L^{p}_{y_{G}}}
\leq
   \|V\|_{L^{r'}}\chi_{\varepsilon}(n-1)^{2}\sqrt{n}\|f_{G,n}\|_{L^{2}_{z',Y'_{n}}L^{q}_{y_{G}}}
\leq C_{\varepsilon}\|V\|_{L^{r'}}\|f_{G,n}\|_{L^{2}_{z',Y'_{n}}L^{q}_{y_{G}}}&&
\\
  \label{eq:borne-aVstareps}
\|a_{G,\varepsilon}^{*}(V)f_{G,n}\|_{L^{2}_{z',Y'_{n+1}}L^{q'}_{y_{G}}}
\leq
   \|V\|_{L^{r'}}\chi_{\varepsilon}(n)^{2}\sqrt{n+1}\|f_{G,n}\|_{L^{2}_{z',Y'_{n}}L^{p'}_{y_{G}}}\leq C_{\varepsilon} \|V\|_{L^{r'}}\|f_{G,n}\|_{L^{2}_{z',Y'_{n}}L^{p'}_{y_{G}}}
&&
\end{eqnarray}
when $V\in L^{q'}(\rz^{d};\cz)\cap L^{r'}(\rz^{d};\cz)$\,,
$\frac{1}{r'}=\frac{1}{2}+\frac{1}{q'}-\frac{1}{p'}$\,, $p',q'\in
[1,2]$\,. All the analysis can thus be carried out with $a_{G}(V)$ and
$a_{G}^{*}(V)$ replaced by $a_{G,\varepsilon}(V)$ and
$a_{G,\varepsilon}^{*}(V)$\,, either with estimates which are uniform
in $\varepsilon\in ]0,1[$\,, or by replacing the $N$-dependent
estimates by constants $C_{\varepsilon}$ depending on $\varepsilon\in
]0,1[$\,.\\
In particular the solution $v_{G,\varepsilon}^{h}$ to 
\begin{equation}
   \label{eq:dynaleavGeps}
     \begin{cases}
       i\partial_{t}v_{G,\varepsilon}^{h} = (\xi-D_{y_{G}})^{2}v_{G,\varepsilon}^{h}+\sqrt{h}[a^{*}_{G,\varepsilon}(V)+a_{G,\varepsilon}(V)]v_{G,\varepsilon}^{h}\,,\\
 v_{G,\varepsilon}^{h}(t=0) = v_{G,\varepsilon,0}^{h}=u_{G,0}^{h}\,,
\end{cases}
 \end{equation} 
satisfies the same properties as the solution $u_{G}^{h}$ to
\eqref{eq:dynaleauG} stated in Theorem~\ref{th:applStri} and
Proposition~\ref{pr:applStri}, uniformly with respect to
$\varepsilon\in ]0,1[$\,. 
\begin{proposition}
\label{pr:comp}
Assume
$\|e^{2\alpha_{1}N}u_{G,0}^{h}\|_{L^{2}_{z}L^{2}_{y_{G}}}\leq
C_{\alpha_{1}}$ for all $h\in ]0,h_{0}[$ like in
Proposition~\ref{pr:applStri}. There exists $\hat{C}_{\alpha_{1}}>0$ and $\hat{T}_{\alpha_{1}}>0$ 
such that the solutions $u_{G}^{h}$ to \eqref{eq:dynaleauG} and
$v_{G,\varepsilon}^{h}$ to \eqref{eq:dynaleavGeps} for $\varepsilon\in
]0,1[$\,, satisfy
$$
\|u_{G}^{h}(t)-v_{G,\varepsilon}^{h}(t)\|_{L^{2}_{z}L^{2}_{y_{G}}}\leq \hat{C}_{\alpha_{1}}\varepsilon
$$
for all $t\in
I^{h}_{\hat{T}_{\alpha}}=]-\frac{\hat{T}_{\alpha_{1}}}{h},\frac{\hat{T}_{\alpha_{1}}}{h}[$\,.\\
Additionally the statement \textbf{b)} of
Proposition~\ref{pr:applStri} holds true when $u_{G}^{h}\,,$ $a_{G}(V)\,,$ 
$a_{G}^{*}(V)$ are replaced by $v_{G,\varepsilon}^{h}\,,$ 
$a_{G,\varepsilon}(V)\,,$ $a_{G,\varepsilon}^{*}(V)$\,. 
\end{proposition}
\begin{proof}
  The statements \textbf{a)} and \textbf{b)} of
  Proposition~\ref{pr:applStri} hold true uniformly with respect 
  to $\varepsilon\in ]0,1[$ for
  $v_{G,\varepsilon}^{h}$ as a consequence of the previous
  arguments.\\
In particular
$v_{G,\varepsilon}^{h}(t)=U(\frac{t}{h})u_{G,0}^{h}+v_{\infty,\varepsilon}^{h}$
where $(v_{\infty,\varepsilon}^{h},v_{2,\varepsilon}^{h},v_{1,\varepsilon}^{h})$ solves the system
\begin{equation}
            \begin{pmatrix}
              v_{\infty,\varepsilon}^{h}\\v_{2,\varepsilon}^{h}\\v_{1,\varepsilon}^h
            \end{pmatrix}
=L_{\varepsilon}
            \begin{pmatrix}
              v_{\infty,\varepsilon}^{h}\\v_{2,\varepsilon}^{h}\\v_{1,\varepsilon}^h
            \end{pmatrix}
+
            \begin{pmatrix}
              f_{\infty,\varepsilon}^{h}\\f_{2,\varepsilon}^{h}\\0
            \end{pmatrix}
\,,\quad L_{\varepsilon}=
  \begin{pmatrix}
    L_{\infty\,\infty\,,\varepsilon}&L_{\infty 2,\varepsilon}&L_{\infty 1,\varepsilon}\\
    0&L_{22,\varepsilon}&0\\
    L_{1\infty,\varepsilon}&0&L_{11,\varepsilon}
  \end{pmatrix}
\,, \label{eq:dynutvtexpeps}
\end{equation}
with
\begin{align}
\label{eq:rhsfheps}
f^{h}_{\varepsilon}(t) =-i&\int_{0}^{t}U(t)U(s)^{*}
a_{G,\varepsilon}^{*}(V) \sqrt{h}U(s)u_{G,0}^h~ds \,,\\
\label{eq:rhsgheps}
f^{h}_{2,\varepsilon}(t)=-i\,a_{G,\varepsilon}(V)&\int_{0}^{t}U(t)U(s)^{*}a_{G,\varepsilon}^{*}(V)
                                     \sqrt{h}U(s)u_{G,0}^h~ds + a_{G,\varepsilon}(V) U(t) u^h_{G,0}\,,
\end{align}
and where the entries $L_{\varepsilon}$ are the same as the ones of $L$
with $a_{G}(V)$ and $a_{G}^{*}(V)$ replaced by $a_{G,\varepsilon}(V)$
and $a_{G,\varepsilon}^{*}(V)$\,. When
$\chi_{\varepsilon}(s)=e^{-\varepsilon s}$\,, one recovers the system for
$u_{G}^{h}$ by taking $\varepsilon=0$\,.\\
We start now with the equation for $u_{G}^{h}$
$$
u_{G}^{h}(t)=U(\frac{t}{h})u_{G,0}^{h}-i\sqrt{h}\int_{0}^{\frac{t}{h}}U(t-s)[a_{G}(V)+a_{G}^{*}(V)]u_{G}^{h}(s)~ds\,,
$$
which implies
\begin{multline*}
\chi_{\varepsilon}(N)u_{G}^{h}(t)=U(\frac{t}{h})\chi_{\varepsilon}(N)u_{G,0}^{h}-i\sqrt{h}\int_{0}^{\frac{t}{h}}U(\frac{t}{h}-s)\chi_{\varepsilon}(N)[a_{G}(V)+a_{G}^{*}(V)]\chi_{\varepsilon}(N)^{2}u_{G}^{h}(s)~ds\\
-i\sqrt{h}\chi_{\varepsilon}(N)\int_{0}^{\frac{t}{h}}U(\frac{t}{h}-s)
[a_{G}(V)+a_{G}^{*}(V)](1-\chi_{\varepsilon}^{2}(N))u_{G}^{h}(s)~ds\,.
\end{multline*}
The function
$w_{G,\varepsilon}^{h}(t)=\chi_{\varepsilon}(N)u_{G}^{h}(t)$ solves
\begin{eqnarray}
\label{eq:weps}
w_{G,\varepsilon}^{h}(t)=U(\frac{t}{h})\chi_{\varepsilon}(N)u_{G,0}^{h}-i\sqrt{h}\int_{0}^{\frac{t}{h}}U(t-s)[a_{G,\varepsilon}(V)+a_{G,\varepsilon}^{*}(V)]w_{G,\varepsilon}^{h}(s)~ds+g_{\infty,\varepsilon}^{h}
&&
\\
\label{eq:defginfeps}
\text{with}\quad
g_{\infty,\varepsilon}^{h}=-i\sqrt{h}\chi_{\varepsilon}(N)\int_{0}^{\frac{t}{h}}U(t-s)[a_{G}(V)+a_{G}^{*}(V)](1-\chi_{\varepsilon}^{2}(N))u_{G}^{h}(s)~ds\,.&&
\end{eqnarray}
The system for
$(w^{h}_{\infty,\varepsilon},w^{h}_{2,\varepsilon},w^{h}_{1,\varepsilon})$ after
decomposing
$w_{G,\varepsilon}^{h}(t)=U(\frac{t}{h})\chi_{\varepsilon}(N)u_{G,0}^{h}+w_{\infty,\varepsilon}^{h}(t)$
is 
$$
\begin{pmatrix}
  w_{\infty,\varepsilon}^{h}\\
w_{2,\varepsilon}^{h}\\
w_{1,\varepsilon}^{h}
\end{pmatrix}
=L_{\varepsilon}
\begin{pmatrix}
  w_{\infty,\varepsilon}^{h}\\
w_{2,\varepsilon}^{h}\\
w_{1,\varepsilon}^{h}
\end{pmatrix}
+
\begin{pmatrix}
  \tilde{f}_{\infty,\varepsilon}^{h}\\
\tilde{f}_{2,\varepsilon}^{h}\\
0
\end{pmatrix}
+
\begin{pmatrix}
  g_{\infty,\varepsilon}^{h}\\
0\\0
\end{pmatrix}\,,
$$
where $\tilde{f}_{\infty,\varepsilon}^{h}$ and
$\tilde{f}_{2,\varepsilon}^{h}$ have the same expressions as
\eqref{eq:rhsfheps}\eqref{eq:rhsgheps} with $u_{G,0}^{h}$ replaced by
$\chi_{\varepsilon}(N)u_{G,0}^{h}$\,.
By taking the difference with \eqref{eq:dynutvtexpeps}, and because
$\|L_{\varepsilon}\|_{\mathcal{L}(\mathcal{E}_{0,\alpha_{1},\gamma})}\leq
1/2$ for $\gamma>0$ small enough, the proof is
done as soon as  the three norms
\begin{eqnarray}
\label{eq:normuG} &&\|u_{G}^{h}(t)-\chi_{\varepsilon}(N)u_{G}^{h}(t)\|_{L^{2}_{z}L^{2}_{y_{G}}}\\
&&
\label{eq:normMftf}
M(\tilde{f}_{\infty,\varepsilon}^{h}-f_{\infty,\varepsilon}^{h},\tilde{f}_{2,\varepsilon}^{h}-f_{2,\varepsilon}^{h},0)\\
\label{eq:normMg}&&
M(g_{\infty,\varepsilon}^{h},0,0)\,,
\end{eqnarray}
are bounded by $\hat{C}_{\alpha_{1}}\varepsilon$\,.\\
Because the time interval is restricted to
$I_{\hat{T}_{\alpha_{1}}}^{h}$ with
$\hat{T}_{\alpha_{1}}<T_{\alpha_{1}}$\,, the weight
$\sqrt{T_{\alpha_{1}}-|ht|}$ or
$\sqrt{T_{\alpha_{1}}-\tau}$ used in
Definition~\ref{de:E0Talpha} or in Proposition~\ref{pr:contracStri}
can be forgotten now (simply multiply $f_{q,\varepsilon}^{h}$, $\tilde{f}_{q,\varepsilon}^{h}$, $q\in\{\infty,2\}$ and $g_{\infty,\varepsilon}^{h}$ by $1_{I_{\frac{\hat{T}_{\alpha_{1}}}{h}}}(t)$).\\
The estimate of \eqref{eq:normuG} is obvious since
$$
\|(1-\chi_{\varepsilon}(N))u_{G}^{h}(t)\|_{L^{2}_{z}L^{2}_{y_{G}}}\leq
\underbrace{\sup_{s\geq
    0}|(1-\chi_{\varepsilon}(s))e^{-\alpha_{1}s}|}_{\mathcal{O}(\varepsilon)}\times
\underbrace{\|e^{\alpha_{1}N}u_{G}^{h}(t)\|_{L^{2}_{z}L^{2}_{y_{G}}}}_{\leq
\tilde{C}_{\alpha_{1}}}\,.
$$
The estimate of \eqref{eq:normMftf} is very similar. Actually in the
proof of Theorem~\ref{th:applStri} we checked 
$M(f_{\infty}^{h},f_{2}^{h},0)\lesssim
\|e^{\alpha_{1}N}u_{G,0}^{h}\|_{L^{2}_{z}L^{2}_{y_{G}}}$\,. It
gives now
$$
M(\tilde{f}_{\infty,\varepsilon}^{h}-f_{\infty,\varepsilon}^{h},\tilde{f}_{2,\varepsilon}^{h}-f_{2,\varepsilon}^{h},0)\lesssim
\|e^{\alpha_{1}N}(\chi_{\varepsilon}(N)-1)u_{G,0}^{h}\|_{L^{2}_{z}L^{2}_{y_{G}}}\leq \hat{C}_{\alpha_{1}}\varepsilon\,.
$$
For \eqref{eq:normMg} let us first decompose $g_{\infty,
  \varepsilon}^{h}$ as
\begin{eqnarray*}
  &&
     g_{\infty,\varepsilon}^{h}=g_{\infty,1,\varepsilon}^{h}+g_{\infty,2,\varepsilon}^{h}\\
\text{with}&&
g_{\infty,1,\varepsilon}^{h}=-i\sqrt{h}\chi_{\varepsilon}(N)\int_{0}^{\frac{t}{h}}U(t-s)\,a_{G}^{*}(V)\,(1-\chi_{\varepsilon}^{2}(N))\,u_{G}^{h}(s)~ds\,\\
\text{and}&&
g_{\infty,2,\varepsilon}^{h}=-i\sqrt{h}\chi_{\varepsilon}(N)\int_{0}^{\frac{t}{h}}U(t-s)\,a_{G}(V)\,(1-\chi_{\varepsilon}^{2}(N))\,u_{G}^{h}(s)~ds\,.
\end{eqnarray*}
The estimate of $g_{\infty,1,\varepsilon}^{h}$ follows the method for
the bound of $M(f_{\infty}^{h},0,0)$ in the proof of Theorem~\ref{th:applStri}, where we simply used the
uniform bound in time for
$\|U(s)e^{\alpha_{1}N}u_{G,0}^{h}\|_{L^{2}_{z}L^{2}_{y_{G}}}$\,. With 
$$
\sup_{t}\|(1-\chi_{\varepsilon}^{2}(N))u_{G}^{h}(t)\|_{L^{2}_{z}L^{2}_{y_{G}}}\leq 
\underbrace{\sup_{s\geq
    0}|(1-\chi^{2}_{\varepsilon}(s))e^{-\alpha_{1}s}|}_{\mathcal{O}(\varepsilon)}\times
\underbrace{\|e^{\alpha_{1}N}u_{G}^{h}(t)\|_{L^{2}_{z}L^{2}_{y_{G}}}}_{\leq
\tilde{C}_{\alpha_{1}}}\,,
$$
this gives
$$
M(g_{\infty,1,\varepsilon}^{h},0,0)\leq \hat{C}_{\alpha_{1}}\varepsilon\,.
$$
For $g_{\infty,2,\varepsilon}^{h}$\,, remember firstly that the
assumption  is
$\|e^{2\alpha_{1}N}u_{G,0}^{h}\|_{L^{2}_{z}L^{2}_{y_{G}}}\leq C_{\alpha_{1}}$ and by possibly
reducing $\hat{T}_{\alpha_{1}}$\,, we may assume $\|e^{\frac{3\alpha_{1}}{2}
  N}u_{G,h}(t)\|_{L^{2}_{z}L^{2}_{y_{G}}}\leq \tilde{C}_{\alpha_{1}}$\,. We now use the obvious relation
$a_{G}(V)\phi(N)=\phi(N+1)a_{G}(V)$ and write
$$
  g_{\infty,2,\varepsilon}^{h}=-i\chi_{\varepsilon}(N)e^{-\frac{\alpha_{1}}{2}(N+1)}(1-\chi_{\varepsilon}^{2}(N+1))e^{\frac{\alpha_{1}}{2}(N+1)}
\int_{0}^{\frac{t}{h}}U(t-s)\sqrt{h}a_{G}(V)u_{G}^{h}(s)~ds\,.
$$
Remember that the equivalent system \eqref{eq:dynaleauG} says
$\sqrt{h}a_G(V)u_{G}^{h}(t)=u_{1}^{h}(t)+\sqrt{h}u_{2}^{h}(t)$ with
$M(0,u_{2}^{h},u_{1}^{h})\lesssim C_{\alpha_{1}}$\,. The above equality
becomes
$$
g_{\infty,2,\varepsilon}^{h}(t)=
\chi_{\varepsilon}(N)(1-\chi_{\varepsilon}^{2}(N+1))e^{-\frac{\alpha_{1}}{2}(N+1)}e^{\frac{\alpha_{1}}{2}(N+1)}[L_{\infty\,1}(u_{1}^{h})+L_{\infty\,2}(u_{2}^{h})]\,.
$$
The bounds for $L_{\infty\,1}$ and $L_{\infty\,2}$ in the
Theorem~\ref{th:applStri}, 
lead to
$$
\||ht|^{-1/2}e^{\frac{\alpha_{1}}{2}(N+1)}[L_{\infty\,1,\varepsilon}(u_{1}^{h})+L_{\infty\,2,\varepsilon}(u_{2}^{h})](t)\|_{L^{\infty}(I_{\hat{T}_{\alpha_{1}}}^{h};L^{2}_{z}L^{2}_{y_{G}})}\lesssim C_{\alpha_{1}}\,.
$$
With 
$$
\|\chi_{\varepsilon}(N)(1-\chi_{\varepsilon}^{2}(N+1))e^{-\frac{\alpha_{1}}{2}(N+1)}\|_{\mathcal{L}(L^{2}_{z}L^{2}_{y_{G}})}\leq
\sup_{s\geq 0}|(1-\chi_{\varepsilon}^{2}(s))e^{-\frac{\alpha_{1}}{2}s}|=\mathcal{O}(\varepsilon)\,,
$$
this proves 
$$
M(g_{\infty,2,\varepsilon}^{h},0,0)\leq \hat{C}_{\alpha_{1}}\varepsilon\,.
$$
\end{proof}
Let us go back to our initial problem and let us compare the evolution
of states for the dynamics
$U(\frac{t}{h})=e^{-it(-\Delta_{x}+\sqrt{h}\mathcal{V})}$ for
$\varepsilon=0$ and the case $\varepsilon>0$ where
$\chi_{\varepsilon}(N)\mathcal{V}\chi_{\varepsilon}(N)$ is a bounded
self-adjoint perturbation of $-\Delta_{x}$\,. Set in particular
\begin{equation}
  \label{eq:UVUVeps}
  U_{\mathcal{V},\varepsilon}=e^{-it(-\Delta_{x}+\sqrt{h}\mathcal{V}_{\varepsilon})}\quad\text{with}\quad \mathcal{V}_{\varepsilon}=\chi_{\varepsilon}(N)\mathcal{V}\chi_{\varepsilon}(N)\,.
\end{equation}
\begin{proposition}
\label{pr:compmes}
  Assume like in Proposition~\ref{pr:equicont}
$$
V\in L^{r'_{\sigma}}(\rz^{d},dx;\rz)\cap H^{2}(\rz^{d};\rz)\quad,\quad
r'_{\sigma}=\frac{2d}{d+2}\quad,\quad d\geq 3\,,
$$
and assume that there exists $\alpha_{1}>0$ such that $\varrho_{h}(0)\in
\mathcal{L}^{1}(L^{2}(\rz^{d}\times \Omega ,dx\otimes
\mathcal{G};\cz))$\,, $\varrho_{h}(0)\geq 0$\,,
$\mathrm{Tr}[\varrho_{h}(0)]=1$ satisfies
$$
\exists C_{\alpha_{1}}>0\,,\, 
\forall h\in ]0,h_{0}[\,, \quad
\mathrm{Tr}\left[e^{\alpha_{1}N}\varrho_{h}(0)e^{\alpha_{1}N}\right]\leq C_{\alpha_{1}}\,.
$$
Call
$\varrho_{h}(t)=U_{\mathcal{V}}(\frac{t}{h})\varrho_{h}(0)U_{\mathcal{V}}^{*}(\frac{t}{h})$
and
$\varrho_{h,\varepsilon}(t)=U_{\mathcal{V},\varepsilon}(\frac{t}{h})\varrho_{h}(0)U_{\mathcal{V},\varepsilon}^{*}(\frac{t}{h})$\,.
When the subset $\mathcal{E}\subset ]0,h_{0}[$\,, $0\in \overline{\mathcal{E}}$\,,
is chosen such that
$$
\forall t\in ]-\hat{T}_{\alpha_{1}},\hat{T}_{\alpha_{1}}[\,,\quad
\mathcal{M}(\varrho_{h}(t),\, h\in
\mathcal{E})=\left\{\mu_{t}\right\}\quad\text{and}\quad 
\mathcal{M}(\varrho_{h,\varepsilon}(t),\,
h\in \mathcal{E})=\left\{\mu_{t,\varepsilon}\right\}
$$
Then  the total variation of $\mu_{t}-\mu_{t,\varepsilon}$ is
estimated by 
$$
\forall t\in ]-\hat{T}_{\alpha_{1}},\hat{T}_{\alpha_{1}} [\,,\quad
|\mu_{t}-\mu_{t,\varepsilon}|(\underbrace{T^{*}\rz^{d}}_{\text{or}~T^{*}\rz^{d}\sqcup
  \left\{\infty\right\}})\leq C'_{\alpha_{1}}\varepsilon\,,
$$
for some constant $C'_{\alpha_{1}}>0$ determined by $\alpha_{1}>0$\,.
\end{proposition}
\begin{proof}
From
\begin{align*}
\varrho_{h}(t)-\varrho_{h,\varepsilon}(t)=&
\left[U_{\mathcal{V}}(\frac{t}{h})\varrho_{h}(0)^{1/2}-U_{\mathcal{V},\varepsilon}(\frac{t}{h})\varrho_{h}(0)^{1/2}\right]\varrho_{h}(0)^{1/2}U_{\mathcal{V}}^{*}(\frac{t}{h})
\\
&
+U_{\mathcal{V},\varepsilon}(\frac{t}{h})\varrho_{h}(0)^{1/2}[\varrho_{h}(0)^{1/2}U_{\mathcal{V}}^{*}(\frac{t}{h})-\varrho_{h}(0)^{1/2}U_{\mathcal{V},\varepsilon}^{*}(\frac{t}{h})]
\end{align*}
we deduce
$$
|\mu(t)-\mu_{\varepsilon}(t)|(T^{*}\rz^{d}\cup \left\{\infty\right\})\leq
4\liminf_{h\in \mathcal{E}, h\to 0}
\|\varrho_{h}(t)-\varrho_{h,\varepsilon}(t)\|_{\mathcal{L}^{1}}
\leq 8\liminf_{h\in \mathcal{E}, h\to 0}\|\Psi^{h}(t)-\Psi^{h}_{\varepsilon}(t)\|_{L^{2}_{x,\omega,\hat{z}}}
$$
with
$\Psi^{h}_{\varepsilon}(t)=U_{\mathcal{V},\varepsilon}(\frac{t}{h})\varrho_{h}(0)^{1/2}\in
\mathcal{L}^{2}(L^{2}(\rz^{d}\times \Omega, dx\otimes
\mathcal{G};\cz))\sim
L^{2}(\rz^{d}\times \Omega\times \hat{Z}, dx\otimes
\mathcal{G}\otimes \mathbf{d\hat{z}};\cz)$ with $\hat{Z}=\rz^{d}\times \Omega$\,, $\mathbf{d\hat{z}}=dx\otimes
\mathcal{G}$\,.\\
But Proposition~\ref{pr:comp} implies
$$
\forall t\in ]-\hat{T}_{\alpha_{1}},\hat{T}_{\alpha_{1}}[\,,\quad 
\|\Psi^{h}(t)-\Psi^{h}_{\varepsilon}(t)\|_{L^{2}_{x,\omega,\hat{z}}}
\leq \hat{C}_{\alpha_{1}}\varepsilon\,.
$$
\end{proof}
\subsection{Asymptotic conservation of energy}
\label{sec:asconsEn}
The result of this paragraph is a consequence of the approximation
of the $U_{\mathcal{V}}$ dynamics by the one of
$U_{\mathcal{V}_{\varepsilon}}$ in terms of wave functions in Proposition~\ref{pr:comp}, states and
semiclassical measures in Proposition~\ref{pr:compmes}
\begin{proposition}
\label{pr:consen}
  Assume like in Proposition~\ref{pr:equicont}
$$
V\in L^{r'_{\sigma}}(\rz^{d},dx;\rz)\cap H^{2}(\rz^{d};\rz)\quad,\quad
r'_{\sigma}=\frac{2d}{d+2}\quad,\quad d\geq 3\,,
$$
and assume that there exists $\alpha_{1}>0$ such that $\varrho_{h}(0)\in
\mathcal{L}^{1}(L^{2}(\rz^{d}\times \Omega ,dx\otimes
\mathcal{G};\cz))$\,, $\varrho_{h}(0)\geq 0$\,,
$\mathrm{Tr}[\varrho_{h}(0)]=1$ satisfies
$$
\exists C_{\alpha_{1}}>0\,,\, 
\forall h\in ]0,h_{0}[\,, \quad
\mathrm{Tr}\left[e^{\alpha_{1}N}\varrho_{h}(0)e^{\alpha_{1}N}\right]\leq C_{\alpha_{1}}\,.
$$
Call
$\varrho_{h}(t)=U_{\mathcal{V}}(\frac{t}{h})\varrho_{h}(0)U_{\mathcal{V}}^{*}(\frac{t}{h})$
 and let the subset $\mathcal{E}\subset ]0,h_{0}[$\,, $0\in \overline{\mathcal{E}}$\,,
be  such that
$$
\forall t\in ]-\hat{T}_{\alpha_{1}},\hat{T}_{\alpha_{1}}[\,,\quad
\mathcal{M}(\varrho_{h}(t),\, h\in
\mathcal{E})=\left\{\mu_{t}\right\}
$$
with the additional assumption at time $t=0$\,,
\begin{equation}
  \label{eq:hypenmu0}
\mathrm{supp}\,\mu_{0}\subset \left\{(x,\xi)\in T^{*}\rz^{d}\,,
 \; |\xi|^{2}\in F\right\}
\end{equation}
where $F$ is a closed subset of $\rz$\,. Then for all $t\in
]-\hat{T}_{\alpha_{1}},\hat{T}_{\alpha_{1}}[$\,, the support of $\mu_{t}$ restricted to
$T^{*}\rz^{d}$ satisfies
$$
\mathrm{supp}\,\mu_{t}\big|_{T^{*}\rz^{d}}\subset \left\{(x,\xi)\in
  T^{*}\rz^{d}\,,\; |\xi|^{2}\in F\right\}\,.
$$
\end{proposition}
\begin{proof}
 For $\varepsilon>0$ and $z\in \cz\setminus\rz$ the resolvent estimate
$$
\|[z+\Delta_{x}]^{-1}-[z-(-\Delta_{x}+\sqrt{h}\mathcal{V}_{\varepsilon})]^{-1}\|_{\mathcal{L}(L^{2}_{x,\omega})}\leq
\frac{C_{\varepsilon}\sqrt{h}}{|\Imag z|^{2}}
$$
with
$\mathcal{V}_{\varepsilon}=\chi_{\varepsilon}(N)\mathcal{V}\chi_{\varepsilon}(N)\in
\mathcal{L}(L^{2}_{x,\omega})$ as in \eqref{eq:UVUVeps}
combined with Helffer-Sj{\"o}strand formula \cite{HeSj} gives
$$
\forall \varepsilon>0\,,\forall \chi\in
\mathcal{C}^{\infty}_{0}(\rz;\cz)\,,
\exists C_{\chi,\varepsilon}>0\,,\quad
\|\chi(-\Delta_{x})-\chi(-\Delta_{x}+\sqrt{h}\mathcal{V}_{\varepsilon})\|_{\mathcal{L}(L^{2}_{x,\omega})}\leq C_{\chi,\varepsilon}\sqrt{h}\,.
$$
The semiclassical calculus then implies
$$
\left\|\chi(-\Delta_{x}+\sqrt{h}\mathcal{V}_{\varepsilon}) \, a^{\Weyl}(hx,D_{x}) \, \chi(-\Delta_{x}+\sqrt{h}\mathcal{V}_{\varepsilon})-[\chi^{2}(|\xi|^{2})a]^{\Weyl}(hx,D_{x})
\right\|_{\mathcal{L}(L^{2}_{x,\omega})}=\mathcal{O}_{a,\chi,\varepsilon}(\sqrt{h})
$$
for all $a\in \mathcal{C}^{\infty}_{0}(T^{*}\rz^{d};\cz)$ and all
$\chi\in \mathcal{C}^{\infty}_{0}(\rz;\cz)$\,.\\
Hence, the assumption \eqref{eq:hypenmu0} implies
$$
\forall \chi\in \mathcal{C}^{\infty}_{0}(\rz\setminus F;[0,1])\,,\quad
\lim_{h\in \mathcal{E}\,, h\to
  0}\|\chi(-\Delta_{x}+\sqrt{h}\mathcal{V}_{\varepsilon}) \, \varrho_{h}(0) \, \chi(-\Delta_{x}+\sqrt{h}\mathcal{V}_{\varepsilon})\|_{\mathcal{L}^{1}(L^{2}_{x,\omega})}=0\,,
$$
and therefore 
$$
\forall \chi\in \mathcal{C}^{\infty}_{0}(\rz\setminus F;[0,1])\,,\,
\forall t\in ]-\hat{T}_{\alpha_{1}},\hat{T}_{\alpha_{1}}[\,,\quad 
\lim_{h\in \mathcal{E}\,, h\to 0}\|\chi(-\Delta_{x}+\sqrt{h}\mathcal{V}_{\varepsilon}) \, \varrho_{h,\varepsilon}(t) \, \chi(-\Delta_{x}+\sqrt{h}\mathcal{V}_{\varepsilon})\|_{\mathcal{L}^{1}(L^{2}_{x,\omega})}=0\,,
$$
with
$\varrho_{h,\varepsilon}(t)=U_{\mathcal{V}_{\varepsilon}}(\frac{t}{h})\varrho_{h}(0)U_{\mathcal{V}_{\varepsilon}}^{*}(\frac{t}{h})$
and
$U_{\mathcal{V}_{\varepsilon}}(t)=e^{-it(-\Delta_{x}+\sqrt{h}\mathcal{V}_{\varepsilon})}$\,.\\
When $\mathcal{E}'\subset \mathcal{E}$\,, $0\in \overline{\mathcal{E}'}$\,, is
such that
$$
\mathcal{M}(\varrho_{h,\varepsilon}(t)\,, h\in \mathcal{E}')=\left\{\mu_{t,\varepsilon}\right\}\,,
$$
Proposition~\ref{pr:compmes} tells us
$$
|\mu_{t}-\mu_{t,\varepsilon}|(T^{*}\rz^{d})\leq C'_{\alpha_{1}}\varepsilon\,.
$$
while 
\begin{multline*}
\int_{T^{*}\rz^{d}}a(x,\xi) \, |\chi|^{2}(|\xi|^{2})~d\mu_{t,\varepsilon}(x,\xi)\\
=\lim_{h\in
\mathcal{E}',h\to 0}\mathrm{Tr}\,\left[\chi(-\Delta_{x}+\sqrt{h}\mathcal{V}_{\varepsilon}) \, a^{\Weyl}(hx,D_{x}) \, \chi(-\Delta_{x}+\sqrt{h}\mathcal{V}_{\varepsilon}) \, \varrho_{h,\varepsilon}(t)\right]=0\,,
\end{multline*}
for $a\in \mathcal{C}^{\infty}_{0}(T^{*}\rz^{d};\cz)$ and $\chi\in
\mathcal{C}^{\infty}_{0}(\rz\setminus F;[0,1])$\,.
We deduce
$$
\forall a\in \mathcal{C}^{\infty}_{0}(T^{*}\rz^{d};\cz)\,, \forall
\chi\in \mathcal{C}^{\infty}_{0}(\rz\setminus F;[0,1])\,, \forall t\in
]-\hat{T}_{\alpha_{1}},\hat{T}_{\alpha_{1}}[\,,\quad
\int_{T^{*}\rz^{d}}  a(x,\xi) \, \chi^{2}(|\xi|^{2})~d\mu_{t}(x,\xi)=0\,,
$$
which yields the result.
\end{proof}
\subsection{Changing $V$}
\label{sec:changV}
The formulation of Theorem~\ref{th:applStri}
$u_{G}^{h}(t)=U_{\mathcal{V}}(t)u_{G,0}^{h}=U(\frac{t}{h})u_{G,0}^{h}+u_{\infty}^{h}(t)$
where $\big(u^h_q\big)_{q\in{\{\infty,2,1\} }}$
is a solution of a fixed point problem, solved in Proposition~\ref{pr:contracStri}, where only
$\|V\|_{L^{r'_{\sigma}}}$\,, $r'_{\sigma}=\frac{2d}{d+2}$\,, is used,
allows to consider perturbations of $V$\,, which can be done
separately in the the terms $a_{G}(V)$ and $a_{G}^{*}(V)$ and with
complex valued perturbations.\\
Remember that our state
$\varrho_{h}(t)=U_{\mathcal{V}}(\frac{t}{h})\varrho_{h}(0)U_{\mathcal{V}}^{*}(\frac{t}{h})$
is written 
$$
\varrho_{h}(t)=[U_{\mathcal{V}}(\frac{t}{h})\varrho_{h}(0)^{1/2}][\varrho_{h}(0)^{1/2}U^*_{\mathcal{V}}(\frac{t}{h})]\,,
$$
and the link with the fixed point problem is done after setting
$$
U(t)u_{G,0}^{h}+u^{h}_{\infty}(t)=u_{G}^{h}(t)=U_{\mathcal{V}}(\frac{t}{h})\varrho_{h}(0)^{1/2}\quad\text{in}~\mathcal{L}^{2}(L^{2}_{x,\omega})\sim
L^{2}_{z,y_{G}}\,,
$$
where the last identification is done via the unitary transform
$U_{G}$ of Section~\ref{sec:centermass}\,, omitted here and explained in the proof of
Proposition~\ref{pr:equicont}.\\
A generalization is done by writing for a pair $\tilde{\mathcal{V}}=(V_{1},V_{2})\in
L^{r'_{\sigma}}(\rz^{d},dy;\cz)^{2}$\,,
\begin{equation}
  \label{eq:generCV}
  \varrho_{h,\tilde{\mathcal{V}}}(t)=u_{G,\tilde{\mathcal{V}}}^{h}(\frac{t}{h})[u_{G,\tilde{\mathcal{V}}}^{h}(\frac{t}{h})]^{*}\in \mathcal{L}^{1}(L^{2}_{x,\omega})\,,
\end{equation}
where
$u_{G,\tilde{\mathcal{V}}}^{h}(t)=U(t)\varrho_{h}(0)^{1/2}+u_{\infty,\tilde{\mathcal{V}}}^{h}(t)$
and  $\big(u^h_{q,\tilde{\mathcal{V}}}\big)_{q\in{\{\infty,2,1\} }}$
solves the fixed point problem
\eqref{eq:dynuinfty}\eqref{eq:dynu2}\eqref{eq:dynu1} with 
$f_{1}^{h}(t)=0$ and $f_{\infty}^{h}$ and $f_{2}^{h}$ given by 
\begin{align}
\label{eq:rhsfh2}
f^{h}_\infty(t) =f^{h}_{\infty,\tilde{\mathcal{V}}}(t)=-i&\int_{0}^{t}U(t)U(s)^{*}
a_{G}^{*}(V_{1}) \sqrt{h}U(s)u_{G,0}^h~ds \,,\\
\label{eq:rhsgh2}
f^{h}_2(t)=f_{2,\tilde{\mathcal{V}}}(t)=-i\,a_{G}(V_{2})&\int_{0}^{t}U(t)U(s)^{*}a_{G}^{*}(V_{1}) \sqrt{h}U(s)u_{G,0}^h~ds + a_G(V_{2}) U(t) u^h_{G,0}\,.
\end{align}
This fixed point problem will be written
\begin{equation}
  \label{eq:LtV}
  \begin{pmatrix}
    u^{h}_{\infty,\tilde{\mathcal{V}}}\\
u^{h}_{2,\tilde{\mathcal{V}}}\\
u^{h}_{1,\tilde{\mathcal{V}}}
  \end{pmatrix}
=L_{\tilde{\mathcal{V}}}\begin{pmatrix}
    u^{h}_{\infty,\tilde{\mathcal{V}}}\\
u^{h}_{2,\tilde{\mathcal{V}}}\\
u^{h}_{1,\tilde{\mathcal{V}}}
  \end{pmatrix}+
  \begin{pmatrix}
    f_{\infty,\tilde{\mathcal{V}}}^{h}\\
f_{2,\tilde{\mathcal{V}}}^{h}\\
0
  \end{pmatrix}\,.
\end{equation}
\begin{proposition}
\label{pr:comptV} For two pairs
$\tilde{\mathcal{V}}_{k}=(V_{1,k},V_{2,k})\in
L^{r'_{\sigma}}(\rz^{d},dy;\cz)^{2}$\,, for
$\|e^{\alpha_{1}N}u_{G,0}^{h}\|\leq C_{\alpha_{1}}$ and by choosing
$\hat{T}_{\alpha_{1}}>0$ small enough, the two solutions to
\eqref{eq:LtV} with the right-hand sides given by
\eqref{eq:rhsfh2}\eqref{eq:rhsgh2} satisfy
$$
\forall t\in ]-\hat{T}_{\alpha_{1}},\hat{T}_{\alpha_{1}}[
\,,\quad
\|u_{\infty,\tilde{\mathcal{V}}_{2}}^{h}(\frac{t}{h})-u_{\infty,\tilde{\mathcal{V}}_{1}}^{h}(\frac{t}{h})\|_{L^{2}_{z,y_{G}}}\leq
C\left[\|V_{1,2}-V_{1,1}\|_{L^{r'_{\sigma}}}+\|V_{2,2}-V_{2,1}\|_{L^{r'_{\sigma}}}\right]
$$
for some constant $C>0$ given by $\alpha_{1}>0$\,, $C_{\alpha_{1}}$\,, the dimension
$d$\,, and $\max_{i,j}\|V_{i,j}\|_{L^{r'_{\sigma}}}$\,.
\end{proposition}
\begin{proof}
  It suffices to notice that  the difference
  $v^{h}=u_{\tilde{\mathcal{V}}_{2}}^{h}-u_{\tilde{\mathcal{V}}_{1}}^{h}$
  with $u_{\tilde{\mathcal{V}_{k}}}^{h}=\big(u^h_{q,\tilde{\mathcal{V}_k}}\big)_{q\in{\{\infty,2,1\} }}
$\,, $k=1,2$\,, solves
$$
v^{h}-L_{\tilde{\mathcal{V}_{1}}}(v^{h})=(L_{\tilde{\mathcal{V}}_{2}}-L_{\tilde{\mathcal{V}}_{1}})(u_{\tilde{\mathcal{V}}_{2}}^{h})+
\begin{pmatrix}
  f_{\infty,\tilde{\mathcal{V}}_{2}}^{h}-f_{\infty,\tilde{\mathcal{V}}_{1}}^{h}\\
f_{2,\tilde{\mathcal{V}}_{2}}^{h}-f_{2,\tilde{\mathcal{V}_{1}}}^{h}\\
0
\end{pmatrix}\,.
$$
Estimates for all the terms of the right-hand side have essentially
been proved for Proposition~\ref{pr:contracStri} and for
Theorem~\ref{th:applStri}. Although they are written for $V_{1}=V_{2}$
real-valued in Theorem~\ref{th:applStri} the generalization is
straightforward (like in Proposition~\ref{pr:contracStri}) and upper
bounds are proportional  the $L^{r'_{\sigma}}$ of the potential which
is either $(V_{1,2}-V_{1,1})$ or $(V_{2,2}-V_{2,1})$\,.\\
The time interval
$]-T_{\alpha_{1}},T_{\alpha_{1}}[=]-2\hat{T}_{\alpha_{1}},2\hat{T}_{\alpha_{1}}[$
is actually chosen like in Proposition~\ref{pr:contracStri} such that
$\|L_{\tilde{\mathcal{V}}_{1}}\|_{\mathcal{L}(\mathcal{E}_{\alpha_{1},-\alpha_{1}},\gamma)}\leq
\frac{1}{2}$ and this ends the proof.
\end{proof}
For a general pair $\tilde{\mathcal{V}}=(V_{1},V_{2})\in
L^{r'_{\sigma}}(\rz^{d},dy;\cz)^{2}$\,, the trace-class operator
$\varrho_{\tilde{\mathcal{V}}}^{h}(t)$ is no more a state and neither
self-adjoint. However it remains uniformly bounded in
$\mathcal{L}^{1}(L^{2}_{x,\omega})$ and complex-valued semiclassical
measures $\mu_{\tilde{\mathcal{V}}}(t)$ make sense for $t\in
]-\hat{T}_{\alpha_{1}},\hat{T}_{\alpha_{1}}[$\,. Moreover the results
of Proposition~\ref{pr:applStri} and Proposition~\ref{pr:equicont} can
be adapted mutatis mutandis for such a general pair, so that
semiclassical measures (extraction process) can be defined
simultaneously for all $t\in
]\hat{T}_{\alpha_{1}},\hat{T}_{\alpha_{1}}[$\,.\\
The above comparison result can be translated in terms of trace-class
operators and asymptotically for semiclassical measures.
\begin{proposition}
\label{pr:tVsemi}
  Assume
$$
V\in L^{r'_{\sigma}}(\rz^{d},dx;\rz)\cap H^{2}(\rz^{d};\rz)\quad,\quad
V_{1},V_{2}\in L^{r'_{\sigma}}(\rz^{d},dx;\cz)\quad,\quad
r'_{\sigma}=\frac{2d}{d+2}\quad,\quad d\geq 3\,,
$$
and assume that there exists $\alpha_{1}>0$ such that $\varrho_{h}(0)\in
\mathcal{L}^{1}(L^{2}(\rz^{d}\times \Omega ,dx\otimes
\mathcal{G};\cz))$\,, $\varrho_{h}(0)\geq 0$\,,
$\mathrm{Tr}[\varrho_{h}(0)]=1$ satisfies
$$
\exists C_{\alpha_{1}}>0\,,\, 
\forall h\in ]0,h_{0}[\,, \quad
\mathrm{Tr}\left[e^{\alpha_{1}N}\varrho_{h}(0)e^{\alpha_{1}N}\right]\leq C_{\alpha_{1}}\,.
$$
Let
$\varrho_{h}(t)=U_{\mathcal{V}}(\frac{t}{h})\varrho(0)U_{\mathcal{V}}^{*}(\frac{t}{h})$
and let $\varrho_{h,\tilde{\mathcal{V}}}(t)$ be defined by
\eqref{eq:generCV}. Then
$$
\exists C>0\,,\,
\forall t\in ]-\hat{T}_{\alpha_{1}},\hat{T}_{\alpha_{1}}[\,,\quad
\|\varrho_{h}(t)-\varrho_{h,\tilde{\mathcal{V}}}(t)\|_{\mathcal{L}^{1}(L^{2}_{x,\omega})}\leq C\left[\|V_{1}-V\|_{L^{r'_{\sigma}}}+\|V_{2}-V\|_{L^{r'_{\sigma}}}\right]\,.
$$
When the subset $\mathcal{E}\subset ]0,h_{0}[$\,, $0\in \overline{\mathcal{E}}$\,,
is chosen such that
$$
\forall t\in ]-\hat{T}_{\alpha_{1}},\hat{T}_{\alpha_{1}}[\,,\quad
\mathcal{M}(\varrho_{h}(t),\, h\in
\mathcal{E})=\left\{\mu_{t}\right\}\quad\text{and}\quad 
\mathcal{M}(\varrho_{h,\tilde{\mathcal{V}}}(t),\,
h\in \mathcal{E})=\big\{\mu_{t,\tilde{\mathcal{V}}}\big\}
$$
Then  the total variation of $\mu_{t}-\mu_{t,\tilde{\mathcal{V}}}$ is
estimated by 
$$
\exists C'>0\,,\,
\forall t\in ]-\hat{T}_{\alpha_{1}},\hat{T}_{\alpha_{1}} [\,,\quad
|\mu_{t}-\mu_{t,\tilde{\mathcal{V}}}|(\underbrace{T^{*}\rz^{d}}_{\text{or}~T^{*}\rz^{d}\sqcup
  \left\{\infty\right\}})\leq C'\left[\|V_{1}-V\|_{L^{r'_{\sigma}}}+\|V_{2}-V\|_{L^{r'_{\sigma}}}\right]\,.
$$
\end{proposition}
\begin{proof}
  It suffices to write
$$
\varrho_{h,\tilde{\mathcal{V}}}(t)-\varrho_{h}(t)=\left[u_{G,\tilde{\mathcal{V}}}^{h}(\frac{t}{h})-u_{G}^{h}(\frac{t}{h})\right][u_{G,\tilde{\mathcal{V}}}^{h}(\frac{t}{h})]^{*}+[u_{G}^{h}(\frac{t}{h})]\left[u_{G,\tilde{\mathcal{V}}}^{h}(\frac{t}{h})-u_{G}^{h}(\frac{t}{h})\right]^{*}
$$
and to remember that Hilbert-Schmidt norms correspond to
$L^{2}_{z,y_{G}}$-norms estimated in Proposition~\ref{pr:comptV}.
\end{proof}

\subsection{Quantum dynamics with low regularity}
\label{sec:quantlow}
We conclude with an easy application of Proposition~\ref{pr:comptV}
which says that the dynamics $(U_{\mathcal{V}}(t))_{t\in\rz}$ is
actually well defined under the sole assumption
\begin{equation}
  \label{eq:VLrsig}
  V\in L^{r^{\prime}_{\sigma}}(\rz^{d};\rz)\quad,\quad
  r^{\prime}_{\sigma}=\frac{2d}{d+2}\quad d\geq 3\,,
\end{equation}
with good approximations when $V_{n}\in
L^{r^{\prime}_{\sigma}}(\rz^{d};\rz)\cap H^{2}(\rz^{d};\rz)$ satisfies
$\lim_{n\to \infty}\|V_{n}-V\|_{L^{r^{\prime}_{\sigma}}}=0$\,. 
\begin{proposition}
Let~$V$ belong to  $L^{r_{\sigma}^{\prime}}(\rz^{d};\rz)$ and let 
$(V_{n})_{n\in\nz}$ be a sequence in
$L^{r_{\sigma}^{\prime}}(\rz^{d};\rz)\cap H^{2}(\rz^{d};\rz)$ such
that~$\lim_{n\to\infty}\|V-V_{n}\|_{L^{r^{\prime}_{\sigma}}}=0$\,. 
Then for any $t\in\rz$ the unitary operator $U_{\mathcal{V}_{n}}(t)$
converges strongly to a unitary operator $U_{\mathcal{V}}(t)$\,.\\
Therefore $(U_{\mathcal{V}}(t))_{t\in\rz}$ is a strongly continous
unitary group in $L^{2}(\rz^{d}\times
Z'',dx\otimes\mathbf{dz''};\Gamma(L^{2}(\rz^{d},dy;\cz)))=L^{2}_{z,\rmsym}L^{2}_{y_{G}}$
with a self-adjoint generator denoted
$(-\Delta_{x}+\sqrt{h}\mathcal{V},D(-\Delta_{x}+\sqrt{h}\mathcal{V}))$\,.\\
The convergence
$(-\Delta_{x}+\sqrt{h}\mathcal{V}_{n},D(-\Delta_{x}+\sqrt{h}\mathcal{V}_{n}))$
to
$(-\Delta_{x}+\sqrt{h}\mathcal{V},D(-\Delta_{x}+\sqrt{h}\mathcal{V}))$
holds in the strong resolvent sense.
\end{proposition}

\begin{remark}
Although the dynamics $(U_{\mathcal{V}}(t))_{t\in\rz}$  and its
self-adjoint generator
$(-\Delta_{x}+\sqrt{h}\mathcal{V},D(-\Delta_{x}+\sqrt{h}\mathcal{V}))$
is well defined for $V\in L^{r^{\prime}_{\sigma}}(\rz^{d};\rz)$\,, 
we have no information on the domain
$D(-\Delta_{x}+\sqrt{h}\mathcal{V})$\,. The approximation process by
$V_{n}\in L^{r^{\prime}_{\sigma}}(\rz^{d};\rz)\cap H^{2}(\rz^{d};\rz)$
for which a core of $\Delta_{x}+\sqrt{h}\mathcal{V}_{n}$ is given by
Proposition~4.4 in \cite{Bre} recalled in Lemma~\ref{le:dyn}, provides
a substitute for the analysis. 

It could be interesting to see if this
Schr{\"o}dinger type approach relying on endpoint Strichartz estimates
could be applied to other quantum field theoretic problem and whether
it would bring additional information of tools as compared with the
euclidean approach (see \cite{Sim} and refences therein).
\end{remark}

\begin{proof}
Actually we can work here with $h=1$\,.
The convergence of 
$$
U_{\mathcal{V}_{n}}(t)u_{G,0}=U(t)u_{G,0}+u_{\infty,\mathcal{V}_{n}}(t)
$$
is deduced from the convergence  (see Proposition~\ref{pr:comptV}) of $u_{\infty,\mathcal{V}_{n}}(t)$ to
$u_{\infty,\mathcal{V}}(t)$ when 
$e^{\alpha_{1}N}u_{G,0}\in L^{2}_{z,\text{sym}}L^{2}_{y_{G}}$ for some
$\alpha_{1}>0$\,.\\
From
$\|U_{\mathcal{V}_{n}}(t)u_{G,0}
\|_{L^{2}_{z,\text{sym}}L^{2}_{y_{G}}}=\|u_{G,0}\|_{L^{2}_{z,\text{sym}}L^{2}_{y_{G}}}$
we deduce
$\|U_{\mathcal{V}}(t)u_{G,0}\|_{L^{2}_{z,\text{sym}}L^{2}_{y_{G}}}=\|u_{G,0}\|_{L^{2}_{z,\text{sym}}L^{2}_{y_{G}}}$\,. This
finally provides the extension of $U_{\mathcal{V}}(t)u_{G,0}$ for any
$u_{G,0}\in L^{2}_{z,\text{sym}}L^{2}_{y_{G}}$ with the convergence of
$U_{\mathcal{V}_{n}}(t)u_{G,0}$ to $U_{\mathcal{V}}(t)u_{G,0}$\,, because $e^{-\alpha_{1}N}L^{2}_{z,\text{sym}}L^{2}_{y_{G}}$ is
dense in $L^{2}_{z,\text{sym}}L^{2}_{y_{G}}$\,.
Passing from the strong convergence of unitary groups to the strong
resolvent convergence of generators is standard.
\end{proof}

\section*{Acknowledgement}
The authors thank Zied Ammari for beneficial first discussions about this project. The research of S.B.~was partly done during a CNRS sabbatical semester.

\pagebreak[4]
\bibliographystyle{plain}

\end{document}